\documentclass[12pt,a4paper]{amsart}

\title[Maps between relatively hyperbolic spaces and boundaries]{Maps between relatively hyperbolic spaces and between their boundaries}
\author{John M. Mackay}
\address{School of Mathematics \\ University of Bristol \\ Bristol, UK}
\email{john.mackay@bristol.ac.uk}
\author{Alessandro Sisto}
\address{Department of Mathematics, Heriot-Watt University, Edinburgh, UK}
	\email{a.sisto@hw.ac.uk}
\date{\today}
\keywords{Relatively hyperbolic groups and spaces, boundary at infinity, quasisymmetric map}
\subjclass[2010]{20F65, 20F67, 20F69, 31L10, 51F99}

\usepackage{amsmath,amsthm,amsfonts,graphicx,amssymb}
\usepackage{color}
\usepackage{hyperref}
\usepackage{stackengine}

\numberwithin{equation}{section}
\newtheorem{theorem}[equation]{Theorem}
\newtheorem{proposition}[equation]{Proposition}
\newtheorem{corollary}[equation]{Corollary}
\newtheorem{lemma}[equation]{Lemma}

\newtheorem{example}[equation]{Example}
\newtheorem{question}[equation]{Question}

\newtheorem{definition}[equation]{Definition}

\theoremstyle{definition}
\newtheorem{construction}[equation]{Construction}
\newtheorem{remark}[equation]{Remark}
\newtheorem{conv}[equation]{Convention}

\newtheoremstyle{citing}
  {3pt}
  {3pt}
  {\itshape}
  {}
  {\bfseries}
  {}
  {.5em}
  {\thmnote{#3}}

\theoremstyle{citing}
\newtheorem*{varthm}{}

\DeclareMathOperator{\id}{id}
\DeclareMathOperator{\diam}{diam}

\DeclareMathOperator{\length}{length}

\newcommand{\trans}{\mathit{trans}}
\newcommand{\deep}{\mathit{deep}}

\newcommand{\alp}{\alpha}

\newcommand{\del}{\delta}
\newcommand{\gam}{\gamma}
\newcommand{\eps}{\epsilon}

\newcommand{\bdry}{\partial_\infty}

\newcommand{\cH}{\mathcal{H}}

\newcommand{\cB}{\mathcal{B}}

\newcommand{\cP}{\mathcal{P}}

\newcommand{\cQ}{\mathcal{Q}}

\newcommand{\ra}{\rightarrow}
\newcommand{\R}{\mathbb{R}}
\newcommand{\C}{\mathbb{C}}
\newcommand{\Sph}{\mathbb{S}}
\newcommand{\N}{\mathbb{N}}
\newcommand{\Z}{\mathbb{Z}}
\newcommand{\HH}{\mathbb{H}}

\newcommand{\Bow}{\mathrm{Cusp}}
\newcommand{\qi}{\approxeq}
\newcommand{\grqi}{\succapprox}
\newcommand{\leqi}{\precapprox}

\def\XXint#1#2#3{{\setbox0=\hbox{$#1{#2#3}{\int}$}
\vcenter{\hbox{$#2#3$}}\kern-.5\wd0}}

\providecommand{\phantomsection}{}
\AtBeginDocument{\let\textlabel\label}
\makeatletter
\newcommand{\mylabel}[2]{\raisebox{.7\normalbaselineskip}{\phantomsection}#1%
  \def\@currentlabel{#1}\textlabel{#2}}
\makeatother
\numberwithin{equation}{section}

\begin{document}

\begin{abstract}
	We study relations between maps between relatively hyperbolic groups/spaces and quasisymmetric embeddings between their boundaries. More specifically, we establish a correspondence between (not necessarily coarsely surjective) quasi-isometric embeddings between relatively hyperbolic groups/spaces that coarsely respect peripherals, and quasisymmetric embeddings between their boundaries satisfying suitable conditions. Further, we establish a similar correspondence regarding maps with at most polynomial distortion.

	We use this to characterise groups which are hyperbolic relative to some collection of virtually nilpotent subgroups as exactly those groups which admit an embedding into a truncated real hyperbolic space with at most polynomial distortion, generalising a result of Bonk and Schramm for hyperbolic groups.
\end{abstract}

\maketitle
\setcounter{tocdepth}{1}
\tableofcontents

\section{Introduction}\label{sec:intro}

The boundary at infinity of a Gromov hyperbolic group is a crucial invariant for studying the properties of the group.
While its topology is already important, one can say more by looking at its metric properties: a result of Paulin~\cite{Pau-96-qm-qi} shows that two hyperbolic groups $G$ and $H$ are quasi-isometric if and only if their boundaries are quasisymmetric.  This additional quasisymmetric structure on the boundary allows one to distinguish quasi-isometry types of hyperbolic groups which have the same topological boundary, amongst other applications. Bonk and Schramm~\cite{BS-00-gro-hyp-embed} further established a correspondence of quasi-isometric embeddings between (certain) Gromov hyperbolic spaces to quasisymmetric embeddings between their boundaries, with refinements for classes of maps as we shall discuss below.

Our goal is to establish a similar correspondence for relatively hyperbolic groups (Definition~\ref{def:rel-hyp-grp}).  The motivating example of a relatively hyperbolic group $G$ is the fundamental group of a finite volume non-compact hyperbolic manifold; $G$ is a non-uniform lattice in $\mathrm{Isom}(\HH^{n})$.  If $n \geq 3$ then $G$ is not Gromov hyperbolic, as it contains $\Z^{n-1}$ parabolic subgroups $\{H_i\}$ due to the cusps, but one still can say the pair \emph{$(G,\{H_i\})$ is relatively hyperbolic (with peripheral subgroups $\{H_i\}$)}.  That is, roughly, gluing in `horoballs' to the left cosets of the $H_i$ one obtains a \emph{cusped space $\Bow(G,\{H_i\})$} which is Gromov hyperbolic (see Definition~\ref{def:rel-hyp}), and so has a boundary $\bdry (G,\{H_i\})$, which we call its \emph{Bowditch boundary}.  This boundary was defined by Bowditch~\cite{Bow-99-rel-hyp} (see also~\cite{Gro-Man-08-dehn-rel-hyp}) who showed that it depends only on $(G,\{H_i\})$ up to homeomorphism, not on the other choices involved in its construction.  In this example $\bdry (G, \{H_i\})$ is homeomorphic to $\Sph^{n-1}$.

Out of the many relevant works on relatively hyperbolic groups, we point out that the topology of the Bowditch boundary has been well-studied, in particular in connection to splittings, see e.g.~\cite{Bow-MathZ, Bow-99-conn-lim-set}.  We also highlight that maps between Bowditch boundaries have been previously studied by Groff~\cite{Groff-13-QI-JSJ}.
Also, Gerasimov--Potyagailo~\cite{Ger-Pot-13-JEMS-qi-floyd} study metric properties of the related Floyd boundaries, and in previous work we established metric properties of Bowditch boundaries~\cite{MS-12-relhypplane}.

Finally, as this paper was being finalised Healy and Hruska released a preprint with an extensive study of the properties of boundary maps for relatively hyperbolic groups \cite{healy2020cusped}. In Remark~\ref{rmk:healy-hruska}, we elaborate upon how the results in \cite{healy2020cusped} and in this paper complement each other.

Since a given group can have different relatively hyperbolic structures (with different
Bowditch boundaries), to have a sensible theory we have to restrict our attention
to maps which respect the peripheral structures, meaning that images and preimages of cosets of peripheral subgroups are close to cosets of peripheral subgroups.
For a precise statement see Definition~\ref{def:coarse-respect-periph} below.
The following example illustrates why we need this notion.

\begin{example}
	Let $G=\langle x,y \rangle$ be the fundamental group of a punctured torus, which is
	hyperbolic relative to the parabolic subgroup $H = \langle [x,y]\rangle$.
	Then $\bdry (G, \emptyset)$ is a Cantor set, while $\bdry (G, \{H\})$ is a circle.
	So the identity map from $(G,\emptyset) \ra (G,\{H\})$ cannot induce
	a homeomorphism of the boundaries,
	while the identity map from $(G, \{H\}) \ra (G, \emptyset)$ cannot induce anything
	sensible on the boundary at all.
\end{example}

The peripheral structures are often automatically respected.
When $(G, \{H_i\})$ and $(G', \{H_i'\})$ are relatively hyperbolic groups 
and no peripheral group is itself properly relatively hyperbolic, 
any quasi-isometry from $G$ to $G'$ will automatically respect the peripheral structures
(see Drutu~\cite[Theorem 5.7]{Dr-relhyp-qiinv} and also Groff~\cite[Theorem 6.3]{Groff-13-QI-JSJ}).
In this context, Groff shows that the two Bowditch boundaries are homeomorphic~\cite[Corollary 6.5]{Groff-13-QI-JSJ}.

A simplified statement of our first result (Theorem~\ref{thm:inside-to-boundary}) is the following.
A function $f:G\ra G'$ \emph{coarsely respects peripherals} if, roughly, images and preimages of peripheral sets are in bounded neighbourhoods of peripheral sets, see Definition~\ref{def:coarse-respect-periph}.
\begin{corollary}\label{thm:inside-to-boundary-simple}
	Suppose that $(G, \{H_i\})$ and $(G', \{H_i'\})$ are relatively hyperbolic groups,
	and that $f:G \ra G'$ is a quasi-isometric embedding that coarsely respects peripherals.
	Then $f$ extends to a quasi-isometric embedding $f_\Bow:\Bow(G, \{H_i\}) \ra \Bow(G', \{H_i'\})$.
	This map induces a quasisymmetric embedding $\bdry f_\Bow: \bdry(G, \{H_i\}) \ra \bdry(G', \{H_i'\})$.
\end{corollary}
Recall that $\Bow(G,\{H_i\})$ denotes the cusped space where one takes the Cayley
graph of $G$ and glues horoballs to the left cosets of groups in $\{H_i\}$ (see Definition~\ref{def:rel-hyp} and \cite{Gro-Man-08-dehn-rel-hyp}).

We note that if $f$ is a quasi-isometry then a large part of
the proof of this result can be replaced by a much shorter argument; this is done in
the proof of \cite[Theorem 6.3]{Groff-13-QI-JSJ}, which contains a proof of the theorem
in this case. In both cases one first proves that the extension $f_\Bow$
is coarsely Lipschitz, but then the difference is that, when $f$ is a
quasi-isometry, using a quasi-inverse of $f$ one can see that $f_\Bow$
has a quasi-inverse. When $f$ is not coarsely surjective, a more
sophisticated argument is needed, and more specifically we will rely on
detailed understanding of geodesics in relatively hyperbolic spaces,
which is not needed for the quasi-isometry case. 

The quasisymmetric embedding $\bdry f$ is not arbitrary; for example, it has to send
parabolic points in $\bdry(G, \{H_i\})$ to parabolic points in $\bdry(G', \{H_i'\})$. We find conditions on $\bdry f$ which are sufficient to give a converse statement;
i.e., any suitable quasisymmetric embedding $h:\bdry(G, \{H_i\}) \ra \bdry(G', \{H_i'\})$ will
induce a quasi-isometric embedding $\hat{h}:G \ra G'$ so that $\bdry \hat{h} = h$
(see Theorem~\ref{thm:boundary-to-inside}),
but for now we state a simpler result.
\begin{corollary}\label{thm:boundary-to-inside-simple}
	Suppose that $(G, \{H_i\})$ and $(G', \{H_i'\})$ are relatively hyperbolic groups with infinite proper peripheral groups,
	and that $h:\bdry(G,\{H_i\}) \ra \bdry(G',\{H_i'\})$ is a
	shadow-respecting quasisymmetry.
	Then there exists a quasi-isometry $\hat{h}:G \ra G'$
	so that $\hat{h}$ extends to give $\bdry (\hat{h})_\Bow = h$.
\end{corollary}
The `shadow' of a horoball in the boundary is, roughly, the ball in the
boundary which consists of the end points of geodesic rays from 
$e \in G$ which meet the horoball.
Given this, `shadow-respecting' means that parabolic points are mapped to
parabolic points, with the corresponding shadows approximately
preserved, and in addition that near parabolic points the map $\hat{h}$ looks like a `snowflake' map; see Definition~\ref{def:shadow-qs}.

In the case of hyperbolic groups $G, G'$, a quasisymmetric embedding $\bdry G \ra \bdry G'$ induces a quasi-isometric embedding $G \ra G'$ by work of Bonk and Schramm~\cite{BS-00-gro-hyp-embed} (see also Paulin~\cite{Pau-96-qm-qi}).
However, when we weaken $h$ in Corollary~\ref{thm:boundary-to-inside-simple} to be a shadow-respecting quasisymmetric embedding, we will find an induced map $\hat{h}:G \ra G'$ that is \emph{polynomially distorted}:
there exists $C>0, \alpha \in (0,1]$
so that for all $x,y \in G$,
	\[
		\tfrac1C d_G(x,y)^\alp -C \leq d_{G'}(\hat{h}(x),\hat{h}(y)) \leq C d_G(x,y)+C.
	\]
This is reasonable to expect:
Bonk and Schramm~\cite{BS-00-gro-hyp-embed} give us an extension of $h$ to
a quasi-isometric embedding between the cusped spaces.  
The shadow-respecting condition gives that this maps
horoballs to horoballs.  However, for pairs of points in the same left coset
of a peripheral group their distance in the group is essentially the
exponential of their distance in the cusped space.
Thus the multiplicative distortion of the quasi-isometry is
exponentiated to a polynomial distortion in the group.

As we will see in Theorem \ref{thm:boundary-to-inside}, one can correct for this by requiring finer control around parabolic
points.
We also strengthen Corollary~\ref{thm:inside-to-boundary-simple} to 
apply to certain {polynomially distorting maps} (Definition \ref{def:poly-subexp-distortion}), giving a correspondence
between maps on the group and maps on the boundary. Polynomially distorting maps have appeared in work of many authors, for example in the study of embeddings into Hilbert spaces, where they are used to define compression exponents \cite{compr_exp}.

Polynomially distorting maps, while not necessarily quasi-isometric embeddings, do preserve
some geometric properties of a group, and in fact the results below provide examples of this.

If $G$ has a polynomially (or indeed subexponentially) distorted embedding 
$f$ into a hyperbolic group $H$,
then $f$ is in fact a quasi-isometric embedding, and $G$ is hyperbolic.
This follows from easily from standard facts about hyperbolic groups; see
also Proposition~\ref{prop:subexp-into-hyp-is-qi}.
In this paper we show an analogous result for relatively hyperbolic groups and spaces (Theorem \ref{sederh}),
generalising known results about quasi-isometric embeddings. Here we state the version for groups.
\begin{corollary}
	\label{cor:subexpdistort-subgroup-rel-hyp}
	Let $G$ be hyperbolic relative to $\{H_i\}$.
	If $H$ is a subexponentially distorted subgroup of $G$ then $H$ is relatively quasiconvex in $G$, and so $H$ is hyperbolic relative to subgroups $\{K_j\}$ each of which is the intersection of a conjugate (in $G$) of some $H_i$ with $H$.
\end{corollary}
For the case of polynomial distortion see also \cite[Theorem C]{Ger-Pot-13-JEMS-qi-floyd}.

The reason for proving Theorem \ref{sederh} here is that we will use it in conjunction with Theorem \ref{thm:boundary-to-inside} to generalise a result of Bonk and Schramm which characterises hyperbolic groups (amongst finitely
generated groups with a word metric) 
as those which admit a quasi-isometric embedding into some $\HH^n$ \cite{BS-00-gro-hyp-embed}.
In an analogous way,
we characterise groups that are hyperbolic relative to virtually nilpotent
subgroups.
A \emph{truncated real hyperbolic space} is the complement of a collection of
(disjoint) horoballs in some $\HH^n$ which is then given the resulting length 
metric~\cite[page 362]{BH-99-Metric-spaces}.
A truncated real hyperbolic space is hyperbolic relative to the collection of boundary horospheres, in the sense of Definition~\ref{def:rel-hyp}.

\begin{theorem}\label{thm:truncated-hyp-poly}
	Suppose $G$ is finitely generated group with a word metric.
	Then the following are equivalent:
	\begin{itemize}
		\item $G$ is hyperbolic relative to some collection of virtually nilpotent subgroups,
		\item $G$ admits a polynomially distorted embedding into some truncated real hyperbolic space, and
		\item $G$ admits a subexponentially distorted embedding into some truncated real hyperbolic space.
	\end{itemize}
\end{theorem}
The inclusion $X \to \HH^n$ of a truncated real hyperbolic space $X$ into its ambient hyperbolic space is a coarse embedding (that is, a uniformly proper coarse Lipschitz map, Definition~\ref{def:poly-subexp-distortion}), but with exponential distortion.  
So if $G$ is hyperbolic relative to virtually nilpotent groups then the above theorem gives a coarse embedding of $G \to \HH^n$, which, as already mentioned, cannot be subexponentially distorted unless $G$ is hyperbolic.
The general problem of which groups admit coarse embeddings into real hyperbolic spaces (with any distortion function) is also of interest, see work of Hume and the second author~\cite{Hume-Sisto-17-fat-bigons} and of Tessera~\cite{Tess-20-coarse-emb-amenable-hyp}; in the latter paper it is shown that the only amenable groups which coarsely embed into real hyperbolic spaces are the virtually nilpotent ones.

The embedding in Theorem~\ref{thm:truncated-hyp-poly} is built in roughly the same way as 
Bonk--Schramm's result: one applies Assouad's embedding theorem to embed the boundary
into some Euclidean space, and then extends the map inside using 
Theorem~\ref{thm:boundary-to-inside} (the stronger version of Corollary~\ref{thm:boundary-to-inside-simple} which applies to relatively hyperbolic spaces such as a truncated real hyperbolic space).
Given such an embedding, to find the relatively hyperbolic structure on $G$ we use that a truncated real hyperbolic space is hyperbolic relative to the horospheres of the removed horoballs, and pull back the relatively hyperbolic structure to the embedded group (via Corollary~\ref{cor:polyembed-truncated-rel-hyp}, a consequence of Theorem \ref{sederh} and work of Dru\c tu~\cite{Dr-relhyp-qiinv} relating `metric' relative hyperbolicity, `asymptotic tree-gradedness' in her terminology, to `algebraic' relative hyperbolicity).

Theorem~\ref{thm:truncated-hyp-poly} cannot be extended to quasi-isometric embeddings: if $G$ is a non-uniform lattice in $\text{Isom}(\C\HH^2)$ then a quasi-isometric embedding into a truncated real hyperbolic space would induce a quasi-isometric embedding of the real Heisenberg group with its Carnot metric into a Euclidean space, which is impossible.
However, Theorem~\ref{thm:truncated-hyp-poly} and some consideration of the geometry of the horoballs motivates the following.
\begin{question}
  Is a finitely generated group hyperbolic relative to some collection of virtually abelian subgroups if and only if the group admits a quasi-isometric embedding into some truncated real hyperbolic space?	
\end{question}

\begin{remark}
	\label{rmk:healy-hruska}
	As this paper was being finished, in independent work Healy--Hruska released a preprint studying metric properties of boundaries of relatively hyperbolic groups~\cite{healy2020cusped}.  
	Similar to us, they show that quasi-isometries of relatively hyperbolic groups extend to quasi-isometries of the cusped spaces and quasisymmetries of the Bowditch boundary \cite[Theorem 1.2]{healy2020cusped}.

	While on the surface the results look similar to results in this paper, we would like to point out that the contents of the papers are actually rather different. In fact, the main point of \cite{healy2020cusped} is to show, roughly, that using different models of horoballs does not affect the cusped space, and they dedicate a lot of work to the study of different types of horoballs in order to do so. They are then able to extend maps on peripheral sets to horoballs, which, as indicated below Corollary \ref{thm:inside-to-boundary-simple}, is basically all one needs to extend quasi-isometries of relatively hyperbolic groups to quasi-isometries of cusped spaces. In contrast, we regard as our main contribution in Corollary \ref{thm:inside-to-boundary-simple} to be the case of non-coarsely-surjective quasi-isometric embeddings, while we only use horoballs very similar to one fixed model, so that our extensions come with less work.

	The papers complement each other in that, roughly, thanks to \cite{healy2020cusped} we can apply our results to different models of cusped spaces, while we strengthen certain results of \cite{healy2020cusped}, like \cite[Theorem 1.2]{healy2020cusped}, in the case of our particular cusped spaces. 
\end{remark}

\subsection{Outline}
In Section~\ref{sec:results} we give the full statements of our results, with definitions of the classes of maps involved on the spaces and their boundaries.  In particular, we highlight the ``shadow-respecting quasisymmetric embeddings'' in Definition~\ref{def:shadow-qs}, which is the appropriate class of maps on boundaries in our context.  The main results are the following.  First, Theorem~\ref{thm:inside-to-boundary} gives extensions of maps between relatively hyperbolic spaces to maps between their Bowditch boundaries, which leads to Corollary~\ref{thm:inside-to-boundary-simple}.  Second, Theorem~\ref{thm:boundary-to-inside} gives the converse, extending maps between the boundaries to maps between the spaces, which leads to Corollary~\ref{thm:boundary-to-inside-simple}.

Section~\ref{sec:rel-hyp-defs} contains preliminary material on the metric geometry of cusped spaces and their boundaries, and basic properties of the classes of quasisymmetric maps introduced in Section~\ref{sec:results}.
In Section~\ref{sec:subexp-distort} we prove Corollary~\ref{cor:subexpdistort-subgroup-rel-hyp} and Theorem~\ref{thm:truncated-hyp-poly}, assuming Theorem~\ref{thm:boundary-to-inside}.

The proofs of Theorems \ref{thm:inside-to-boundary} and \ref{thm:boundary-to-inside} are contained in Sections \ref{sec:extend-qi-to-bdry} and \ref{sec:extend-bdry-to-qi} respectively.

\subsection{Notation}

$A \preceq B$ means $A \leq \lambda B$ for some $\lambda \geq 1$;
$A \asymp B$ means $A \preceq B$ and $B \preceq A$.

$A \lesssim B$ means $A \leq B+C$ for some $C$;
$A \approx B$ means $A \lesssim B$ and $B \lesssim A$.

$A \leqi B$ means $A \leq \lambda B+\mu$ for some $\lambda \geq 1$ and $\mu \geq 0$; $A \qi B$ means $A \leqi B$ and $B \leqi A$.
Moreover $A\qi_{\lambda,\mu} B$ means that $A/\lambda-\mu\leq B\leq \lambda A+\mu$. 

For $U$ a subset of a metric space $X$, and $C >0$,
we denote the $C$-neighbourhood of $U$ by
$N_C(U) = \{x \in X : d(x,U) \leq C \}$.

\subsection{Acknowledgements}
We thank Christopher Hruska, David Hume and referee(s) for helpful comments.
The first author was partially supported by EPSRC grant EP/K032208/1.

\section{Statement of results}\label{sec:results}

There are various characterisations of relatively hyperbolic spaces $(X, \cP)$, where $X$ is a geodesic metric space and $\cP$ is a collection of (coarsely connected) subspaces of $X$ called `peripheral' subsets.  (Note that we allow peripheral sets to be bounded, which is not always the case in the literature. Also, we always assume $X$ is geodesic.)
In this paper, the main characterisation we use is that, roughly, $X$ is hyperbolic relative to $\cP$ (i.e.\ $(X,\cP)$ is relatively hyperbolic) if the cusped space $\Bow(X,\cP)$ formed by gluing in horoballs to each $P \in \cP$ is Gromov hyperbolic.
We allow (and later need) different models for these horoballs, entailing some preliminary discussion, so we defer the full definition to Section~\ref{sec:rel-hyp-defs}, see Definition~\ref{def:rel-hyp}.
In a few papers, including~\cite{DSp-05-asymp-cones, Dr-relhyp-qiinv}, relative hyperbolicity is referred to as `asymptotic tree-gradedness'.

Given a finitely generated group $G$ and subgroups $\{H_i\}$ we say that the pair $(G, \{H_i\})$ is relatively hyperbolic if any Cayley graph $X$ of $G$ is relatively hyperbolic with peripheral sets $\cP$ given by all left cosets of all $\{H_i\}$.
We give the precise definition in the next section (see Definition~\ref{def:rel-hyp-grp}).

To state our results, we have to first define the varieties of coarse maps and of quasisymmetric maps that we work with.

\subsection{Coarse maps}\label{ssec:coarse-defs}
We denote the metric on a metric space $X$ by $d_X$, or just $d$ if there is no possibility of confusion.
\begin{definition}\label{def:poly-subexp-distortion}
	A map $f:X\ra X'$ between metric spaces is \emph{coarsely Lipschitz} if there exists $C > 0$ so that for all $x,y \in X$,
	\[
		d_{X'}(f(x),f(y)) \leq C d_X(x,y)+C.
	\]		
	
	We say $f$ is \emph{$\tau$-uniformly proper},
	for some $\tau:[0,\infty) \ra \R$ with $\lim_{t \ra \infty}\tau(t)=\infty$,
	if it is coarsely Lipschitz and if for all $x,y \in X$,
	\[
		\tau(d_X(x,y)) \leq d_{X'}(f(x),f(y)).
	\]

	We say a $\tau$-uniformly proper map $f$ is 
	\begin{enumerate}
		\item \emph{a quasi-isometric embedding} if we can take $\tau$ linear,
		\item \emph{polynomially distorted} if we can take $\tau(t)=t^\alpha-C$ for
			some $\alp\in (0,1]$ and $C>0$, or
		\item \emph{subexponentially distorted} if we can take $\tau$ so that we have
			$\lim_{t\to\infty} \tau (t) / \log t=\infty$.
	\end{enumerate}
\end{definition}
These notions are progressively weaker: 
$(1) \Rightarrow (2) \Rightarrow (3)$.

When mapping between relatively hyperbolic spaces, our maps should interact well with
the peripheral sets.
\begin{definition}\label{def:coarse-respect-periph}
	A map $f:X \ra X'$ between relatively hyperbolic spaces $(X, \cP), (X', \cP')$
	\emph{coarsely respects peripherals} if
	\begin{enumerate}
		\item $\exists C>0$ so that
			for all $P \in \cP$ there exists $P' \in \cP'$ so that $f(P) \subset N_C(P')$,
		\item $\forall C'>0 \ \exists C>0$ so that
			for all $P' \in \cP'$ 
			either $\diam f^{-1}(N_{C'}(P')) \leq C$, or
			there exists $P \in \cP$ so that
			$f^{-1}(N_{C'}(P')) \subset N_C(P)$,
	\end{enumerate}
\end{definition}

We are also interested in polynomially distorting maps which have better control on peripheral sets, namely they behave like coarse snow\-flake maps on each.
\begin{definition}\label{def:snowflake-on-peripherals}
	A polynomially distorted map $f:X \ra X'$ between relatively hyperbolic spaces 
	$(X, \cP), (X', \cP')$
	which coarsely respects peripherals is a \emph{snowflake on peripherals}
	if $\exists C>0$ so that for every $P \in \cP$ there exists $\lambda \in (0,1]$ so that 
	for all $x,y \in P$,
	\[
		\frac{1}{C} d_X(x,y)^\lambda -C \leq d_{X'}(f(x),f(y)) \leq C d_X(x,y)^\lambda+C.
	\]	 
\end{definition}
Assuming (as we often will) that each $P \in \cP$ is unbounded, $\lambda$ is uniformly bounded away
from zero by a constant depending on the polynomial distortion of $f$.
A natural example of a snowflake on peripherals map is a quasi-isometric embedding which coarsely respects peripherals; in this case each $\lambda=1$.  
Fairly immediately, relatively hyperbolic spaces with unbounded peripherals and snowflake on peripheral maps form a category, see Proposition~\ref{prop:category-snowflake-periph}.

\subsection{Quasisymmetric maps}\label{ssec:qs-defs}
We now consider boundaries and the maps between them.
As discussed in Section~\ref{sec:rel-hyp-defs}, gluing a horoball $\cH(P)$ to each $P \in \cP$
gives a Gromov hyperbolic space $\Bow(X, \cP)$.
This space has Gromov boundary $\bdry \Bow(X, \cP) = \bdry(X, \cP)$ which comes
decorated with the following data:
\[
	\big( \bdry(X,\cP), \rho, \{(a_P,r_P)\}_{P \in \cP} \big),
\]
where $\rho$ is a visual metric with visual parameter $\eps>0$ with respect to a
basepoint $o \in X$,
each $a_P  = \bdry P \in \bdry(X, \cP)$ is the parabolic point associated to $\cH(P)$,
and each $r_P = e^{-\eps d(o, P)}$ is the radius of the ``shadow'' of $P$ in $\bdry(X, \cP)$.
We formalise this as follows, considering $\cP$ as an abstract index set for points $a_P$ and radii $r_P$ of balls in a given metric space $(Z,\rho)$.

\begin{definition}\label{def:shadow-mtc-spc}
	A \emph{shadow decorated metric space} is a tuple
	\[ \big(Z, \rho, \cB=\{(a_P,r_P)\}_{P\in \cP}\big), \]
	where $(Z,\rho)$ is a metric space, and for $P \in \cP$ we have
	$a_P \in Z, r_P >0$, with all $a_P$ distinct and $\sup r_P \preceq \diam Z$.
\end{definition}
The notion of quasisymmetric mappings was generalised to metric spaces by Tukia and V\"ais\"al\"a~\cite{TV-80-qs}.
\begin{definition}\label{def:qs}
	A \emph{distortion function} is a homeomorphism $\eta:[0,\infty) \ra [0,\infty)$.
	A topological embedding (respectively homeomorphism) $h:(Z,\rho)\ra (Z',\rho')$ between metric spaces $(Z,\rho)$ and $(Z',\rho')$ is called an \emph{$\eta$-quasi\-sym\-met\-ric embedding (resp.\ homeomorphism)}, for some distortion function $\eta$, if for every triple of points $x,y,z \in Z$ and $t \geq 0$, 
	\[
		\rho(x,y) \leq t \rho(x,z) \ \Longrightarrow \ 
		\rho'(h(x),h(y)) \leq \eta(t) \rho'(h(x),h(z)) .
	\]
	If $\eta$ is not specified we just call $h$ a \emph{quasisymmetric embedding (resp.\ homeomorphism)}.
\end{definition}
For maps between shadow decorated metric spaces, we refine this definition
as follows.  
The reader should have in mind boundaries of relatively hyperbolic spaces, and as we will see the definition characterises the extension to the boundary of a quasi-isometric embedding between spaces coarsely respecting peripherals, in particular the behaviour of such extensions around parabolic points.
To explain the terminology we use in the last property below, in analysis on metric spaces, for a metric space $(Z, \rho)$ and $\epsilon>0$ such that $(Z,\rho^\epsilon)$ is also a metric space (for example, $\epsilon \in (0,1]$), the identity map $(Z,\rho) \to (Z, \rho^\epsilon)$ is called a `snowflake' transformation.  
\begin{definition}\label{def:shadow-qs}
	Suppose $(Z,\rho,\cB), (Z', \rho', \cB')$ are shadow decorated metric spaces.
	A \emph{shadow-respecting $\eta$-quasisymmetric embedding $h$}
	is an $\eta$-quasisymmetric embedding $h:(Z, \rho) \ra (Z', \rho')$ so that, for a constant $C\geq 1$,
	\begin{enumerate}
		\item for all $(a, r) \in \cB$, there exists $(a',r') \in \cB'$
		where $a'=h(a)$ and for all $b \in Z$,
			\[ \rho(a,b) \leq r \ \Longrightarrow \ \rho'(a',h(b)) \leq C r'; \]
		\item for all $(a',r') \in \cB'$ if $a'=h(a), (a,r) \in \cB$ then for all $b \in Z$,
			\[ r \leq \rho(a,b) \ \Longrightarrow \ r' \leq C \rho'(a',h(b)), \]
			and if no such $(a,r)\in \cB$ exists, $\rho'(a',h(Z)) \geq r'/C$;
		\item for all $(a,r) \in \cB$, there exists $\lambda_{a}\in (0,\infty)$ so that,
		writing $(a',r')$ as in (1),
		for all $b,c \in Z$ if 
		\[
			\left(\frac{\rho(a, b)}{r}\right) \, \rho(a, b) 
			\leq \rho(b, c) \leq \rho(a, b) \leq r,
		\]
		then
		\[
			\frac{\rho'(h(b),h(c))}{r'}
			\asymp_{C} \left(\frac{\rho(b,c)}{r} \right)^{\lambda_a}.
		\]
	\end{enumerate}
	If $h$ is a homeomorphism we say that it is a
	\emph{shadow-respecting $\eta$-quasisymmetry}.
	
	If $\lambda_a$ equals a fixed $\lambda$ for all $(a,r)\in \cB$,
	we say $h$ \emph{asymptotically $\lambda$-snowflakes}.
\end{definition}
The first two properties here say that $h$ should respect the sizes of shadows,
while the third says that as we get close to the centre of a shadow the map should look
more and more like a suitably scaled snowflake embedding.
A basic example of a shadow-respecting quasisymmetry that asymptotically snowflakes is given by a change of visual metric on the boundary of a fixed relatively hyperbolic space: if $\rho$ and $\rho'$ are visual metrics with parameters $\epsilon$ and $\epsilon'$ respectively, then $\rho' \asymp \rho^{\epsilon'/\epsilon}$, and the radii $r_P = e^{-\epsilon d(o,P)}$ and $r_P' = e^{-\epsilon' d(o,P)}$ also satisfy $r_P' = (r_P)^{\epsilon'/\epsilon}$. Definition~\ref{def:shadow-qs} then holds with each $\lambda_a = \epsilon'/\epsilon$; note that the hypothesis in property (3) is superfluous.

\begin{remark}\label{rmk:shadow-qs}
	In Definition~\ref{def:shadow-qs}, if $h$ is a homeomorphism, we can state (1), (2) more simply as:
$h$ induces a bijection $\cB \ra \cB', (a,r) \mapsto (h(a), r'_{h(a)}) \in \cB'$,
and for all $(a,r) \in \cB$
\[
	B(a',r'_{h(a)}/C) \subset h( B(a,r)) \subset B(a',Cr'_{h(a)}).
\]
\end{remark}

It is not clear that shadow-respecting quasisymmetric maps form a category, however they do when the spaces involved are ``uniformly perfect'', see Proposition~\ref{prop:category-shadow-resp}.

\subsection{A correspondence between maps}

We can now state our two main theorems.
\begin{theorem}\label{thm:inside-to-boundary}
	Suppose $(X,\cP), (X', \cP')$ are relatively hyperbolic spaces, where each $P\in\cP$ is unbounded.
	
	Suppose $f:X \ra X'$ is a map which is a snowflake on peripherals
	(and so also polynomially distorting and coarsely respecting peripherals).
	Then $f$ extends to a quasi-isometric embedding $f_\Bow:\Bow(X,\cP)\ra \Bow(X',\cP')$,
	quantitatively.
	
	The boundary map $\bdry f_\Bow: \bdry (X, \cP) \ra \bdry (X', \cP')$ is
	a shadow-respecting quasisymmetric embedding, quantitatively.
	
	Moreover, if $f$ is a quasi-isometric embedding, then $\bdry f_\Bow$ 
	asymptotically $\frac{\eps'}{\eps}$-snowflakes, where $\eps, \eps'$ denote the visual
	parameters of the boundary metrics on $\bdry(X,\cP), \bdry(X',\cP')$, respectively.

	Finally, if $(X'',\cP'')$ is also a relatively hyperbolic space and $g:X'\to X''$ is a snowflake on peripherals, then $(g\circ f)_\Bow$ is bounded distance to $g_\Bow \circ f_\Bow$, and $\bdry (g\circ f)_\Bow = (\bdry g_\Bow) \circ (\bdry f_\Bow)$.
	That is, $(X,\cP) \mapsto \bdry \Bow(X,\cP), f\mapsto \bdry f_\Bow$ is a functor from relatively hyperbolic spaces having unbounded peripherals with snowflake on peripheral maps to shadow decorated metric spaces with shadow-respecting quasisymmetric embeddings.
\end{theorem}
Corollary~\ref{thm:inside-to-boundary-simple} is an immediate consequence, since a quasi-isometry that coarsely respects peripherals is automatically a snowflake on peripherals (with each $\lambda=1$), and we may discard any finite peripheral groups from $\{H_i\}$.
(As an aside, the assumption above on unbounded peripherals is necessary in order to get the uniqueness of extensions required for functoriality.)

The proof of Theorem \ref{thm:inside-to-boundary} is in Section \ref{sec:extend-qi-to-bdry}.

\begin{theorem}\label{thm:boundary-to-inside}
	Suppose $(X, \cP)$, $(X', \cP')$ are two relatively hyperbolic spaces
	that are visually complete (Definition~\ref{def:visually-complete}),
	and $h: \bdry(X,\cP) \ra \bdry(X',\cP')$
	is a shadow-respecting, $\eta$-quasisymmetric embedding.
	Then there exists a polynomially distorted embedding $\hat{h}:X \ra X'$, which is a snowflake on
	peripherals, so that $h=\bdry (\hat{h})_\Bow$, and $\hat{h}$ is unique up to bounded distance. 
	
	Moreover, if $h$ also asymptotically 
	$\frac{\eps'}{\eps}$-snowflakes then $\hat{h}$ is a quasi-isometric embedding, where $\eps, \eps'$ denote the visual
	parameters of the boundary metrics on $\bdry(X,\cP), \bdry(X',\cP')$, respectively.

	Finally, if $(X'',\cP'')$ is also a relatively hyperbolic space that is visually complete, and $j:\bdry(X',\cP')\to\bdry(X'',\cP')$ is a shadow-respecting, quasisymmetric embedding, then $\widehat{j\circ h}$ and $\hat{j}\circ \hat{h}$ are at bounded distance.
	That is, $\bdry(X,\cP) \mapsto X, h \mapsto \hat{h}$ is a functor from shadow decorated boundaries of relatively hyperbolic spaces with shadow-respecting quasisymmetric embeddings, to relatively hyperbolic spaces with equivalence classes of snowflake on peripheral maps, where two such maps are considered equivalent if they are bounded distance apart.
\end{theorem}
The property of being `visually complete' (Definition~\ref{def:visually-complete}) is satisfied by relatively hyperbolic groups having proper infinite peripheral groups (Proposition~\ref{prop:bowditch-visually-complete}), and is equivalent to the Bowditch boundary being uniformly perfect (Lemma~\ref{lem:bowditch-unif-pfct}).
The importance of this condition is two-fold: first, by work of Tukia--V\"ais\"al\"a it allows us to upgrade quasisymmetric maps to ``power'' quasisymmetric maps, which have quasi-isometric extensions by Bonk--Schramm, see Theorem~\ref{thm:bonk-schramm} and discussion.
Second, without it we cannot hope for uniqueness of extensions $\hat{h}$: consider the many quasi-isometries of $\Z$ which induce the same boundary maps.
Healy--Hruska also show that boundaries of relatively hyperbolic groups are uniformly perfect, see~\cite[Section 1.1]{healy2020cusped} and references within, particularly to work of Meyer~\cite{Meyer-09-diss-unif-perfect}.

Corollary~\ref{thm:boundary-to-inside-simple} now follows from Theorem~\ref{thm:boundary-to-inside}.
\begin{proof}[Proof of Corollary~\ref{thm:boundary-to-inside-simple}]
	If $h:\bdry(G,\{H_i\})\to\bdry(G',\{H_i'\})$ is a shadow-respecting quasisymmetry, then by Proposition~\ref{prop:category-shadow-resp} $h^{-1}$ is also a shadow-respecting quasisymmetry, and so $\id_G \approx \widehat{\id_{\bdry(G,\{H_i\})}} = \widehat{h \circ h^{-1}} \approx \hat{h} \circ \widehat{h^{-1}}$, where $\approx$ means the maps agree up to bounded error.
	Thus as $\hat{h}, \widehat{h^{-1}}$ are both coarsely Lipschitz, they must be quasi-isometries.
\end{proof}

The proof of Theorem \ref{thm:boundary-to-inside} is in Section \ref{sec:extend-bdry-to-qi} (with the map $\hat h$ being $\hat f|_X$ where $\hat f$ is a map between cusped spaces constructed via Theorem \ref{thm:bonk-schramm} due to Bonk and Schramm).

\section{Preliminaries on horoballs and relative hyperbolicity}\label{sec:rel-hyp-defs}

In this section we define models of horoballs and the notions of relatively hyperbolic groups and spaces, and their (Bow\-ditch) boundaries.  We then collect some preliminary results on quasi-centres in hyperbolic spaces, separation of horoballs, visual completeness and transient sets.

\subsection{Horoballs}
 There are many definitions of relatively hyperbolic groups and metric spaces.
 We will give one in terms of cusped spaces defined in a way that allows us to consider both real hyperbolic horoballs and the following combinatorial horoballs.
 
 \begin{definition}\label{def:graph-horoball}
 	Suppose $\Gamma$ is a connected graph with vertex set $V$ and edge set $E$,
 	where every edge has length one.
 	The \emph{horoball} $\cH(\Gamma)$ is defined to be the graph with
 	vertex set $V \times \N$ and
 	edges $((v, n), (v, n+1))$ of length $1$, for all $v \in V$, $n \in \N$,
 	and edges $((v, n), (v', n))$ of length $e^{-n}$, for all $(v, v') \in E$.	
 \end{definition}
 
 Note that $\cH(\Gamma)$ is quasi-isometric to the metric space
 constructed from $\Gamma$ by gluing to each edge in $E$ a copy of the strip
 $[0,1] \times [1, \infty)$ in the upper half-plane model of $\HH^2$,
 where the strips are attached to each other along $v \times [1,\infty)$.
 
 As is well known, these horoballs are hyperbolic with boundary a single point.
 Moreover, it is easy to estimate distances in horoballs. 
 (Recall that we write $A\approx B$ if $|A-B|$ is bounded by some constant.)
 \begin{lemma}\label{distestim}
 	Suppose $\Gamma$ and $\cH(\Gamma)$ are defined as above.
 	Let $d_\Gamma$ and $d_\cH$ denote the corresponding path metrics.
 	Then for each $(x,m), (y,n) \in \cH(\Gamma)$, we have 
 	\[
 		d_\cH((x,m),(y,n)) \approx 2 \log(d_\Gamma(x,y)e^{-\max\{m,n\}}+1)+|m-n|.
 	\]
 \end{lemma}
 \begin{proof}
	We may assume that $m \leq n$.
 	We can suppose $d_\Gamma(x,y)\geq 2 e^n$, as if not $|m-n| \leq d_\cH((x,m),(y,n))\leq 2+|m-n|$. 
	Consequently
	 $\log(d_\Gamma(x,y)e^{-n}+1)\approx \log d_\Gamma(x,y)-n$. By construction of $\cH(\Gamma)$, any geodesic 
 	$\gamma$ between $x$ and $y$ in $\cH(\Gamma)$ must, for some $t\geq n$, go from $(x,m)$ to $(x,t)$ changing
 	only the second coordinate,
 	then follow a geodesic $\gamma' \subset \Gamma \times \{t\} \subset \cH(\Gamma)$
 	to $(y,t)$, then back to $(y,n)$.	
 	Thus
 	\begin{equation}\label{eq-horoball-distance-estimate}
 		d_\cH((x,m),(y,n)) \leq 2(t-n)+|n-m|+ e^{-t} d_\Gamma(x,y),
 	\end{equation}
 	with equality for the best choice of $t$. It is readily seen that this value is attained for the least $t$ so that $l_t=e^{-t} d_\Gamma(x,y)$ satisfies $l_t/e+2\geq l_t$, that is, $l_t\leq C:=2e/(e-1)$.
 	So the best choice of $t$ is $t = \lceil \log (d_\Gamma(x,y)/C)\rceil$ which is $\geq n$, 
 	and the right hand side of \eqref{eq-horoball-distance-estimate}
	 is $\approx_{2+C} 2(\log d_\Gamma(x,y)-n)-2\log C + |n-m|$.
 \end{proof}

This metric distance estimate is all we really need in our results, so we abstract it as follows.
A metric space $P$ is \emph{($\epsilon$-)coarsely connected} if there exists $\epsilon>0$ so that for any $x,y\in P$ there exists $x=x_0, x_1,\ldots, x_n=y$ in $P$ with $d(x_{i-1},x_{i}) \leq \epsilon$ for each $i=1,\ldots,n$.
\begin{definition}
	\label{def:admissible-horoballs}
	Let $(P,d_P)$ be a coarsely connected metric space. 
	An \emph{admissible horoball} is a geodesic metric space $(\cH(P),d_{\cH(P)})$ which is $\delta$-hyperbolic with a single boundary point $\{a_P\}=\bdry\cH(P)$, and for which there exists $C>0, L \geq 1$ satisfying the following properties.
	There exists a subset $\hat{P}\subset P$ with $d_P(x,\hat{P})\leq C$ for all $x \in P$, and $\hat{P}$ is identified with a topologically embedded subset of $\cH(P)$.  Moreover we require:
	\begin{enumerate}
		\item for each $x \in \hat P\subset \cH(P)$ there exists a geodesic ray $\gamma_x:[0,\infty)\to\cH(P)$ with $\gamma_x(0)=x$ and (necessarily) $\gamma_x(\infty)=a_P$.
		\item for every $p\in\cH(P)$ there exists $x \in \hat P$ with $d_{\cH(P)}(p,\gamma_x)\leq C$, 
		\item for any $x,y\in \hat P$ with $d_{\cH(P)}(x,y)\leq 1$ we have $d_{\cH(P)}(x,y) \geq d_P(x,y)/L$,
		\item and for any $x,y \in \hat P$ and $m,n \geq 0$,
 	\[
		d_{\cH(P)}(\gamma_x(m),\gamma_y(n)) \approx_C 2 \log(d_P(x,y)e^{-\max\{m,n\}}+1)+|m-n|.
 	\]
	\end{enumerate}

	A collection of admissible horoballs on a collection of spaces is \emph{uniform} if the constants $\delta, C, L$ may be chosen uniformly.
\end{definition}
In our definition of a relatively hyperbolic space below, and in all our results, we require uniformly admissible horoballs in order to control the horoball geometry up to uniformly bounded errors.
\begin{example}[{\cite[Definitions 3.1,3.2]{Si-metrrh}}]
	\label{ex:graph-approx-horoballs}
	For a $k$-coarsely connected metric space $(P,d_P)$, define a graph approximation $\Gamma_P$ by choosing as vertex set a maximal $\epsilon$-separated subset $\hat{P}$ in $(P,d)$, and connecting by an edge any $x,y \in \hat{P}$ with $d_P(x,y) \leq R$ for a suitable $R$ chosen to make $\Gamma_P$ connected.  Then let $\cH(P)$ be the graph horoball (Definition~\ref{def:graph-horoball}) on $\Gamma_P$.
	Properties (1,2,3) of Definition~\ref{def:admissible-horoballs} follow by construction, and Definition~\ref{def:admissible-horoballs}(4) by Lemma~\ref{distestim}.
\end{example}

Other than graph horoballs, the only other example we work with directly is the following.
\begin{proposition}\label{prop:real-horoballs-admissible}
	Let $H$ be a horoball in $\HH^N$; in the upper half plane model $\R^N_+$ we may take $H = \{x = (x_i) \in \R^N_+ : x_N \geq 1\}$.
	Then $H$ is an admissible horoball for $\bdry H \cong \R^{N-1} \subset \overline{\R^N_+ \setminus H}$.
\end{proposition}
\begin{proof}
	Write $\|\cdot\|_2$ for the standard norm on $\R^{N-1}$.
	Definition~\ref{def:admissible-horoballs}(1,2) holds trivially.
	The hyperbolic distance between $(x,e^m),(y,e^n) \in \R^{N-1}\times [1,\infty)=H$ satisfies:
	\begin{align}
		& d_H((x,e^m),(y,e^n))
		\notag\\ & = 2\log\left(\frac{\sqrt{\|x-y\|_2^2+(e^m-e^n)^2} + \sqrt{\|x-y\|_2^2+(e^m+e^n)^2}}{2\sqrt{e^me^n}}\right)
		\notag\\ & \approx \log\left(\frac{\|x-y\|_2^2+(e^m+e^n)^2}{e^me^n}\right). \label{eq:roughisomhoroballs}
	\end{align}
	Without loss of generality suppose $m\leq n$.
	If $\|x-y\|_2 \leq e^{n}$, \eqref{eq:roughisomhoroballs} is $\approx \log e^{n-m} = |m-n|$.
	Otherwise, $\|x-y\|_2 \geq e^n$ and \eqref{eq:roughisomhoroballs} is 
	$\approx \log (\|x-y\|_2^2e^{-(m+n)}) 
	= 2\log(d_{\R^{N-1}}(x,y)e^{-n})+|n-m|$.
	In either case, this approximately equals the distance estimate of Definition~\ref{def:admissible-horoballs}(4).

	Finally,  the shortest path distance between $(x,1), (y,1) \in \bdry H \subset \overline{\HH^N\setminus H}$ is $\|x-y\|_2$, so
	\begin{align*}
		\frac{d_H((x,1),(y,1))}{\|x-y\|_2}
		& = \frac{2 \log\left( \|x-y\|_2+\sqrt{\|x-y\|_2^2+4}\right)-2\log(2)}{\|x-y\|_2}
		\\ & \to 1 \text{ as } \|x-y\|_2 \to 0,
	\end{align*}
	and Definition~\ref{def:admissible-horoballs}(3) follows.
\end{proof}

\subsection{Cusped spaces}

\begin{definition}[{cf.\ \cite[Definitions 3.3--3.4]{Si-metrrh}}]\label{def:rel-hyp}
	Suppose $(X,d)$ is a geodesic metric space, $\cP$ a collection of uniformly coarsely connected subsets of $X$ with metrics $d_P = d|_P$ for $P\in\cP$, and $\{\cH(P)\}_{P\in\cP}$ a collection of uniformly admissible horoballs for $\cP$.
	
	The \emph{cusped space} $\Bow(X,\cP)=\Bow(X,\{\cH(P)\}_{P \in \cP})$
	is the path metric space obtained from $X$ and $\{ \cH(P)\}_{P \in \cP}$
	by identifying each $\hat{P}$ with its copies in $\cH(P)$ and in $P \subset X$.

	If $\Bow(X, \cP)$ is a Gromov hyperbolic geodesic space, we say that $(X, \cP)$ is \emph{relatively hyperbolic}, or that $X$ is \emph{hyperbolic relative to} the collection of subsets $\cP$.
	The elements of $\cP$ are called \emph{peripheral sets}.
	We also assume $X$ and $\Bow(X,\cP)$ are proper.
\end{definition}
The original definition in \cite{Si-metrrh} used only the admissible horoballs of Example~\ref{ex:graph-approx-horoballs}, and showed that $(X,\cP)$ being relatively hyperbolic was independent of the choice of graph approximation horoballs \cite[Proposition 4.9]{Si-metrrh}.

Except for the application to truncated hyperbolic spaces (Theorem~\ref{thm:truncated-hyp-poly}), we are mainly interested in relatively hyperbolic group pairs $(G, \{H_i\})$ as defined below, and the reader may primarily keep this case in mind.
\begin{definition}\label{def:rel-hyp-grp}
	Suppose $G$ is a finitely generated group, and $\{H_i\}_{i=1}^n$
	a collection of finitely generated subgroups of $G$.
	Let $S$ be a finite generating set for $G$, so that $S \cap H_i$ generates
	$H_i$ for each $i=1, \ldots, n$.

	We say that $(G, \{H_i\})$ is \emph{relatively hyperbolic} if
	a Cayley graph $X$ of $G$ with respect to $S$,
	is hyperbolic relative to the collection $\cP$ of copies of Cayley graphs of $H_i$ in $X$ on each left
	coset of $H_i$, for each $i$.  (We may use graph horoballs, as in Definition~\ref{def:graph-horoball}, on the corresponding Cayley graphs.)
\end{definition}

Definitions~\ref{def:rel-hyp} and \ref{def:rel-hyp-grp} are equivalent to the other usual definitions of (strong) relative hyperbolicity;
see \cite[Proposition 4.9]{Si-metrrh}, \cite[Theorem 3.25]{Gro-Man-08-dehn-rel-hyp},
and also \cite{Bow-99-rel-hyp}.  It is also equivalent to being `asymptotically tree-graded', as in \cite{DSp-05-asymp-cones} and \cite{Dr-relhyp-qiinv}.

Note that, as previously remarked, we do not assume peripheral sets are unbounded.  We remark that the hyperbolicity of $\Bow(X,\cP)$ implies that neighbourhoods of different peripheral sets have bounded intersection, in particular two unbounded peripheral sets have infinite Hausdorff distance.  

The fact that we may take any uniformly admissible horoballs in Definition~\ref{def:rel-hyp} is allowed by the following result.
\begin{proposition}\label{prop:change-admissible-horoballs}
	Suppose $(X,d)$ is a geodesic metric space, $\cP$ a collection of uniformly coarsely connected subsets of $X$, and $\{\cH(P)\}_{P\in\cP}$ and $\{\cH'(P)\}_{P\in\cP}$ are two uniformly admissible collections of horoballs, with resulting cusped spaces $\Bow(X,\{\cH(P)\})$ and $\Bow(X,\{\cH'(P)\})$.
	Then there exists $C$ so that the following holds.

	For $x\in P\cap \cH(P)$ let $f(x) \in P\cap\cH'(P)$ be a point with $d_X(x,f(x)) \leq C$ for uniform $C$.
	Define $F:\Bow(X,\{\cH(P)\})\to \Bow(X,\{\cH'(P)\})$ by letting $F|_X=\id_X$, and for $p\in \cH(P) \subset \Bow(X,\{\cH(P)\})$ with $d_{\cH(P)}(p,\gamma_x(m))\leq C$ for some $x \in P\cap \cH(P)$, let $F(p)=\gamma'_{f(x)}(m)$.  Here $\{\gamma_x\}_{x\in P\cap \cH(P)}, \{\gamma'_x\}_{x\in P\cap\cH'(P)}$ are the two families of geodesic rays of the construction.

	Then $F:\Bow(X,\{\cH(P)\})\to \Bow(X,\{\cH'(P)\})$ is a quasi-isometry; in particular, $\Bow(X,\{\cH(P)\})$ is Gromov hyperbolic if and only if $\Bow(X,\{\cH'(P)\})$ is.
\end{proposition}
\begin{proof}
	We first show $F$ is coarsely Lipschitz.
	Suppose we have $p,q \in \Bow(X,\{\cH(P)\})$ connected by a geodesic $\gamma=[p,q]$.  We wish to estimate $d_{\Bow(X,\{\cH'(P)\})}(F(p),F(q))$ from above by building a curve $\gamma' \subset \Bow(X,\{\cH'(P)\})$ from $F(p)$ to $F(q)$.

	Let $\gamma_1,\ldots,\gamma_n$ be the closed essentially disjoint subgeodesics of $\gamma$ which have $\length(\gamma_i) \geq 1$, each lie in some $\cH(P_i)$, and have endpoints in $\cH(P_i)\cap P_i$.  The length condition ensures $n \leq \length([p,q])$.  If $p \in \cH(P_0)$ for some $P_0$ let $\gamma_0$ be the subsegment of $\gamma$ from $p$ to the first time it meets $X$, otherwise let $\gamma_0=\{p\}$.  Likewise if $q \in \cH(P_{n+1})$ for some $P_{n+1}$ let $\gamma_{n+1}$ be the last subsegment joining $X$ to $q$, else let $\gamma_{n+1}=\{q\}$.

	For each $i=1,\ldots,n$, replace $\gamma_i$ by a curve $\gamma'_i$ in $X \cup \cH'(P_i)$ consisting of bounded length segments in $X$ joining the endpoints of $\gamma_i$ to nearby points in $\cH'(P_i)\cap X$, and a geodesic in $\cH'(P_i)$ joining those points.  
	By Definition~\ref{def:admissible-horoballs}(4) applied to $\cH(P_i)$ and $\cH'(P_i)$, we have $\length(\gamma'_i) \leq \length(\gamma_i) + C_1$ for some constant $C_1$.

	If $p\in\cH(P_0)$, then let $x \in {P_0}\cap\cH({P_0})$ be given with $d_{\cH({P_0})}(p,\gamma_x(m)) \leq C$ for some $m \geq 0$.  Write $F(p)=\gamma'_{f(x)}(m)$ as hypothesised.
	Suppose $\gamma_0$ joins $p$ to some $z\in {P_0}\cap\cH({P_0})$, so 
	\[ \length(\gamma_0) \approx_C 2\log(d_{P_0}(x,z)e^{-m}+1)+m. \] 
	Now 
	\begin{align*}
		d_{\cH'({P_0})}(F(p),F(z))
		& = d_{\cH'({P_0})}(\gamma'_{f(z)}(m),z)
		\\ & \approx_C 2\log\left( d_{P_0}(f(x),z) e^{-m}+1\right)+m
		\approx \length(\gamma_0)
	\end{align*}
	since $d_{P_0}(f(x),x)\leq C$.  So we can replace $\gamma_0$ by a curve $\gamma'_0$ joining $F(p)$ to $z$ with 
	$\length(\gamma'_0) \leq \length(\gamma_0)+C_1$, possibly increasing $C_1$.  Do likewise with $\gamma_{n+1}$.

	The rest of the curve $\gamma\setminus\bigcup \gamma_i$ consists of points of $X$ and (countably many) non-trivial short paths in different horoballs $\cH(P)$.  By Definition~\ref{def:admissible-horoballs}(3), any path $[x,y] \subset \cH(P)$ with endpoints in $\cH(P)\cap X$ and $d_{\cH}(x,y) \leq 1$ can be replaced by a path in $X$ with the same endpoints of length $\leq L \length([x,y])$, where $L\geq 1$.

	Thus we have found a curve $\gamma'$ joining $F(p)$ to $F(q)$ with
	\begin{align*}
		\length(\gamma') & \leq \sum_{i=0}^{n+1} \length(\gamma'_i) + \length\left(\gamma'\setminus\bigcup \gamma'_i \right)
		\\ & \leq \sum_{i=0}^{n+1}\left(\length(\gamma_i)+C_1\right)+L \length\left (\gamma\setminus\bigcup \gamma_i\right)
		\\ & \leq L \length(\gamma) + (n+2)C_1
		\leq (L+C_1)\length(\gamma) + 2C_1
	\end{align*}
	as desired.

	If we define $G:\Bow(X,\{\cH'(P)\})\to\Bow(X,\{\cH(P)\})$ in a similar way to $F$, it also will be coarsely Lipschitz, and $G\circ F$ and $F \circ G$ are a bounded distance to the identity, since by hyperbolicity and Definition~\ref{def:admissible-horoballs}(4) of $\cH(P)$, if $x,y\in P\cap\cH(P)$ have $d_P(x,y) \approx 0$ then $d_{\cH(P)}(\gamma_x(m),\gamma_y(m)) \approx 0$.
	So $F$ is a quasi-isometry.
\end{proof}
\begin{remark}
	\label{rmk:bow-inclusion-unif-prop}
	A very similar argument, replacing this time each $\gamma_i$ by a path in $X$, allows us to show that the inclusion $\iota:X \to \Bow(X,\{\cH(P)\})$ is $\tau$-uniformly proper, with $\tau(t) = 2\log(t+1)-C$ for some $C$.
\end{remark}

\subsection{Visual metrics}
\label{visual-metric}
Let $Y$ be a proper, geodesic, Gromov hyperbolic space, with fixed basepoint $o \in Y$.
One equivalent definition of the boundary $\bdry Y$ is as the set of equivalence classes of geodesic
rays $\gamma: [0, \infty) \ra Y$, with $\gamma(0)=o$, where two rays are equivalent if they are at finite Hausdorff distance
from each other.
Let $(a | b) = (a | b)_o$ denote the Gromov product on $\bdry Y$ with respect to $o$.
Up to an additive error, $(a|b)$ equals the infimal distance from $o$ to
some (any) geodesic line from $a$ to $b$, which exists since we assume $Y$ is proper. 

A metric $\rho$ on $\bdry Y$ is a visual metric with visual parameter $\epsilon>0$ if 
there exists $C_0 >0$ so that $\frac{1}{C_0} e^{-\eps (a|b)} \leq \rho(a,b) \leq C_0 e^{-\eps (a|b)}$ 
for all $a, b \in \bdry Y$. 
Boundaries of cusped spaces will always be endowed with a visual metric.

Any two visual metrics $\rho, \rho'$ on $\bdry Y$ are \emph{snowflake equivalent}: there exists $\lambda \geq 1$ and $\alpha > 0$ so that $\rho' \asymp_\lambda \rho^\alpha$ (in fact $\alpha$ is the ratio of their visual parameters).

Suppose that $X$ is hyperbolic relative to a collection of subsets $\cP$, and let $\Bow(X,\cP)$ be a cusped space.
Endow $\bdry \Bow(X, \cP)$ with a visual metric $\rho$ with parameter $\epsilon>0$.
For each $P \in \cP$ we have a horoball with a unique limit point $a_P \in \bdry \Bow(X,\cP)$.  
We let $r_P = e^{-\epsilon d(o,P)}$ and call it the \emph{radius} of the ``shadow'' of $P$ in $\bdry \Bow(X,\cP)$.  (It is the approximate size of the set of limit points of all geodesic rays from $o$ that pass close to $P$.)
This makes
\[
	\left( \bdry \Bow(X, \cP), \rho, \{(a_P, r_P)\}_{P \in \cP} \right)
\]
a shadow decorated metric space (see Definition~\ref{def:shadow-mtc-spc}).

For results on the geometric properties of $\bdry \Bow(X, \cP)$, see \cite{MS-12-relhypplane}.

\subsection{Quasi-centres}
We make frequent use of the notion of a `quasi-centre'.
\begin{definition}
	\label{def:quasicentre}
	Suppose $Y$ is a $\delta_Y$--Gromov-hyperbolic geodesic space and $x,y,z \in Y\cup\bdry Y$ are distinct.
	Then a \emph{quasi-centre} for $x,y,z$ is a point $p \in Y$ which lies within $2\delta_Y$ of geodesics joining $x$ to $y$, $y$ to $z$, and $z$ to $y$.
\end{definition}
We record three elementary properties of quasi-centres.
\begin{lemma}[Quasi-centres exist and are coarsely unique]
	\label{lem:quasicentres-exist}
	If $Y$ is $\delta_Y$-hyperbolic and proper (to ensure existence of bi-infinite geodesics between boundary points), any distinct $x,y,z \in Y\cup\bdry Y$ have a quasi-centre $p$.  Moreover, there exists $C$ so that if a point $q \in Y$ lies within a distance $A$ of geodesics $[x,y]$, $[y,z]$ and $[z,x]$, then $d(q,p) \leq C+A$.
\end{lemma}
\begin{proof}
	Consider $p\in [x,y]$ as $p$ moves from $x$ to $y$; at a point where $p$ switches from being $2\delta_Y$-close to $[x,z]$ to being $2\delta_Y$-close to $[z,y]$, $p$ is a quasi-centre.
	For the last claim, consider a tree approximation $f:V\to T$ where $V$ is the union of $x,y,z,p,q$ and the geodesics between them, $T$ is a tree, and $f$ is a rough isometry, that is a quasi-isometry with only additive errors.  Let $b\in T$ be the unique point on geodesics $[f(x),f(y)],[f(y),f(z)]$ and $[f(z),f(x)]$.  Then $f(p)$ is within $C$ of $b$, and $f(q)$ is within $C+A$ of $b$ for some $C$ independent of $x,y,z,q,A$.
\end{proof}

\begin{lemma}[Quasi-centres preserved by quasi-isometries]\label{splittosplit}
	Let $Y,Y'$ be Gromov hyperbolic and proper, and let $f:Y\to Y'$ be a quasi-isometric embedding. Then there exists $C$ so that if $p$ is a quasi-centre of $x,y,z\in Y\cup \bdry Y$, then $f(p)$ is at most $C$ from a quasi-centre of $\overline{f}(x),\overline{f}(y),\overline{f}(z)$.
 (Here $\overline{f}$ denotes the extension of $f$ to $\overline{f}:Y \cup \bdry Y \ra Y' \cup \bdry Y'$.)
\end{lemma}
\begin{proof}
	Since $p$ is within $2\delta_{Y}$ from geodesics between $x,y,z$, we have $f(p)$ is within bounded distance of quasi-geodesics between $\overline{f}(x), \overline{f}(y), \overline{f}(z)$; by the Morse lemma these are bounded Hausdorff distance from geo\-des\-ics with the same endpoints.
	The distance estimate follows from Lemma~\ref{lem:quasicentres-exist}.
\end{proof}

\begin{lemma}[Quasi-centres and boundary distances]\label{splitdist}
 Suppose we have a proper Gromov hyperbolic space $Y$ with basepoint $o$ and a visual metric $\rho$ with parameter $\epsilon$. 
	There exists $C$ so that for any distinct $a,b\in\bdry X$, and quasi-centre $s$ of $o,a,b$, we have
 $$d(o,s)\approx_C \frac{-1}{\epsilon} \log \rho(a,b).$$
\end{lemma}
\begin{proof}
	Both sides of the equation are $\approx (a|b)_o$.
\end{proof}

\subsection{Separation of horoballs}
We now characterise the shadow decorations which arise from disjoint horoballs; both directions of this characterisation are needed in the proof of Theorem~\ref{thm:truncated-hyp-poly}.

Recall that the \emph{Busemann function} on a hyperbolic space $Y$ corresponding to $a \in \bdry Y, o \in Y$ is the function $\beta_a(\cdot, o):Y \to \R$ defined by
\[
	\beta_a(x,o) = \sup_{\gamma} \left( \limsup_{t\to\infty} (d(\gamma(t),x)-t) \right),
\]
where the supremum is taken over all geodesic rays $\gamma:[0,\infty)\to Y$ from $o$ to $a$.
\begin{remark}
	\label{rmk:horoballs-hn}
	Suppose $Y=\HH^2$ is the upper half-plane model, with $a=\infty$ the point at infinity, and $o=(0,1)$.
	The geodesic ray from $o$ to $a$ is $\gamma(t)=(0,e^t)$.
	Given $x=(x_1,x_2)$ and $t \geq \log x_2$, the vertical path from $x$ to $(x_1,e^t)$ has length $t-\log x_2$, and the horizontal path from $(x_1,e^t)$ to $(0,e^t)$ has length $|x_1|e^{-t}$, so
	\[
		t-\log x_2 \leq d(\gamma(t),x) \leq t-\log x_2 + |x_1|e^{-t},
	\]
	thus $\beta_a(x,o)=-\log x_2$.
	Therefore for any $a\in \bdry \HH^2, o\in \HH^2$ a set of the form $\{ y\in Y: \beta_a(y,o) \leq C\}$ defines a (classical) horoball, and likewise for $Y=\HH^n$.

	For a space $X$ hyperbolic relative to a collection of subsets $\cP$, Definition~\ref{def:admissible-horoballs}(4) gives that for $P\in \cP$ and $x\in P$, 
	$\beta_{a_P}(x,o) \approx -\log d(o,P)$.
\end{remark}
\begin{lemma}
	\label{lem:separationhoroballs}
	If a space $X$ is hyperbolic relative to a collection of subsets $\cP$ and $Z = \bdry \Bow(X,\cP)$ is a boundary with visual metric $\rho$ and the shadow decoration $\{(a_P, r_P)\}_{P \in \cP}$ described in Subsection \ref{ssec:qs-defs}, then there exists $C$ so that for all $P \neq Q \in \cP$,
	\begin{equation}\label{eq:sep}
		\rho(a_P, a_Q) \succeq_C \sqrt{r_P r_Q}.
	\end{equation}

	Conversely, suppose $Y$ is Gromov hyperbolic with basepoint $o$ and boundary $Z = \bdry Y$ having visual metric $\rho$ with parameter $\epsilon$.
	Suppose further that  $\{(a_P,r_P)\}_{P \in \cP}$ is any (abstract) shadow decoration on $Z$ that satisfies \eqref{eq:sep} for some $C$.
	Then there exists $t_0$ so that for any $t \geq t_0$ the horoballs defined by 
	\[ 
		H_{P,t}= \Big\{y \in Y : \beta_{a_P}(y,o) \leq -t+\frac{1}{\epsilon} \log r_P \Big\}, \quad P \in \cP
	\]
	are pairwise separated by a distance $\geq 2(t-t_0)$.
\end{lemma}
\begin{proof}
	First, suppose $P\neq Q \in \cP$, so by construction the horoballs $\cH(P), \cH(Q) \subset X$ are essentially disjoint (being combinatorial horoballs attached to distinct peripheral cosets).
	Let $p \in [o,a_P)$, $q \in [o,a_Q)$ be points with $r_P = e^{-\epsilon d(o,p)}, r_Q=e^{-\epsilon d(o,q)}$, so they are roughly the closest points to $o$ in $\cH(P), \cH(Q)$ respectively.
	Suppose $r_Q \leq r_P$.

	If $[o,a_Q)$ passes through $\cH(P)$ then the geodesics $[o,a_P), [o,a_Q)$ branch in $\cH(P)$ and
	$d(o,q) \gtrsim d(o,p)+2\left( (a_P|a_Q) - d(o,p) \right)$ (see Figure~\ref{fig:horoballsep1}).
	Thus $(a_P|a_Q) \lesssim \frac{1}{2}\left(d(o,p)+d(o,q)\right)$ and \eqref{eq:sep} follows.
\begin{figure}[h]
   \def\svgwidth{0.7\textwidth}
   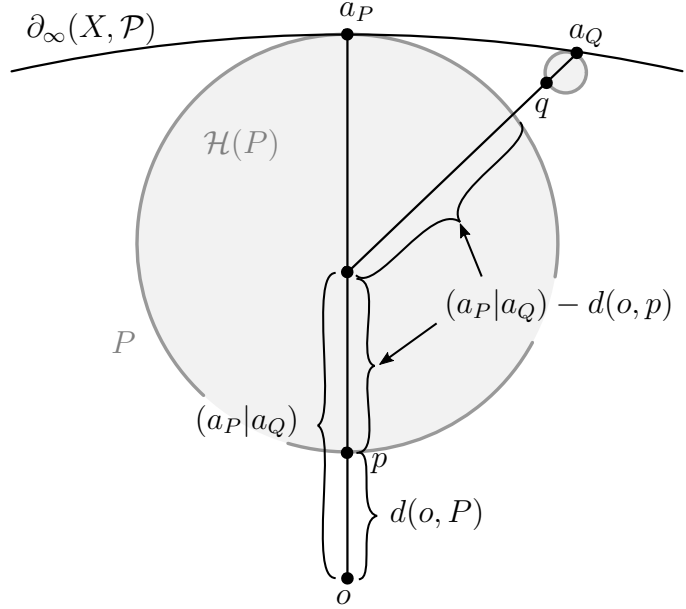
	\caption{Disjoint horoballs}
	\label{fig:horoballsep1}
\end{figure}

	If $[o,a_Q)$ does not pass through $\cH(P)$ then $(a_P|a_Q) \lesssim d(o,p)$ and moreover
	$0 \leq d(p,q) \approx (d(o,p)-(a_P|a_Q))+(d(o,q)-(a_P|a_Q))$
	so $(a_P|a_Q) \lesssim \frac{1}{2}\left(d(o,p)+d(o,q)\right)$
	and \eqref{eq:sep} follows.

	For the converse, 
	it suffices to find $t_0$ so that the horoballs are pairwise disjoint, since in that case for any $t \geq t_0$, as $\beta_{a_P}(\cdot,o)$ is $1$-Lipschitz $d(H_{P,t}, Y\setminus H_{P,t_0}) \geq t-t_0$, and likewise for $H_{Q,t}, H_{Q,t_0}$.  Considering a geodesic from $H_{P,t}$ to $H_{Q,t}$ we have $d(H_{P,t},H_{Q,t}) \geq 2(t-t_0)$. 

	Suppose for some $t_0$ we can find $P, Q \in \cP$ so that $H_{P,t_0}\cap H_{Q,t_0} \neq \emptyset$.
	By the quasiconvexity of the horoballs, there must be a point $z$ along the geodesic $(a_{P},a_{Q})$ almost in the horoballs, i.e., with 	
	$\beta_{a_{P}}(z,o) \lesssim -t_0+\frac{1}{\epsilon}\log r_{P}$ 
	and $\beta_{a_{Q}}(z,o) \lesssim -t_0+\frac{1}{\epsilon}\log r_{Q}$.  
	Let $y \in (a_{P},a_{Q})$ be a quasi-centre for $o,a_{P},a_{Q}$; without loss of generality, $a_{P}, z, y, a_{Q}$ appear in this order.
	Then
	$\beta_{a_{P}}(z,o) \approx -d(o,y)-d(y,z)$
	and
	$\beta_{a_{Q}}(z,o) \approx -d(o,y)+d(y,z)$
	so by \eqref{eq:sep}
	\begin{align*}
		-2(a_{P}|a_{Q})
		& \approx -2d(o,y)
		\approx \beta_{a_{P}}(z,o) + \beta_{a_{Q}}(z,o) 
		\\ & \lesssim -2t_0+\frac{1}{\epsilon}(\log r_{P}+\log r_{Q})
		\lesssim -2t_0-2(a_{P}|a_{Q}),
	\end{align*}
	thus $t_0 \lesssim 0$.  So choosing $t_0$ large enough we have a contradiction. 
\end{proof}

\subsection{Uniform perfectness and visual completeness}

When studying quasisymmetric maps it is useful to have, near any point, points of approximately any possible distance away.
\begin{definition}\label{def:unif-perfect}
	A metric space $(Z,\rho)$ is \emph{uniformly perfect} if for some $\lambda > 1$, for any $z \in Z$ and $0 < r \leq \diam Z$ we have that $B(z,r) \setminus B(z,r/\lambda) \neq \emptyset$.
\end{definition}
An easy observation is that
\begin{lemma}
	$(Z,\rho)$ is uniformly perfect if and only if there exists $\lambda \geq 1$ so that for any $B(z,r) \subset Z, r \leq \diam(Z)$, we have $\diam B(z,r) \asymp_\lambda r$.
\end{lemma}
For cusped spaces, there is the following related characterisation.
\begin{definition}\label{def:visually-complete}
 Suppose a space $X$ is hyperbolic relative to $\cP$.  We will say that $X$ is \emph{visually complete} if for some (hence, any) $w\in \Bow(X,\cP)$ there exists $C\geq 0$ so that 
 for any $x\in \Bow(X,\cP)$ there exist geodesic rays $[w,a)$, $[w,b)$, for some distinct $a,b\in \bdry \Bow(X,\cP)$, so that $x$ is within distance $C$ of $[w,a)$, $[w,b)$ and some bi-infinite geodesic from $a$ to $b$.
\end{definition}

Recall that $\Bow(X,\cP)$ is \emph{visual} if for some (hence any) basepoint $w$ there exists $C \geq 0$ so that every point $x \in \Bow(X,\cP)$ is within distance $C$ from some infinite geodesic ray from $w$.
\begin{lemma}\label{lem:bowditch-unif-pfct}
	Suppose a space $X$ is hyperbolic relative to $\cP$.
	Then $X$ is visually complete if and only if $\Bow(X,\cP)$ is visual and $\bdry\Bow(X,\cP)$ is uniformly perfect.
\end{lemma}
\begin{proof}
	We may assume that $o$ is the basepoint of $\Bow(X,\cP)$, and that the visual metric on $Z = \bdry \Bow(X,\cP)$
	satisfies $\rho(a,b) \asymp e^{-\eps (a|b)}$.

	If $X$ is visually complete, $\Bow(X,\cP)$ is certainly visual.
	Let $t$ be a constant to be determined.
	Given $a \in \bdry \Bow(X,\cP)$ and $r \leq \diam(Z)$, let $x$ be a point on $[o,a)$ at distance $-\frac{1}{\epsilon}\log(r)+t$ from $o$; this is always possible for $t \geq \frac{1}{\epsilon}\log(\diam(Z))$.
	Then by visual completeness there exist $b,c \in \bdry \Bow(X,\cP)$ so that $(b|a)$ and $(c|a)$ are $\gtrsim t-\frac{1}{\epsilon}\log(r)$ and $(b|c) \approx t-\frac{1}{\epsilon}\log(r)$.
	Therefore $\rho(a,b)$ and $\rho(a,c)$ are $\preceq e^{-\eps t} r$.
	We now fix $t$ large enough to ensure that $b,c \in B(a,r)$.
	Since $\rho(b,c) \asymp e^{-\eps t} r$ we have $\diam B(a,r) \asymp_\lambda r$ for $\lambda$ independent of $a,r$ as desired.
	
	Conversely, given $x \in \Bow(X,\cP)$, find a geodesic ray $[o,a)$ within $C$ of $x$.
	By uniform perfectness, there exists $b \in \bdry \Bow(X,\cP)$ with $\rho(a,b) \asymp e^{-\epsilon d_{\Bow}(o,x)}$
	thus $(a|b) \approx d_{\Bow}(o,x)$ and so $a,b$ satisfy the desired condition.
\end{proof}

Relatively hyperbolic groups are always visually complete, and hence have uniformly perfect boundaries (for the latter fact and a variation on visual completeness, see also \cite[Proposition 6.3]{healy2020cusped}).

\begin{proposition}\label{prop:bowditch-visually-complete}
 Let $G$ be a group hyperbolic relative to its proper, infinite subgroups $H_1,\dots,H_n$. Then if $G$ is not virtually cyclic it is visually complete.
\end{proposition}

\begin{proof}
	We will show that there exists $D$ so that for any $x\in \Bow(G,\{H_i\})$ there exist three geodesic rays $\gamma_i$ towards $a_i\in\bdry(G,\{H_i\})$ emanating from $x$ so that $(a_i|a_j)_x\leq D$ for any $i\neq j$.  To check visual completeness for given $o,x$, observe that for (at least) two of the rays as above, say $\gamma_1,\gamma_2$, we have $(a_1|o)_x, (a_2|o)_x \leq D'$, for a suitable $D'$. Setting $a_1=a,a_2=b$, visual completeness easily follows from the fact that concatenating geodesics in a hyperbolic space yields a quasi-geodesic with constants only depending on the hyperbolicity constant and the Gromov product at the concatenation point.

 The rays can be chosen as follows. 
 If $n=0$ then, as is well known, as $G$ is not virtually cyclic the boundary at infinity has at least three points, and we can move the quasi-centre of these three points to a bounded distance from any $x$ using the action of $G$.

	If $n\geq 1$ then for $x\in G$, we can take 3 rays towards the points at infinity of distinct horoballs each within uniformly bounded distance of $x$.
 
 For $x=(g,m)$ in a graph horoball $\cH(P)$, for convenience we describe 3 quasi-geodesic rays instead of 3 geodesic rays, and they are the following:
 \begin{itemize}
	 \item a ray $\{g\}\times[m,\infty)$ towards the point at infinity of the horoball,
  \item the concatenation of a geodesic from $(g,m)$ to $g$ and a ray from $g$ towards the point at infinity of some horoball $\cH(P')\neq \cH(P)$ within uniformly bounded distance of $g$, and
  \item the concatenation of a horizontal geodesic of length $l$ with $1\leq l \leq e$ connecting $(g,m)$ to, say, $(g',m)$, a geodesic from $(g',m)$ to $g'$ and a ray from $g'$ towards the point at infinity of some horoball $\cH(P'')\neq \cH(P),\cH(P')$ within uniformly bounded distance of $g'$.
	  \qedhere
 \end{itemize}
\end{proof}

\subsection{Uniform perfectness and shadow decorations}
There are consequences of uniform perfectness for shadow decorated spaces.
In the proof of Theorem~\ref{thm:boundary-to-inside}, we use the lower bound in the following proposition to get polynomial distortion of the extension of the boundary map to the inside.
\begin{proposition}
	\label{prop:bounded-snowflake-param}
	If $(Z,\rho,\cB), (Z',\rho',\cB')$ are shadow decorated metric spaces, and $(Z,\rho)$ is bounded and uniformly perfect, then any shadow-respecting quasisymmetric embedding $h:Z\to Z'$ will have the exponents $\lambda_a, (a,r)\in \cB$, uniformly bounded away from zero and infinity.
\end{proposition}
\begin{proof}
	The uniform perfectness and boundedness of $Z$ implies $h$ is bi-H\"older~\cite[Theorem 3.14]{TV-80-qs}: there exists $\alpha \geq 1$ so that for all $x,y \in Z$, 
	\[
		\rho(x,y)^\alpha \preceq \rho'(h(x),h(y)) \preceq \rho(x,y)^{1/\alpha}.
	\]

	Take $(a,r) \in \cB$, with corresponding $(a',r') \in \cB'$.
	By (uniform) perfectness there exists a sequence $(b_i)$ in $Z\setminus \{a\}$ with $b_i \to a$.
	Applying Definition~\ref{def:shadow-qs}(3) with $b=b_i, c=a$ for all large $i$ we have
	$\rho'(h(b_i),h(a)) \asymp r' r^{-\lambda_a} \rho(b_i,a)^{\lambda_a}$,
	so
	\[
	\rho(b_i,a)^{\alpha} \preceq r' r^{-\lambda_a} \rho(b_i,a)^{\lambda_a} \preceq \rho(b_i,a)^{1/\alpha}.
	\]
	Letting $i\to\infty$, we deduce that $\alpha \geq \lambda_a \geq 1/\alpha$.
\end{proof}

The following proposition shows that shadow-respecting quasisymmetric embeddings between uniformly perfect, shadow decorated spaces form a category, motivating the definition of such maps.
The only places we use this proposition are in the statement of Theorem~\ref{thm:boundary-to-inside} and its application to Corollary~\ref{thm:boundary-to-inside-simple} (if $h$ is a shadow-respecting quasisymmetry then so is $h^{-1}$ and $\widehat{h \circ h^{-1}} \approx \hat{h} \circ \widehat{h^{-1}}$).
Therefore, as the proof is somewhat tedious, we suggest the reader skips it on a first reading.

\begin{proposition}
	\label{prop:category-shadow-resp}
	Let $(Z_i,\rho_i,\cB_i)$ for $i=1,2,3$ be uniformly perfect, shadow decorated metric spaces.

	If $h_i:Z_i\to Z_{i+1}$ are shadow-respecting quasisymmetric embeddings for $i=1,2$, then $h_2\circ h_1:Z_1\to Z_3$ is a shadow-respecting quasisymmetric embedding, quantitatively.
	The identity map is a shadow-respecting quasisymmetric embedding, so uniformly perfect shadow decorated spaces and shadow-respecting quasisymmetric embeddings form a category.

	If $h:Z_1\to Z_2$ is a shadow-respecting quasisymmetric embedding that is a quasisymmetry, then $h^{-1}:Z_2\to Z_1$ is a shadow-respecting quasisymmetry.
\end{proposition}
\begin{proof}
	We first show $h_2 \circ h_1$ is shadow-respecting.  If $h_i$ is an $\eta_i$-quasisymmetric embedding for $i=1,2$ then $h_2 \circ h_1$ is an $\eta_2 \circ \eta_1$-quasisymmetric embedding.

	(1,2) Given $(a_1,r_1) \in \cB_1$, there exists $(a_2,r_2)\in \cB_2$ and $(a_3,r_3)\in \cB_3$ with $h_1(a_1)=a_2$ and $h_2(a_2)=a_3$.
	By uniform perfectness, there exists $p,q \in Z_2$ with $\rho_2(p,a_2) \leq r_2 \leq \rho_2(q,a_2) \asymp \rho_2(p,a_2)$.
	For (1), if $\rho_1(a_1,b) \leq r_1$, then $\rho_2(a_2,h_1(b))\preceq r_2 \asymp \rho_2(a_2,p)$, so by quasisymmetry and (1) for $h_2$,
	\[
		\rho_3(a_3,h_2h_1(b)) \preceq \rho_3(a_3,h_2(p)) \preceq r_3.
	\]
	Likewise for (2), if $\rho_1(a_1,b) \geq r_1$, then $\rho_2(a_2,h_1(b)) \succeq r_2 \asymp \rho_2(a_2,q)$, so by quasisymmetry and (2) for $h_2$,
	\[
		\rho_3(a_3,h_2h_1(b)) \succeq \rho_3(a_3,h_2(q)) \succeq r_3.
	\]
	If $(a_3,r_3) \in \cB_3$ has no $(a_1,r_1) \in \cB_1$ with $h_2  h_1(a_1)=a_3$, then either $a_3 \notin h_2(Z_2)$ so $\rho_3(a_3,h_2  h_1(Z_1)) \geq \rho_3(a_3,h_2(Z_2)) \succeq r_3$,
	or $a_3=h_2(a_2)$ for some $(a_2,r_2) \in \cB_2$ but $a_2 \notin h_1(Z_1)$, in which case $\rho_2(a_2,h_1(Z_1)) \succeq r_2$, which again implies $\rho_3(a_3,h_2h_1(Z_1)) \succeq r_2$.

	(3) For $(a_i,r_i)\in \cB_i$, $i=1,2,3$ as in (1), suppose $b,c\in Z_1$ satisfy
	\begin{equation}
		\label{eq:cat-sr0}
		\frac{\rho_1(a_1,b)^2}{r_1} \leq \rho_1(b,c) \leq \rho_1(a_1,b) \leq r_1.
	\end{equation}
	If 
	\begin{equation}
		\label{eq:cat-sr3}
		\frac{\rho_2(a_2,h_1(b))^2}{r_2} \leq \rho_2(h_1(b),h_1(c)) \leq \rho_2(a_2,h_1(b)) \leq r_2 
	\end{equation}
	we are done, since
	\begin{equation}
		\label{eq:cat-sr2}
		\frac{\rho_3(h_2h_1(b),h_2h_1(c))}{r_3}
		\asymp \left( \frac{\rho_2(h_1(b),h_1(c))}{r_2}\right)^{\lambda_{a_2}}
		\asymp \left(  \frac{\rho_1(b,c)}{r_1}\right)^{\lambda_{a_1}\lambda_{a_2}}.
	\end{equation}
	So it remains to consider when \eqref{eq:cat-sr3} fails.

	We claim that if $\rho_1(a_1,b) \asymp r_1$ we are done.
	Indeed, if $\rho_1(a_1,b) \asymp r_1$ we can find $e\in Z_1$ with $\rho_1(a_1,e) \geq r_1$ and $\rho_1(a_1,e)\asymp r_1 \asymp \rho_1(a_1,b)$, so by (1,2) and quasisymmetry we have 
	\[
		\rho_2(a_2,h_1(e)) \succeq r_2 \succeq \rho_2(a_2,h_1(b)) \asymp \rho_2(a_2,h_1(e)),
	\]
	thus $\rho_2(a_2,h_1(b)) \asymp r_2$.
	Moreover, $\rho_1(a_1,b)\asymp r_1$ and \eqref{eq:cat-sr0} imply that $\rho_1(b,c) \asymp \rho_1(a_1,b)$ so by quasisymmetry,
	\[
		\rho_2(h_1(b),h_1(c)) \asymp \rho_2(a_2,h_1(b)) \asymp r_2.
	\]
	The same argument applied to $h_2$ then gives that $\rho_3(h_2h_1(b),h_2h_1(c)) \asymp \rho_3(a_3,h_2h_1(b)) \asymp r_3$.
	So both $\rho_3(h_2h_1(b),h_2h_1(c))/r_3$ and $\rho_1(b,c)/r_1$ are $\asymp 1$ giving \eqref{eq:cat-sr2}, proving the claim that $\rho_1(a_1,b)\asymp r_1$ suffices.
	
	Thus if $\rho_2(a_2,h_1(b)) > r_2$ by (1) we have $\rho_1(a_1,b) \asymp r_1$ and by the claim we are done.  So suppose $\rho_2(a_2,h_1(b)) \leq r_2$.  In particular, applying (3) to $h_1(b),a_2$ and $b,a_1$ we have
	\[ 
	\frac{\rho_3(h_2h_1(b),a_3)}{r_3} \asymp \left(\frac{\rho_2(h_1(b),a_2)}{r_2}\right)^{\lambda_{a_2}} \asymp \left(\frac{\rho_1(b,a_1)}{r_1}\right)^{\lambda_{a_1}\lambda_{a_2}}. \]
	Assume from now on that $\rho_2(a_2,h_1(b))\leq r_2$.

	If $\rho_2(h_1(b),h_1(c)) > \rho_2(a_2,h_1(b))$ then by quasisymmetry and \eqref{eq:cat-sr0} we have $\rho_2(h_1(b),h_1(c)) \asymp \rho_2(a_2,h_1(b))$.
	Thus by uniform perfectness there exists $e \in Z_2$ so that $\rho_2(h_1(b),h_1(c)) \asymp \rho_2(h_1(b),e) \leq \rho_2(a_2,h_1(b))$.
	If $\rho_2(a_2,h_1(b))^2/r_2 > \rho_2(h_1(b),e)$ then $\rho_2(a_2,h_1(b)) /r_2 \succeq 1$ and again the claim above implies we are done.
	So by quasisymmetry and (3) applied to $h_1(b),e$ and $b,c$ we have
	\begin{equation}\label{eq:cat-sr1}\begin{split}
		& \frac{\rho_3(h_2h_1(b),h_2h_1(c))}{r_3}
		\asymp \frac{\rho_3(h_2h_1(b),h_2(e))}{r_3}
		\asymp \left( \frac{\rho_2(h_1(b),e)}{r_2} \right)^{\lambda_{a_2}}
		\\ & \asymp \left( \frac{\rho_2(h_1(b),h_1(c))}{r_2}\right)^{\lambda_{a_2}}
		\asymp \left( \frac{\rho_1(b,c)}{r_1}\right)^{\lambda_{a_1}\lambda_{a_2}}.
	\end{split}\end{equation}
	Assume from now on that $\rho_2(h_1(b),h_1(c)) \leq \rho_2(a_2,h_1(b))$.

	Finally, if $\rho_2(h_1(b),h_1(c)) < \rho_2(h_1(b),a_2)^2/r_2$ then since by \eqref{eq:cat-sr0}
	\[
		\frac{\rho_2(h_1(b),h_1(c))}{r_2} 
		\asymp \left( \frac{\rho_1(b,c)}{r_1} \right)^{\lambda_{a_1}} 
		\geq \left( \frac{\rho_1(a_1,b)}{r_1}\right)^{2\lambda_{a_1}}
		\asymp \left( \frac{\rho_2(a_2,h_1(b))}{r_2} \right)^2
	\]
	we have $\rho_2(h_1(b),h_1(c)) \asymp \rho_2(a_2,h_1(b))^2/r_2$. So again we can find $e \in Z_2$ with $\rho_2(h_1(b),e) \asymp \rho_2(h_1(b),h_1(c))$ and $\rho_2(a_2,h_1(b))^2/r_2 \leq \rho_2(h_1(b),e) \leq \rho_2(a_2,h_1(b)) \leq r_2$, and a similar argument to \eqref{eq:cat-sr1} shows \eqref{eq:cat-sr2} holds.  This completes the proof of property (3) for $h_2   h_1$.

	Now suppose $h:Z_1\to Z_2$ is a shadow-respecting quasisymmetric embedding that is a quasisymmetry.  
	As is standard, if $h$ is an $\eta$-quasisymmetry, then $h^{-1}$ is a $\eta'$-quasisymmetry, where we set $\eta'(t) = 1/\eta^{-1}(1/t)$.
	It remains to check (1,2,3).

	(1,2) For $(a_1,r_1) \in \cB_1, (a_2,r_2) \in \cB_2$ with $h(a_1)=a_2$, by uniform perfectness pick $p,q \in Z_1$ with $\rho_1(p,a_1) \leq r_1 \leq \rho_1(q,a_1)$ and $\rho_1(p,a_1) \asymp \rho_1(q,a_1)$.
	Then (1,2) for $h$ and quasisymmetry give
	\[
		\rho_2(a_2,h(q))\asymp \rho_2(a_2,h(p)) \preceq r_2 \preceq \rho_2(a_2,h(q)).
	\]
	To see (1) for $h^{-1}$, given $b \in Z_2$, if 
	$\rho_2(a_2,b)\leq r_2$ then $\rho_2(a_2,b) \preceq \rho_2(a_2,h(q))$ so by the quasisymmetry of $h^{-1}$ we have $\rho_1(a_1,h^{-1}(b)) \preceq \rho_1(a_1,q) \asymp r_1$ as desired.  A similar proof gives (2).

	(3) Suppose $\rho_2(a_2,h(b))^2/r_2 \leq \rho_2(h(b),h(c)) \leq \rho_2(a_2,h(b)) \leq r_2$.
	If $\rho_1(a_1,b)^2/r_1 \leq \rho_1(b,c) \leq \rho_1(a_1,b) \leq r_1$ then by (3) for $h$ we have
		\begin{equation}\label{eq:cat-sr4}
			\frac{\rho_1(b,c)}{r_1} \asymp \left(\frac{\rho_2(h(b),h(c))}{r_2}\right)^{1/\lambda_{a_1}}.
		\end{equation}
	If $r_1 \preceq \rho_1(a_1,b)$, then we have $r_2 \preceq \rho_2(a_2,h(b)) \leq r_2$, so $\rho_2(h(b),h(c)) \asymp r_2 \asymp \rho_2(a_2,h(b))$, and by quasisymmetry $\rho_1(b,c) \asymp \rho_1(a_1,b)$ and again \eqref{eq:cat-sr4} holds with both sides $\asymp 1$ so we are done.  

	Assume from now on that $\rho_1(a_1,b) \leq r_1$ so by (3) applied to $b,a_1$ we have $\rho_2(h(b),a_2)/r_2 \asymp (\rho_1(b,a_1)/r_1)^{\lambda_{a_1}}$.

	If $\rho_1(b,c) > \rho_1(a_1,b)$ then by quasisymmetry $\rho_1(b,c) \asymp \rho_1(a_1,b)$ and $\rho_2(h(b),h(c)) \asymp \rho_2(a_2,h(b))$, so
	\[
		\frac{\rho_1(b,c)}{r_1} \asymp \frac{\rho_1(b,a_1)}{r_1}
		\asymp \left( \frac{\rho_2(h(b),a_2)}{r_2}\right)^{1/\lambda_{a_1}}
		\asymp \left( \frac{\rho_2(h(b),h(c))}{r_2}\right)^{1/\lambda_{a_1}}.
	\]
	Finally, suppose $\rho_1(b,c) < \rho_1(a_1,b)^2/r_1$.
	By uniform perfectness we can find $e\in Z_1$ with
	$\rho_1(b,a_1)^2/r_1 \leq \rho_1(b,e) \leq \rho_1(b,a_1)$ and $\rho_1(b,a_1)^2/r_1 \asymp \rho_1(b,e)$,
	unless $\rho_1(b,a_1) \succeq r_1$ which as we have seen suffices to give \eqref{eq:cat-sr4}.
	Since $\rho_1(b,c) < \rho_1(b,e)$, quasisymmetry and (3) for $h$ applied to $b,e$ gives:
	\begin{equation}\label{eq:cat-sr5}\begin{split}
		\left(\frac{\rho_2(h(b),a_2)}{r_2}\right)^2
		& \leq \frac{\rho_2(h(b),h(c))}{r_2}
		\preceq \frac{\rho_2(h(b),h(e))}{r_2}
		\asymp \left(\frac{\rho_1(b,e)}{r_1}\right)^{\lambda_{a_1}}
		\\ & \asymp \left(\frac{\rho_1(b,a_1)}{r_1}\right)^{2\lambda_{a_1}}
		\asymp \left(\frac{\rho_2(h(b),a_2)}{r_2}\right)^2.
	\end{split}\end{equation}
	Thus all $\preceq$ are $\asymp$, and $\rho_2(h(b),h(c))\asymp \rho_2(h(b),h(e))$, which by quasisymmetry gives $\rho_1(b,c)\asymp \rho_1(b,e)$.
	Then \eqref{eq:cat-sr5} implies \eqref{eq:cat-sr4} and we are done.
\end{proof}

\subsection{Transient sets}
We consider again a space $X$ hyperbolic relative to a collection of subsets $\cP$.

Let $\mu,R$ be constants and $\alpha$ a geodesic in $X$. Denote by $deep_{\mu,R}(\alpha)$ the set of points $p$ of $\alpha$ that belong to some subgeodesic $[x,y]$ of $\alpha$ with endpoints in $N_{\mu}(P)$ for some $P\in\cP$ and so that $d(p,x),d(p,y)>R$. Denote by $\trans_{\mu,R}(\alpha)=\alpha\backslash deep_{\mu,R}(\alpha)$ the set of \emph{transient points} \cite[Definition~8.9]{Hr-relqconv}, \cite[Definition~3.9]{Si-metrrh}. 

We collect the following properties of transient and deep sets from \cite{Hr-relqconv,Si-metrrh,Si-projrelhyp}, using as well results in \cite{DSp-05-asymp-cones} which are however not phrased in terms of these notions.

\begin{lemma}
Suppose a space $X$ is hyperbolic relative to $\cP$. For each $P\in\cP$ denote by $\pi_P:X\to P$ a coarse closest point projection, i.e. a function so that $d(x,\pi_P(x))\leq d(x,P)+1$.
\label{transprop}
 There exist $\mu,R,D,t,C$ with the following properties: For all $x,y,z \in X$,
\begin{enumerate}
 \item $[$Relative Rips condition$]$ We have
$$\trans_{\mu,R}[x,y]\subseteq N_{D}(\trans_{\mu,R}[x,z]\cup \trans_{\mu,R}[z,y]).$$
 \item $deep_{\mu,R}[x,y]$ is contained in a disjoint union of subgeodesics of $[x,y]$ each contained in $N_{t\mu}(P)$ for some $P\in\cP$ and called a \emph{deep component along} $P$.
	 Moreover, any geodesic connecting points of $N_L(P),L \geq 1$ is contained in
	 $N_{tL}(P)$.
 \item The endpoints of the deep component of $[x,y]$ along $P\in\cP$ (if it exists) are $C$-close to $\pi_P(x),\pi_P(y)$.
 \item If for some $P\in\cP$ we have $d(\pi_P(x),\pi_P(y))>C$, then $[x,y]$ has a deep component along $P$ of length at least $d(\pi_P(x),\pi_P(y))-C$.
\end{enumerate}
\end{lemma}
\begin{proof}
	(1) For groups, this follows combining \cite[Proposition 8.13]{Hr-relqconv} and \cite[Theorem 3.26]{Osin-rel-hyp}, while for general spaces this is \cite[Proposition 4.6(3)]{Si-metrrh}, see also \cite[Corollary 4.27]{DSp-05-asymp-cones}.

	(2) This follows from uniform quasiconvexity of peripheral sets \cite[Lemma 4.15]{DSp-05-asymp-cones} and the fact that neighbourhoods of distinct peripheral sets have bounded intersection \cite[Theorem 4.1$(\alpha_1)$]{DSp-05-asymp-cones}, see \cite[Proposition 5.7(1)]{Si-metrrh}.
	
	(3) By \cite[Lemma 1.13(1)]{Si-projrelhyp} a geodesic from $x$ to $y$ that has points near $P$ must travel from $x$, enter $N_\mu(P)$ near $\pi_P(x)$ and leave $N_\mu(P)$ near $\pi_P(y)$.

	(4) By \cite[Lemma 1.15(1)]{Si-projrelhyp} if $d(\pi_P(x),\pi_P(y))$ is large enough, any geodesic from $x$ to $y$ must pass close to $\pi_P(x)$ and $\pi_P(y)$, which implies the existence of a suitable deep component.
\end{proof}

\begin{conv}
 When we refer to transient points without explicit mention of the constants we always imply a choice of constants as in the lemma.
\end{conv}

\section{Subexponential distortion and (relative) hyperbolicity}
	\label{sec:subexp-distort}

 The main goal of this section is to prove the following.

	\begin{theorem}
\label{sederh}
 Let $f:X\to Y$ be a subexponentially distorted embedding, and suppose that $Y$ is hyperbolic relative to a collection of subsets $\cP$. Then there exists $\rho_0\geq 0$ so that for all $\rho\geq \rho_0$ we have that $X$ is hyperbolic relative to $\{f^{-1}(N_\rho(P))\}_{P\in\cP}$.
\end{theorem}

As we will see at the end of this section, this statement about relative hyperbolicity of metric spaces implies the group case stated in Corollary~\ref{cor:subexpdistort-subgroup-rel-hyp}.
It also gives one part of Theorem~\ref{thm:truncated-hyp-poly}.
\begin{corollary}
	\label{cor:polyembed-truncated-rel-hyp}
	Let $Y$ be a truncated real hyperbolic space, that is, $\mathbb H^n$ with a family of open disjoint horoballs removed, endowed with its path metric.
	If a group $G$ admits a subexponentially 
	distorted embedding into $Y$ then $G$ is hyperbolic relative to virtually nilpotent subgroups.
\end{corollary}
\begin{proof}
	Recall that \cite[Theorem 1.5]{Dr-relhyp-qiinv} says that if a group $G$ has a Cayley graph which is, as a space, hyperbolic relative to some collection of subsets $\cP$ then $G$ is also hyperbolic relative to a collection of subgroups, each of which is contained in a uniform neighbourhood of one of the subsets in $\cP$.  (Recall that Dru\c tu's terminology of being `asymptotically tree-graded' is equivalent to being relatively hyperbolic.)
	
	Thus by Theorem \ref{sederh} we get that $G$ is hyperbolic relative to subgroups each of which admits a uniformly proper embedding into some (polynomially growing) horospheres of the truncated space.
In particular, each peripheral subgroup has (at most) polynomial growth, and is therefore virtually nilpotent \cite{Gr-polgrowth, vdDW-ascones}.  
\end{proof}

\subsection{Uniformly proper maps into hyperbolic spaces}

Towards proving Theorem \ref{sederh}, we begin with two essentially standard facts.
\begin{lemma}\label{lem:unifprop-coarsesurj}
	Suppose $Y'$ is a geodesic metric space.
	Then any uniformly proper, coarsely Lipschitz and coarsely surjective
	map $f:Y \ra Y'$ is a quasi-isometry, quantitatively.
\end{lemma}
A proof of this may be found, for example, in~\cite[Lemma 8.29]{DrKap-18-GGT-book}.
%
%
%
This has the following well-known consequence
(cf.\ \cite[Proposition 2.6]{Kap-01-Comb-thm-qcvx}), which we include as a warm-up to Proposition \ref{bndhauss}.

\begin{proposition}\label{prop:subexp-into-hyp-is-qi}
	Suppose $Y$ and $Y'$ are geodesic metric spaces, with $Y'$ Gromov hyperbolic.
	If $f:Y \ra Y'$ is a subexponentially distorted map, then $f$ is
	a quasi-isometric embedding.
\end{proposition}

\begin{proof}
	It suffices to show that $f$ restricted to any geodesic $[p,q] \subset Y$
	is a quasi-isometric embedding, quantitatively.
	By the definition of subexponential distortion (Definition~\ref{def:poly-subexp-distortion}),
	$f$ is coarsely Lipschitz with some constant $C_1$,
	and there exists $\tau:[0,\infty) \ra \R$ with
	$\lim_{t \ra \infty} \tau(t)/\log(t)= \infty$ 
	and for any $x,y \in X$,
	\[
		\tau(d_Y(x,y)) \leq d_{Y'}(f(x),f(y)).
	\]
	
	We now follow the proof of \cite[Theorem III.H.1.7]{BH-99-Metric-spaces}
	to show that the Hausdorff distance between $f([p,q])$ and a 
	geodesic $[f(p),f(q)]\subset Y'$ is at most $R$, for some $R$ depending only
	on $C_1, \tau$ and the hyperbolicity constant $\delta_{Y'}$ of $Y'$.
	Let $c:[0,b]\ra Y'$, where $b=d_Y(p,q)$,
	be the composition of the unit-speed parametrisation
	of $[p,q] \subset Y$ and $f$.
	For each interval $I$ in $[0,1], [1,2], \ldots, [\lfloor b \rfloor, b]$,
	replace $c|_I$ with a geodesic segment in $Y'$ with the same endpoints,
	to find a continuous rectifiable path $c':[0,b] \ra Y'$.
	This path $c'$ satisfies $d_{Y'}(c(t),c'(t)) \leq C_2$ for all $t$,
	and is $\tau'$-uniformly proper for $\tau'(t)=\tau(t)-C_2$,
	where $C_2$ is a constant which only depends on $C_1$ and $\tau$.
	Moreover, for any $s,t \in [0,b]$, we have the length bound
	$l(c'|_{[s,t]}) \leq C_2 |s-t|$.
	
	Suppose $x_0 \in [f(p),f(q)]$ is a point which maximises $d_{Y'}(x, \mathrm{im}(c'))$ for
	$x \in [f(p),f(q)]$, and let $D$ be this maximal distance.
	Let $y \in [f(p),x_0] \subset [f(p),f(q)]$ be the point at distance $2D$ from $x$ 
	(if $d_{Y'}(f(p),x_0)<2D$, set $y=f(p)$), and choose $z \in [x_0, f(q)]$ likewise.
	Now choose points $y' = c'(s), z'=c'(t)$ with $d_{Y'}(y,y'),d_{Y'}(z,z') \leq D$.
	
	Let $\gamma$ be the path obtained concatenating a geodesic $[y,y']$, the path $c'|_{[s,t]}$ 
	and a geodesic $[z',z]$.  This path has length
	$l(\gamma) \leq 2D+l(c'|_{[s,t]}) \leq 2D+C_2 |s-t|$.
	Observe that 
	\[ \tau'(|s-t|) \leq d_{Y'}(y',z') \leq d_{Y'}(y',y)+d_{Y'}(y,z)+d_{Y'}(z,z') \leq 6D, \]
	while 
	\[ C_2 |s-t| \geq l(c'|_{[s,t]}) \geq d_{Y'}(y',z') \geq 2D. \]
	By \cite[Proposition III.H.1.6]{BH-99-Metric-spaces} we have
	\begin{align*}
		D-1 \leq \del_{Y'} | \log_2 l(\gamma) |.
	\end{align*}
	So for large $D$ (and hence large $|s-t|$) these bounds combine to get
	\begin{align*}
			\tau'(|s-t|) 
			\leq 6D \leq 7(D-1)
			\leq 7\del_{Y'}  \log_2 (2C_2 |s-t| )
	\end{align*}
	which implies that $D$ is bounded by some $D_0$ 
	depending only on $\tau'$, $C_2$ and $\del_{Y'}$.
	
	To bound the Hausdorff distance between $f([p,q])$ and $[f(p),f(q)]$,
	it remains to show that $c'$ lies in a $D_1$-neighbourhood of $[f(p),f(q)]$
	for some $D_1$.
	Suppose that some $x=c'(u)$ satisfies $d_{Y'}(x,[f(p),f(q)])>D_0$.
	Every point in $[f(p),f(q)]$ is within a distance of $D_0$ from $c'$,
	so by connectedness there exists a point $w \in [f(p),f(q)]$ and values $s<u$ and $t>u$ 
	so that $d_{Y'}(w,c'(s)), d_{Y'}(w,c'(t)) \leq D_0$.
	
	Therefore, $d_{Y'}(c'(s), c'(t)) \leq 2D_0$, and so by uniform properness
	$|s-t| \leq C_3$ for some $C_3$ depending only on $D_0$ and $\tau'$.
	Because $s<u<t$, we have $|s-u| \leq C_3$ also, thus
	the distance between $x$ and $[f(p),f(q)]$ is at most
	\[
		d_{Y'}(x,w) \leq d_{Y'}(x,c'(s)) + d_{Y'}(c'(s),w)
		\leq C_2|u-s|+D_0 \leq C_2 C_3 +D_0;
	\]	
	setting $D_1 = C_2C_3+D_0$, we are done.
	
	Since $c'$ and $[f(p),f(q)]$ are at Hausdorff distance at most $D_1$,
	we can adjust the values $c'$ by at most $D_1$ to find $c'':[0,d_Y(p,q)] \ra [f(p),f(q)]$.
	This map is uniformly proper and coarsely Lipschitz.
	Since it maps endpoints to endpoints it is coarsely surjective,
	so by Lemma~\ref{lem:unifprop-coarsesurj} it is a 
	quasi-isometry, quantitatively.
	Because $c''$ is within finite distance of $c$, $f|_{[x,y]}$ is also a quasi-isometry,
	quantitatively.
\end{proof}

\subsection{Transient points are close to image paths}
\label{core}
We fix the notation of Theorem \ref{sederh} from now on. In particular, $Y$ denotes a fixed relatively hyperbolic space with fixed constants for transient/deep points as in Lemma~\ref{transprop}.

The basic observation that we will use to prove Theorem \ref{sederh} is contained in the following lemma.

\begin{lemma}
\label{sublin}
For every $L$ there exists a constant $K$ with the following property.
 Let $\gamma$ be a geodesic in $Y$ and let $\alpha:([a,b]\subseteq \R)\to Y$ be a coarsely $L$-Lipschitz path with the same endpoints $x,y$ as $\gamma$. Then every transient point $p\in\gamma$ satisfies $d(p,\alpha)\leq K\log(|b-a|+1)+K$. 
\end{lemma}

The proof is a minor variation on that of \cite[Proposition III.H.1.6]{BH-99-Metric-spaces}: $\alpha$ splits into two subpaths with domain half the size of that of $\alpha$ and one can consider the triangle with endpoints the ``middle'' point of $\alpha$ and the endpoints of $\gamma$, apply the relative Rips condition, and then repeat. The details are left to the reader.

In this subsection we improve the simple bound one obtains by applying Lemma \ref{sublin} to images in $Y$ of quasi-geodesics in $X$. The idea is to reapply the lemma to subpaths.

\begin{proposition}
\label{bndhauss}
 Let $\gamma$ be a geodesic in $Y$ and let $\alpha$ be the image of 
 a subexponentially distorted coarsely Lipschitz map $[0,T]\ra Y$
 with the same endpoints as $\gamma$. 
 Then every transient point $p\in\gamma$ satisfies $d(p,\alpha)\leq E$,
 where $E$ is a constant depending only on the distortion of $\alpha$ and on $Y$.
\end{proposition}

\begin{proof}
Let $p\in\gamma$ be ``the worst'' transient point, meaning one maximising the distance from $\alpha$. Set $\eta=d(p,\alpha)$. We can assume $\eta\geq 1$. Also, we assume that $\alpha$ is minimal in the sense that for any transient point $p'$ on a geodesic with endpoints on $\alpha$ there is a point $q'$ on $\alpha$ with $d(p',q')\leq\eta+1$. Let $x^\pm$ be the endpoints of $\gamma$.

Because $\alpha$ is subexponentially distorted, there exists $\tau:[0,\infty)\to\R$ so that for all $s,t \in [0,T]$ we have 
$\tau(|s-t|)\leq d(\alpha(s),\alpha(t)) \leq L|s-t|+L$,
where $\tau(t)/ \log t\to\infty$ as $t\to\infty$.

We will consider two cases (the second one of which is irrelevant to the special case
of Theorem \ref{sederh} for $Y$ hyperbolic). The constant $M\geq 1$ is a large enough
constant, depending on $Y,\cP,\tau,L$ only, to allow us enough ``space'' for the arguments in Case 2 to go through.

\medskip
\emph{Case 1.} There exist transient points $p^{\pm}\in\gamma$ respectively preceding and following $p$ so that $d(p,p^{\pm})\in [\min\{d(x^{\pm},p),10\eta\}, M\eta]$.
\medskip

In this case we choose $q^{\pm}$ on $\alpha$ at distance at most $\eta+1$ from $p^\pm$ (choose $q^\pm=p^\pm$ if $p^\pm=x^\pm$). Consider the path $\beta$ obtained concatenating in the suitable order geodesics from $p^\pm$ to $q^\pm$ and a subpath of $\alpha$ between $q^-$ and $q^+$.

Suppose $q^{\pm} = \alpha(t^\pm)$.
Assuming $\gamma$ is at least $20\eta$-long, at least one of $p^-,p^+$ is at distance
at least $10\eta$ from $p$.  This combines with the fact that $\alpha$ is coarsely
$L$-Lipschitz to give that
\[
	L|t^+-t^-|+L \geq d(q^-,q^+) \geq 10\eta-2(\eta+1) \geq 6\eta,
\]
so for large $\eta$ we have $|t^+-t^-| \geq (5/L)\eta$.
On the other hand, certainly $d(p,\beta) \geq \eta$, and we can apply
Lemma~\ref{sublin} to $\beta$ to see that
\begin{equation}
	\label{eq:sublin-distort}
	\eta \leq d(p,\beta) \leq K\log(|t^+-t^-|+2\eta+2+1)+K
	\leq K \log(|t^+-t^-|)+K',
\end{equation}
for constants $K,K'$.
Using the distortion bound on $\alpha$ and the estimate $d(p^-,p^+) \leq 2M\eta$, we have
\[
	\tau(|t^+-t^-|)\leq d(q^-,q^+) \leq d(p^-,p^+)+2\eta+2 \leq 6M\eta.
\]
Our hypothesis on $\tau$ gives that $\log|t^+-t^-| \leq o_{\eta\ra\infty}(1) \tau(|t^+-t^-|)$, where
$o_{\eta\ra\infty}(1)$ is a function that goes to zero as $\eta$ (and hence
$|t^+-t^-|$) go to infinity.
Combined with \eqref{eq:sublin-distort} this gives
\[
	\eta \leq K o_{\eta\ra\infty}(1) \tau(|t^+-t^-|)+K'
	\leq 6MK \eta o_{\eta\ra\infty}(1)+K',
\]
which in turn implies that $\eta$ is bounded depending only on required data.

\medskip
\emph{Case 2.} $\gamma$ contains a deep component along, say, $P\in\cP$ on one side of $p$ with one endpoint at distance $>M\eta$ from $p$ and the other one at distance $<10\eta$ from $p$.
\medskip

Suppose that the deep component in the statement lies before $p$.
We would like to reduce to Case 1 by suitably changing $\alpha$ and $\gamma$. As $\pi_P$ is coarsely Lipschitz there exists $x'\in \alpha$ so that $d(p,\pi_P(x'))\in [M\eta/4,M\eta/2]$ (in fact, the values of $\pi_P$ along $\alpha$ vary between $\pi_P(x^-)$ and $\pi_P(x^+)$, and the latter coarsely coincides with $\pi_P(p)$, which in turn coarsely coincides with $p$). Choose $q\in\alpha$ so that $d(p,q)\leq \eta+1$. Any geodesic $[x',q]$ contains a transient point $p^-$ close to $\pi_P(x')$, and $p^-$ is within distance $\eta+1$ from some $x\in \alpha$ (recall that we chose a minimal $\alpha$).

Also, any geodesic $\gamma'$ from $p^-$ to the final point of $\gamma$ contains some transient point $p'$ close to $p$. Notice that the hypotheses of Case 1 are now satisfied by $p$ and $\gamma'$ ``on one side'' of $p'$.
Up to possibly reapplying the argument above to ``the other side'' of $\gamma'$ to obtain a new geodesic $\gamma''$ and point $p''$ on it, we end up in the situation of Case 1.
\end{proof}

\subsection{Morse lemma for transient sets}

As seen in Proposition~\ref{prop:subexp-into-hyp-is-qi}, the image of a
subexponentially distorting map from an interval into a hyperbolic space lies within a
bounded Hausdorff distance of a geodesic with the same endpoints.
The following proposition shows an analogous property for
the transient sets of geodesics. We will not need this to prove Theorem \ref{sederh} (but we will use it later). The proof builds on Proposition~\ref{bndhauss}.
Recall Definition~\ref{def:coarse-respect-periph} for the notion of a map coarsely respecting peripherals.
\begin{proposition}\label{prop:transhauss}
	Suppose $f:X \ra X'$ is a subexponentially distorting map 
	between relatively hyperbolic spaces $(X,\cP), (X',\cP')$ that coarsely respects peripherals, and fixed deep/transient set constants as in Lemma~\ref{transprop}.
	Then for any $x,y \in X$, we have that
	$f(\trans(x,y))$ and $\trans(f(x),f(y)) \subset X'$
	are at Hausdorff distance $\leq C$, where $C$ is independent of $x,y \in X$, and depends only on the
	data of relative hyperbolicity, subexponential distortion, coarsely respecting peripherals and deep/transient constants.
\end{proposition}

We require the following estimate on projections.
\begin{lemma}\label{lem:subexp-proj}
 Suppose we are in the situation of Proposition~\ref{prop:transhauss}. 
	Then there exists $D$ with the following property: for each $x\in X$ and $P\in\mathcal P$, we have $d(f(\pi_P(x)),\pi_{P'}(f(x)))\leq D$, where $P'\in\cP'$ is given by Definition~\ref{def:coarse-respect-periph}(1).
\end{lemma}
\begin{proof}
	Denote $y=\pi_P(x)$, and consider a geodesic $[x,y]$. Then $f\circ[x,y]$ connects $f(x)$ to $f(y)$, and the latter lies in the $C$-neighbourhood of $P'$, where $C$ is as in the definition of $f$ coarsely respecting peripherals. For later purposes, we increase $C$ so that Lemma \ref{transprop} applies with constant $C$. Suppose by contradiction that we have $d(f(y),\pi_{P'}(f(x)))> D$, where $D$ is a sufficiently large constant to be determined below (we note the order of choice of constants is $C, E, C', D$). Whenever $D>C$, any geodesic $[f(x),f(y)]$ contains a transient point $C$-close to $\pi_{P'}(f(x))$ by Lemma \ref{transprop}-(3)-(4). In particular, $f\circ[x,y]$ contains a point $f(z)$ which is $(E+C)$-close to $\pi_{P'}(f(x))$, where $E$ is as in Proposition~\ref{bndhauss}. Note that $f(z)$ is $(D-E-C)$-far from $f(y)$.

	Now, since $f$ coarsely respects peripherals we have $f^{-1}(N_{E+C}(P'))\subseteq N_{C'}(P)$, for some $C'$. Therefore, we have that $z$ lies in the neighbourhood $N_{C'}(P)$, and since $f$ is uniformly proper, for $D$ sufficiently large we have $d(z,y)\geq C'+3$. But then we see that $d(x,\pi_P(z))<d(x,y)-1$, which contradicts $y=\pi_P(x)$.
\end{proof}
\begin{proof}[Proof of Proposition~\ref{prop:transhauss}]
	The argument is a variation on the proof of
	Proposition~\ref{prop:subexp-into-hyp-is-qi}.
	
	Given $x,y \in X$, fix geodesics $[x,y]\subset X$ and $[f(x),f(y)]\subset X'$.
	Let $\alp = f([x,y])$; then $\alp$ is a subexponentially-distorted, coarse
	Lipschitz `quasi'-geodesic.
	By Proposition~\ref{bndhauss}
	every transient point of $[f(x),f(y)]$ lies at distance of at most $E$ from
	$\alp$.

	Let $\Pi:\trans(f(x),f(y))\ra \alp$ be a map which displaces each point by at 
	most $E$.
	Suppose $[p^-,p^+]$ is a (maximal) connected interval of $\trans(f(x),f(y))$.
	We claim that $[p^-,p^+]$ and $\alp[\Pi(p^-),\Pi(p^+)]$
	are at bounded Hausdorff distance from each other,
	where $\alp[\Pi(p^-),\Pi(p^+)]$ denotes the subpath of $\alp$ from $\Pi(p^-)$ to $\Pi(p^+)$.
	Indeed, the same argument as in the proof of 
	Proposition~\ref{prop:subexp-into-hyp-is-qi} shows that as we follow $z$ along
	$[p^-,p^+]$, the point $\Pi(z)$ has to follow $\alp$ along $\alp[\Pi(p^-),\Pi(p^+)]$ 
	only omitting subpaths of controlled size.

	What happens as we cross large deep components of $[f(x),f(y)]$?
	Suppose $[q^-,q^+]$ is a connected component of $\deep(f(x),f(y))$ of length
	$\geq N$, where $N$ is sufficiently large as determined below.
	Let $x^-,x^+ \in [x,y]$ be points with $f(x^\pm)=\Pi(q^\pm)$.
	Since $f$ is coarsely Lipschitz, we have $d(x^-,x^+)\geq c N - C_1$;
	here all constants $c,C_1,C_2,\ldots$ are independent of $N$.

	By the definition of deep components, $q^-$ and $q^+$ lie $C_2$-close to a peripheral
	set $P'\in\cP'$.
	Since $f$ coarsely respects peripherals, $x^-,x^+$ lie $C_3$-close to a peripheral
	set $P \in \cP$, and as $d(x^-,x^+)$ is large $P$ and $P'$ uniquely correspond to each other in Definition~\ref{def:coarse-respect-periph}.
	By Lemma~\ref{lem:subexp-proj} and Lemma~\ref{transprop}(3) we have that $d(f(\pi_P(x)),q^-)$ and $d(f(\pi_P(y)),q^+)$ are at most $C_4$,
	so the uniform properness of $f$ gives $d(\pi_P(x),x^-), d(\pi_P(y),x^+) \leq C_5$.
	Therefore $d(\pi_P(x),\pi_P(y)) \geq cN-C_1-2C_5$, so for sufficiently large $N$ we can apply
	Lemma~\ref{transprop}(4) to $[x^-,x^+]$ to find that $[x^-,x^+]$ is within
	Hausdorff distance $C_6$ of a (the) deep component of $[x,y]$ along $P$.

	By Lemma~\ref{transprop}(2), $[x^-,x^+]$ lies in the $C_7$-neighbourhood of $P$.
	Again, since $f$ coarsely respects peripherals, $f([x^-,x^+])$ lies in the
	$C_8$-neigh\-bour\-hood of $P'$.
	
	So, follow $z$ along $[f(x),f(y)]$, and consider $f^{-1}\circ \Pi(z) \in [x,y]$:
	as we go along components of $\trans(f(x),f(y))$ we have a coarsely Lipschitz map,
	while if we jump over a large deep component of $[f(x),f(y)]$ we jump over a
	corresponding deep component of $[x,y]$.  This shows that $f(\trans(x,y))$ is
	contained in a bounded neighbourhood of $\trans(f(x),f(y))$.
	
	It remains to check the converse inclusion.
	Suppose $[w^-,w^+]$ is a deep component in $[x,y]$ along $P \in \cP$ of length $\geq N'$,
	where $N'$ is sufficiently large to be determined, and $P$ has not already been
	considered above, i.e. $P$ does not correspond to a $P' \in \cP'$ with a long deep
	component in $[f(x),f(y)]$.
	Then since $[x,y]$ is connected, $[w^-,w^+]$ coarsely lies in
	the image of $f^{-1}\circ\Pi$, in particular in the image of $\trans(f(x),f(y))$.
	So we have points $y^\pm \in [f(x),f(y)]$ with $d(y^\pm, f(w^\pm))\leq E$.
	By Lemma~\ref{transprop}(3) we have $d(w^-,\pi_P(x))\leq C$ and $d(w^+,\pi_P(y)) \leq C$.
	Thus by Lemma~\ref{lem:subexp-proj} we have 
	\begin{align*}
		d(\pi_{P'}(f(x)),\pi_{P'}(f(y))) & \approx_{2D} d(f(\pi_P(x)),f(\pi_P(y)))
		\\ & \geq \tau(d(w^-,w^+)-2C) \geq \tau(N'-2C),
	\end{align*}
	where $\tau$ is the uniform properness function of $f$.
	Provided $N'$ is large enough that $\tau(N'-2C)-2D \geq N$, we have a contradiction.
	
	Thus, $\trans(f(x),f(y))$ is contained in a bounded neighbourhood of the image of
	$[x,y]$ under $f$, omitting large deep components of $[x,y]$, which in turn lies within
	finite Hausdorff distance from $f(\trans(x,y))$.	
\end{proof}

\subsection{Proof of Theorem \ref{sederh}}

We are now ready to prove Theorem \ref{sederh}.

It will be convenient to use the following equivalent characterisation of relative hyperbolicity. We note that it might be possible to use the cusped space characterisation of relative hyperbolicity directly, but this runs into issues with controlling geodesics in $\Bow(X,\cP)$ in terms of paths in $X$ \emph{without} knowing yet that $X$ is relatively hyperbolic.
Recall that being `asymptotically tree-graded' with respect to a collection of subsets $\cP$ is equivalent to being hyperbolic relative to $\cP$ (in the sense of Definition~\ref{def:rel-hyp}).

\begin{theorem}[{\cite[Theorem 1.7]{Dr-relhyp-qiinv}}]\label{thm:rel-hyp-quasi-triangles}
 A geodesic metric space $X$ is hyperbolic relative to a collection of subsets $\cP$ if and only if 

$\bullet$ for each $K$ there exists $B$ so that $diam(N_K(P)\cap N_K(Q))\leq B$ for each distinct $P,Q\in \cP$,

$\bullet$ there exists $\epsilon\in (0,1/2)$ and $M\geq 0$ so that for each $P\in\cP$ and $x,y\in X$ with $x,y\in N_{\epsilon d(x,y)}(P)$ any geodesic from $x$ to $y$ intersects $N_M(P)$,

$\bullet$ for every $c\geq 0$ there exist $\sigma,\delta$ so that every triangle with continuous $(1,c)$-quasi-geodesic edges $\gamma_1,\gamma_2,\gamma_3$ satisfies either

\ $(1)$ there exists a ball of radius $\delta$ intersecting all sides of the triangle, or

\ $(2)$ there exists $P\in\cP$ with $N_\sigma (P)$ intersecting all sides of the triangle and the entrance (resp.\ exit) points $x_i$ (resp.\ $y_i$) of the sides $\gamma_i$ in (from) $N_\sigma(P)$ satisfy $d(y_i,x_{i+1})\leq \delta$.
\end{theorem}

\begin{proof}[Proof of Theorem \ref{sederh}]
We have a subexponentially distorting embedding $f:X\to Y$ into a space $Y$ hyperbolic relative to a collection of subsets $\cP$.
We fix constants as in Lemma \ref{transprop} for $Y$.

	We will use a variation of the third property in Theorem~\ref{thm:rel-hyp-quasi-triangles} for $Y$.

\begin{lemma}\label{lem:relhyp-thin-or-deep}
 For a geodesic space $Y$ hyperbolic relative to a collection of subsets $\cP$, there exists $\delta_0$ so that the following holds. For any geodesic triangle in $Y$ either
 
\ $(1')$ there exists a ball of radius $\delta_0$ intersecting the transient sets of all sides, or 
 
\ $(2')$ there exists $P\in\cP$ so that all sides have deep components along $P$.
\end{lemma}

\begin{proof}
	Consider a geodesic triangle with sides $\gamma_i=[p_i,p_{i+1}]$, $i=0,1,2$ modulo $3$. By Lemma~\ref{transprop}(1), every transient point along $\gamma_0$ is $D$-close to a transient point on either $\gamma_1$ or $\gamma_2$, and hence one of the following holds. 

	(\emph{a}) There are two transient points $x_1,x_2$ on $\gamma_0$ within distance $10C$ and so that $d(x_i,\trans(\gamma_i))\leq D$ (here $C$ is given by Lemma~\ref{transprop}). In this case, $(1')$ holds.

	(\emph{b}) There is a deep component $[x_1,x_2]$ of $\gamma_0$, say along $P$, of length at least $10C$ so that $d(x_i,\trans(\gamma_i))\leq D$. By Lemma~\ref{transprop}(3) $x_1$ is $C$-close to $\pi_P(p_1)$ and $x_2$ is $C$-close to $\pi_P(p_0)$. Since $\pi_P(p_2)$ cannot be $C$-close to both $\pi_P(p_0)$ and $\pi_P(p_1)$, without loss of generality, $\gamma_1$ has a deep component along $P$ as well, by Lemma~\ref{transprop}(4). If $\gamma_2$ does as well, we are in case $(2')$, so suppose that it does not. Then $\pi_P(p_2)$ is $C$-close to $\pi_P(p_0)$. In particular, $x_2$ is close to $\pi_P(p_0)$, which is close to $\pi_P(p_2)$, which is close to a transient point of $\gamma_1$. So, $x_2$ is within uniformly bounded distance from the transient sets of both $\gamma_1$ and $\gamma_2$, and hence $(1')$ holds.
\end{proof}

	For $X$, we first prove the third item in the equivalent definition of relative hyperbolicity given by Theorem~\ref{thm:rel-hyp-quasi-triangles}, since this is the part of the proof where $\rho_0$ gets chosen.

	Suppose we have $p_0,p_1,p_2 \in X$ and $\gamma_i$ is a $(1,c)$-quasi-geodesic connecting $p_i$ to $p_{i+1}$ (modulo $3$) for $i=0,1,2$, and consider a geodesic triangle in $Y$ with vertices $f(p_1),f(p_2),f(p_3)$. If such triangle is as in Lemma~\ref{lem:relhyp-thin-or-deep}(1$'$) then by the lemma and Proposition \ref{bndhauss}, we see that each $f(\gamma_i)$ intersects a specified ball of radius $\delta_0+E$, which easily implies that a ball of some radius $\delta$ meets each $\gamma_i$.

	If the triangle is as in Lemma~\ref{lem:relhyp-thin-or-deep}(2$'$), we show that $(2)$ holds for the $\gamma_i$. Since all sides have deep components along $P$, and the endpoints of a deep component are transient points, in view of Proposition \ref{bndhauss} we see that each $\gamma_i$ intersects $Q=f^{-1}(N_\rho(P))$, for any $\rho\geq \rho_0$ where $\rho_0$ is chosen suitably large. For later purposes, we assume $\rho_0\geq C+E$ with $C$ from Lemma~\ref{transprop} and $E$ from Proposition~\ref{bndhauss}. Let now $x_i,y_i$ be the entrance/exit points with respect to $Q$ as in the statement of $(2)$ (setting $\sigma=0$). We will prove that $d(f(x_i),\pi_P(f(p_i)))$ can be uniformly bounded by, say, $K$, and similarly for $y_i$. This easily implies that $d(y_i,x_{i+1})$ is uniformly bounded, as required. Suppose that $d(f(x_i),\pi_P(f(p_i)))$ is large and take a subpath $\gamma'_i$ of $\gamma_i$ with final point $x_i$. Increasing $K$, we can assume $d(\pi_P(f(x_i)),\pi_P(f(p_i)))\geq C$ (if $f(x_i)$ is far from $\pi_P(f(p_i))$ then so is its projection, since $f(x_i)$ is close to $P$). In particular, $\pi_P(f(p_i))$ is $(C+E)$-close to some point in $f(\gamma'_i)$, necessarily different from $f(x_i)$ if $K>C+E$, because $\pi_P(f(p_i))$ is $C$-close to a transient point on $[f(p_i),f(x_i)]$ and such transient point is $E$-close to $f(\gamma'_i)$ by Proposition \ref{bndhauss}. This contradicts the fact that $x_i$ is the entrance point in $f^{-1}(N_\rho(P))$. This concludes the proof of the third property of Theorem~\ref{thm:rel-hyp-quasi-triangles}.

	The first item of Theorem~\ref{thm:rel-hyp-quasi-triangles} easily follows from the fact that the analogous property holds in $Y$ and the uniform properness of $f$.

	Let us now show the second property in Theorem~\ref{thm:rel-hyp-quasi-triangles}. First of all, we claim that there exists $K$ with the following property. Let $Q=f^{-1}(N_{\rho}(P))$ for some $P\in\cP$ and let $x\in X$. Then for any $\overline{x}\in Q$ that satisfies $d(x,\overline{x})\leq d(x,Q)+1$, we have $d(f(\overline{x}),\pi_P(f(x)))\leq K$.

	Let us first argue that the claim implies the second property, for $\epsilon=1/3$ and a suitable $M$. In fact, let $Q=f^{-1}(N_{\rho}(P))$ for some $P\in\cP$ and let $x,y$ satisfy $x,y\in N_{d(x,y)/3}(P)$. We can also assume that $d(x,y)$ is large enough so that 
	\[ d(f(\overline{x}), f(\overline{y})) \geq \tau(d(\overline{x},\overline{y})) \geq \tau\big(\tfrac{1}{3}d(x,y)-2\big) > 2K+C,\] 
	where $\tau$ is the uniform properness function of $f$. In this case by the claim we have $d(\pi_P(f(x)),\pi_P(f(y))) > C$, which in view of Proposition \ref{bndhauss} and Lemma \ref{transprop} implies that $f([x,y])$ passes $(C+E)$-close to $P$, whence we get the second property.

	Let us now prove the claim. If the claim did not hold, then any geodesic from $f(x)$ to $f(\overline{x})$ would have a transient point $C$-close to $\pi_P(f(x))$, and hence $f([x,\overline{x}])$ would contain a point $f(p)$ that is $(C+E)$-close to $P$. Thus $d(f(p),f(\overline{x})) \geq K-(C+E)$ so as $f$ is coarse Lipschitz $d(p,\overline{x})\geq \phi(K)$, for some diverging (linear) function $\phi(K)$. Also, $d(p,Q)\leq C'$, for some $C'$. But then
$$d(x,Q)\leq d(x,p)+C'\leq d(x,\overline{x}) +(C'-\phi(K)),$$
which for $K$ large enough contradicts the choice of $\overline{x}$.
\end{proof}

\subsection{Truncated hyperbolic spaces}
We now show our main application of Theorem \ref{sederh}. Namely, assuming Theorem~\ref{thm:boundary-to-inside}, we now prove Theorem~\ref{thm:truncated-hyp-poly}: a finitely generated group is hyperbolic relative to virtually nilpotent groups if and only if it admits a polynomially distorted embedding into some truncated real hyperbolic space, if and only if it admits a subexponentially distorted embedding into some such space.

\begin{proof}[Proof of Theorem~\ref{thm:truncated-hyp-poly}]
	If $G$ is a finitely generated group and $G$ admits a subexponentially distorted embedding into some truncated real hyperbolic space, then by Corollary~\ref{cor:polyembed-truncated-rel-hyp} $G$ is hyperbolic relative to some collection of virtually nilpotent subgroups.

	Conversely, suppose $G$ is hyperbolic relative to virtually nilpotent subgroups $H_1,\ldots,H_n$.  It remains to construct a polynomially distorted embedding into some truncated real hyperbolic space.
	Let $X$ be the Cayley graph of $G$ and $\cP = \{gH_i\}$ the collection of left cosets of each $H_i$,
	let $\Bow(X,\cP)$ be the corresponding cusped space with basepoint $o$, and endow $Z = \bdry \Bow(X,\cP)$ with a visual metric $\tilde{\rho}$ with parameter $2\epsilon>0$.
	
	By \cite[Proposition 4.5]{MS-12-relhypplane} $(Z,\tilde{\rho})$ is a doubling metric space (see the citation for details),
	so by Assouad's embedding theorem~\cite{Ass-83-snowflake} $(Z, \tilde{\rho}^{1/2})$ admits a bi-Lipschitz embedding into some $\R^N$, and hence (as it is bounded) into $\Sph^N$.

	So we may assume that there is a bi-Lipschitz embedding $h:(Z,\rho)\ra \Sph^N$, where $\rho=\tilde{\rho}^{1/2}$ is a visual metric on $Z = \bdry \Bow(X,\cP)$ with parameter $\epsilon$.
	Let $\cB=\{(a_P,r_P=e^{-\epsilon d(o,P)})\}_{P\in\cP}$ be the usual shadow decoration on $Z$.
	By Lemma~\ref{lem:separationhoroballs}, $\rho(a_P,a_Q) \succeq \sqrt{r_P r_Q}$ uniformly for all $P\neq Q \in \cP$.

	Recall that $\Sph^N = \bdry \HH^{N+1}$, and the Euclidean metric $\rho'$ on $\Sph^N$ is a visual metric with parameter $1$ and any fixed basepoint $o'$.
	For $P\in \cP$, let $a'_P=h(a_P)$.
	Since $h$ is bi-Lipschitz, there exists $C$ so that $\rho'(a'_P,a'_Q) \succeq_C \sqrt{r_P r_Q}$ for all $P\neq Q\in\cP$. 
	Apply Lemma~\ref{lem:separationhoroballs} to find $t_0$ so that the horoballs
	\[
		H_P = \left\{ y\in \HH^{N+1}: \beta_{a'_P}(y,o') \leq -t_0+\log r_P \right\}, P\in \cP
	\]
	are separated in $\HH^{N+1}$ by a distance greater than $1$.
	Let $X' = \HH^{N+1}\setminus \bigcup_{P\in\cP} int(H_P)$ (viewing $\cP$ now as an abstract index set), and consider $\{H_P\}_{P\in\cP}$ as a uniformly admissible set of horoballs for $X'$ (see Proposition~\ref{prop:real-horoballs-admissible}), so then $\Bow(X',\{H_P\})=\HH^{N+1}$.
	We give $\Sph^N=\bdry \Bow(X',\{H_P\})$ the standard shadow structure $\cB'=\{(a'_P,r'_P)\}_{P \in \cP}$ by setting $r'_P = e^{-d_{\HH^{N+1}}(o',H_P)}$.

	By the definition of $H_P$, $r'_P\asymp r_P$ uniformly.
	So if $\rho(a_P,b) \leq r_P$ for some $P \in \cP, b \in Z$, then
	\[
		\rho'(a'_P,h(b)) \asymp \rho(a_P,b) \leq r_P \asymp r'_P.
	\]
	Likewise if $\rho(a_P,b) \geq r$ then $\rho'(a'_P,h(b)) \succeq r'_P$, and $h$ maps $\{a_P\}$ bijectively to $\{a'_P\}$, so properties (1) and (2) of Definition~\ref{def:shadow-qs} hold.  Property (3) holds easily with each $\lambda_{a_P}=1$ as $h$ is bi-Lipschitz and each $r'_P \asymp r_P$.  Thus $h$ is shadow-respecting as a map from $\bdry (X,\cP) \to \bdry \Bow(X',\{H_P\})$.

	We have $(Z,\rho,\cB) = \bdry(X, \cP)$ is visually complete by Lemma~\ref{prop:bowditch-visually-complete}, as is $\Bow(X',\{H_P\})=\HH^{N+1}$.  
	Therefore Theorem~\ref{thm:boundary-to-inside} gives a polynomially distorted embedding from $X$ (the Cayley graph of $G$) to $X'$ (a truncated hyperbolic space).
	Theorem~\ref{thm:truncated-hyp-poly} is proved.
\end{proof}

\subsection{Relative quasiconvexity}
Finally, we apply the metric relative hyperbolicity result of Theorem~\ref{sederh} in the case of groups.
\begin{varthm}[Corollary~\ref{cor:subexpdistort-subgroup-rel-hyp}]
	Let $G$ be hyperbolic relative to $\{H_i\}$.
	If $H$ is a subexponentially distorted subgroup of $G$ then $H$ is relatively quasiconvex in $G$, and so $H$ is hyperbolic relative to subgroups $\{K_j\}$ each of which is the intersection of a conjugate (in $G$) of some $H_i$ with $H$.
\end{varthm}

\begin{proof}
	We will use the characterisation $(QC-3)$ of relative quasiconvexity from \cite{Hr-relqconv}, and more specifically we will show that there is a constant $C$ so that, given a geodesic $\gamma$ in $\Bow(G,\{H_i\})$ with endpoints in $H$, we have that $\gamma\cap G$ is contained in the $C$-neighbourhood of $H$.

	By Theorem \ref{thm:truncated-hyp-poly}, $H$ is (as a metric space) relatively hyperbolic, and it is readily seen that the inclusion map coarsely respects peripherals. Therefore, Proposition \ref{prop:transhauss} implies that, given a geodesic in $G$ with endpoints in $H$, its transient set lies in a controlled neighbourhood of $H$. Now, $\trans(x,y)$ coarsely coincides (in $\Bow(G,\{H_i\})$)
	with the intersection with $G$ of a geodesic $[x,y]$ from $x$ to $y$ in $\Bow(G,\{H_i\})$ by \cite[Propositions 7.9 and 8.13]{Hr-relqconv}. Therefore, we see that the aforementioned characterisation applies.

	Having shown $H$ is relatively quasiconvex, the corollary follows from \cite[Theorem 9.1]{Hr-relqconv}.
\end{proof}

\section{Extending quasi-isometries to the boundary}\label{sec:extend-qi-to-bdry}
In this section, we prove Theorem~\ref{thm:inside-to-boundary}. We fix relatively hyperbolic spaces $(X,\cP)$ and $(X',\cP')$, and a map
$f:X \ra X'$ which is a snowflake on peripherals.  As hypothesised in Theorem~\ref{thm:inside-to-boundary}, throughout this section we assume that each $P\in\cP$ is unbounded.

\subsection{Extending to cusped spaces}

Our first goal is to extend $f$ to a quasi-isometric embedding
$f_{\Bow}:\Bow(X,\cP) \ra \Bow(X',\cP')$ between their cusped spaces.

First we observe that a rough $\lambda$-snowflake map between graphs (i.e.\ a map with \eqref{eq:rough-snowflake} holding) extends to
a rough $\lambda$-similarity between the corresponding horoballs (i.e.\ a map with \eqref{eq:rough-similarity} holding).
\begin{lemma}\label{lem:snowflake-into-horoball}
	Suppose $P$ and $P'$ are coarsely connected metric spaces and $f:P \ra P'$ satisfies
	\begin{equation}\label{eq:rough-snowflake}
		\frac{1}{C} d_P(x,y)^\lambda -C \leq d_{P'}(f(x),f(y)) 
		\leq C d_P(x,y)^\lambda+C.
	\end{equation}
	Consider admissible horoballs $\cH(P)$ and $\cH(P')$ with corresponding geodesic rays $\{\gamma_x\}_{x\in P\cap\cH(P)}$ and $\{\gamma'_x\}_{x\in P'\cap\cH(P')}$.
	Define a map $\bar f: \cH(P) \ra \cH(P')$ by setting $\bar f(p) = \gamma'_{x'}(\lambda m)$ where 
	for each $p \in \cH(P)$ we choose $x \in P\cap\cH(P)$ and $m \geq 0$ with $d_{\cH(P)}(p,\gamma_{x}(m)) \leq C$, and $x'\in P'$ with $d_{P'}(f(x),x') \leq C$.
	Then $\bar f$ is a quasi-isometry, in fact it is a rough $\lambda$-similarity: there exists $C'$ so that for all $p,q \in \cH(P)$,
	\begin{equation}
		\label{eq:rough-similarity}
		d_{\cH(P')}(\bar f(p), \bar f(q)) \approx_{C'} \lambda\; d_{\cH(P)}(p,q)
	\end{equation}
\end{lemma}
\begin{proof}
	By hyperbolicity and Definition~\ref{def:admissible-horoballs}(2) the map $\bar f$ is coarsely onto.
	For $p,q \in \cH(P)$, suppose we have  $d_{\cH(P)}(p,\gamma_x(m))\leq C$ and $d_{\cH(P)}(q,\gamma_y(n)) \leq C$ with $\bar f(p)=\gamma'_{x'}(\lambda m), \bar f(q)=\gamma'_{y'}(\lambda n)$, $d_{P'}(f(x),x')\leq C$ and $d_{P'}(f(y),y')\leq C$.
	Then by Definition~\ref{def:admissible-horoballs}(4) we see that
	\begin{align*}
		 d_{\cH(P')}(\bar f(p),\bar f(q))
		 & \approx
		 2 \log \big(d_{P'}(x',y') e^{-\max \{\lambda m, \lambda n \}} +1\big) +
			\big|  \lambda m - \lambda n \big|
		\\ & \approx 2 \log \big(d_{P}(x,y)^\lambda e^{- \lambda \max \{m,n\}}
			+1\big) + \lambda | m - n |
		\\ & \approx \lambda \Big( 2 \log \big(d_{P}(x,y)e^{- \max \{m,n\}}
			+1\big) + | m - n | \Big)
		\\ & \approx \lambda\; d_{\cH(P)}\big(p,q\big).\qedhere
	\end{align*}
\end{proof}

We now build $f_\Bow$.
\begin{construction}\label{constr:bowditch-extension}
Since $f:X \ra X'$ coarsely respects peripherals, there exists $C >0$ so that
for all $P \in \cP$ there exists $P' \in \cP'$ with $f(P) \subset N_C(P')$.
(Since $f(P)$ is unbounded, the choice of $P'$ is unique.)
So after adjusting $f$ by a bounded distance, we have a rough snowflake map
from $P$ to $P'$.
We then use Lemma~\ref{lem:snowflake-into-horoball} to define 
a map $f_\Bow$ from $\cH(P) \subset \Bow(X,\cP)$ to $\cH(P') \subset \Bow(X',\cP')$.
These extensions combine to define $f_\Bow:\Bow(X,\cP) \ra \Bow(X',\cP')$.
\end{construction}

\begin{proposition}\label{prop:fbow:qi}
	Let $(X,\cP), (X',\cP'), f:X \to X'$ be given by Theorem~\ref{thm:inside-to-boundary}.
	Then $f_\Bow$ as given by Construction~\ref{constr:bowditch-extension} is a quasi-isometric embedding.
\end{proposition}

Towards proving the proposition, we start with the following.

\begin{lemma}\label{lem:fbow:unifprop}
	Under the assumptions of Proposition~\ref{prop:fbow:qi}, $f_\Bow$ is uniformly proper.
\end{lemma}
\begin{proof}
	We know that $f$ is uniformly proper, 
	and $f_\Bow$ restricted to the horoballs is uniformly proper,
	so $f_\Bow$ is also uniformly proper. 
	
	In slightly more detail, suppose $d_{\Bow(X',\cP')}(f_\Bow(x), f_\Bow(y)) \leq C$
	for some $x,y \in X$ and large $C$.
	
	From the definition of $f_\Bow$, if $f_\Bow$ maps points into a horoball
	$\cH(P'), P' \in \cP'$, then these points come from a horoball $\cH(P), P \in \cP$.
The choice of $P$ is unique: 
	for any such $P$, we have that $f(P)$ is unbounded and contained in a neighbourhood of
	$P'$, so $f^{-1}(N_{C'}(P'))$ is also unbounded for suitable $C'$.
	Since $f$ coarsely respects peripherals, $f^{-1}(N_{C'}(P'))$ is contained in a neighbourhood of $P$, and therefore the choice of $P$ is unique.
	
	If $f_\Bow(x)$ is in a horoball $P' \in \cP'$, 
	and $d(f_\Bow(x), \partial P') \geq 2C$,
	then $f_\Bow(x)$ and $f_\Bow(y)$ both lie within the image of a 
	single horoball $\cH(P), P \in \cP$.
	Since $f_\Bow$ restricted to $\cH(P)$ is uniformly proper by 
	Lemma~\ref{lem:snowflake-into-horoball}, $d_{\Bow(X,\cP)}(x,y)$ is bounded
	by a uniform constant depending on $C$.
	
	So we are reduced to the case that $f_\Bow(x), f_\Bow(y)$ lie within $2C$ of 
	$X' \subset \Bow(X',\cP')$, and so $x$ and $y$ both lie within $C'$ of $X \subset \Bow(X,\cP)$,
	for suitable $C'$.
	Since $X$ is uniformly distorted inside $\Bow(X,\cP)$ (Remark~\ref{rmk:bow-inclusion-unif-prop})
	we have that $f_\Bow$ restricted to $X$ is uniformly proper, and so
	$d_{\Bow(X,\cP)}(x,y)$ is uniformly bounded depending on $C$.
\end{proof}

\begin{proof}[Proof of Proposition~\ref{prop:fbow:qi}]
	Since $f$ is coarsely Lipschitz and each horoball extension is coarsely Lipschitz,
	$f_\Bow$ is coarsely Lipschitz.
	
	We begin by considering points $x,y \in X$, with the goal of showing that $f_\Bow$ is a quasi-isometry (with uniform constants) when restricted to any geodesic $[x,y] \subset \Bow(X,\cP)$.
	We will use that $\trans(x,y)$ coarsely coincides (in $\Bow(X,\cP)$)
	with the geodesic $[x,y]$ from $x$ to $y$ in $\Bow(X,\cP)$ intersected with $X \subset \Bow(X,\cP)$; for groups this follows from \cite[Propositions 7.9 and 8.13]{Hr-relqconv} and for general space from \cite[Proposition 4.9]{Si-metrrh}.
	
	By Proposition~\ref{prop:transhauss} the image of $\trans(x,y)$ under $f$ coarsely coincides with $\trans(f(x),f(y)) \subset X'$.
	So $f_\Bow(\trans(x,y))$ is contained in a neighbourhood of a geodesic 
	$[f(x),f(y)] \subset \Bow(X',\cP')$, again by the relation between geodesics in $X'$ and $\Bow(X',\cP')$ explained above.
	
	On the other hand, by Construction~\ref{constr:bowditch-extension} any maximal subgeodesic $\alpha$ of $[x,y]$ contained in a horoball
	is mapped to a quasi-geodesic in the corresponding horoball in $\Bow(X',\cP')$
	with endpoints close to $[f(x),f(y)]$. In particular $f(\alpha)$ lies close to $[f(x),f(y)]$.
	
	Combined, we get that $f([x,y])$ lies in a uniformly bounded neighbourhood of $[f(x),f(y)]$. 
	We can then think of a perturbation of $f_\Bow|_{[x,y]}$ as a map $g:[x,y] \ra [f(x),f(y)]$.
	This map $g$ is coarsely surjective, 
	since the endpoints of $[f(x),f(y)]$ are in the image and $f$ is coarsely Lipschitz.
	Because $f_\Bow$ is uniformly proper (Lemma~\ref{lem:fbow:unifprop}), $g$ is also
	uniformly proper.  So by Lemma~\ref{lem:unifprop-coarsesurj} $g$ is a quasi-isometric
	embedding, quantitatively.
	
	Thus $f_\Bow$ is a quasi-isometry when restricted to any $[x,y] \subset \Bow(X,\cP)$ 
	for arbitrary $x,y \in X \subset \Bow(X,\cP)$, with uniform constants.
	
	If we have, say, $x \in X$ and $y \in \Bow(X,\cP) \setminus X$, then $y \in \cH(P)$ for some
	$P \in \cP$.
	The geodesic $[x,y]$ passes within bounded distance of $\pi_P(x) \in P$ by Lemma~\ref{transprop}(3); let 
	$z \in \cH(P)$ be the last point of $[x,y] \cap X$.
	Again, $d_X(\pi_P(x),z)$ is uniformly bounded.
	
	Choose some $z' \in P\cap \cH(P)$ so that $d(y,\gamma_{z'})\leq C$ where $\{\gamma_x\}_{x\in P\cap\cH(P)}$ are the geodesic rays in the horoball (Definition~\ref{def:admissible-horoballs}).
	Let $\gam$ be the geodesic segment $[x,y] \cap \cH(P)$ from $z$ to $y$,
	and let $w \in \cH(P)$ be a quasi-centre of $z,y,a_P$.
	By the construction of the horoballs, up to bounded distances $\gamma$ follows $\gamma_{z}$ from $z$ to $w$, then follows $\gamma_{z'}$ from $w$ back to $y$.

	Thus $y$ lies close to a geodesic from $z$ to $z'$.
	By the hyperbolicity of $\Bow(X,\cP)$, there is a constant $C_1$ so that if $d(y,z)\geq C_1$ then both $z$ and $y$ are within bounded distance of a geodesic $[x,z']\subset \Bow(X,\cP)$. 
	As both $x$ and $z'$ are in $X \subset \Bow(X,\cP)$, the first case considered
	above shows that $f_\Bow$ restricted to $[x,z']$ is a quasi-isometry,
	and hence so is $f_\Bow$ restricted to $[x,y]$.
	
	If $y$ is within $C_1$ from $z$, then $f_\Bow$ restricted to $[x,y]$ is within bounded distance of $f_\Bow$ restricted to $[x,z]$ which again is a quasi-isometry by the argument of the first case.

	Similar arguments apply to show that $f_\Bow|_{[x,y]}$ is a quasi-isometry when
	$x$ and $y$ lie in different horoballs in $\Bow(X,\cP)$.  If $x$ and $y$
	lie in the same horoball, the conclusion follows from 
	Lemma~\ref{lem:snowflake-into-horoball}.
\end{proof}	

\subsection{Extending to boundaries}
Having extended $f$ to a quasi-isometric embedding $f_\Bow:\Bow(X,\cP)\ra \Bow(X',\cP')$,
we now show that the boundary map is controlled.

Recall that $Z=\bdry (X, \cP)$ comes with a shadow decoration:
$\rho$ is a visual metric on $\bdry(X, \cP)$, with visual parameter
$\eps$ and constant $C$, and
$\cB = \{ (a_P, r_P)\}_{P \in \cP}$ is the collection of balls with
for each $P \in \cP$, $a_P \in \bdry(X,\cP)$ is the boundary point corresponding to some (any) geodesic ray $\{\gamma_x\}_{x\in P\cap \cH(P)}$ in the horoball $\cH(P) \subset \Bow(X,\cP)$,
and we set $r_P = e^{-\eps d_{\Bow(X,\cP)}(o,P)} \in (0,\infty)$.

Likewise, $\bdry (X', \cP')$ comes with the analogous data 
$\rho', \eps', C', \cB' = \{ (a_{P'},r_{P'})\}_{P' \in \cP'}$.
It makes sense to ask whether $\bdry f_{\Bow}$
is shadow-respecting with respect to these shadow decorations,
and that is what we show in the following proposition.
\begin{proposition}\label{prop:bdry-map-control}
	Let $(X,\cP), (X',\cP'), f:X \to X'$ be given by Theorem~\ref{thm:inside-to-boundary}, and $f_\Bow$ by Construction~\ref{constr:bowditch-extension}.
	Then the induced map $\bdry f_\Bow:\bdry(X,\cP) \ra \bdry(X',\cP')$
	is a shadow-respecting quasisymmetric embedding, quantitatively.
	If $f$ is a quasi-isometry, then $\bdry f_\Bow$ asymptotically
	$\frac{\eps'}{\eps}$-snowflakes, where as above $\epsilon, \epsilon'$ are the visual parameters of the metrics on $\bdry(X,\cP), \bdry(X',\cP')$.
\end{proposition}
\begin{proof}
	Since $f_\Bow$ is a quasi-isometric embedding (Proposition~\ref{prop:fbow:qi})
	by \cite[Theorem 6.5(2)]{BS-00-gro-hyp-embed}, 
	$f_\Bow$ extends to a (power) quasisymmetric embedding
	$\bdry f_\Bow:\bdry(X,\cP) \ra \bdry(X',\cP')$.
	We may assume $f_\Bow(o)=o'$, by allowing our constants to depend on $d'(f_\Bow(o),o')$; changing the basepoint in the image space introduces controlled multiplicative errors in $\rho'$ and $r_{P'}$.
	We prove the properties of Definition~\ref{def:shadow-qs} in turn.

	\vspace{0.5ex}{\noindent \emph{Claim (1):}}

	First note that for any horoball $\cH(P) \subset \Bow(X,\cP)$,
	$f_\Bow$ is defined to send geodesic rays $\{\gamma_{x}\}$ in $\cH(P)$
	to geodesic rays $\{\gamma'_{x'}\}$ in a horoball $\cH(P') \subset \Bow(X',\cP')$,
	therefore $\bdry f_\Bow$ maps $a_P$ to $a_{P'}$.

	For simplicity, we write $d=d_{\Bow(X,\cP)}$, and likewise $d'=d_{\Bow(X',\cP')}$.

	Assume we have $(a_P,r_P)\in \cB$ and $b\in \bdry(X,\cP)=\bdry \Bow(X,\cP)$, with $\rho(a_P,b)\leq r_P$.
	Let $p$ be a quasi-centre of $o,a_P,b$ and let $x$ be a closest point of $P$ to $o$. 
	Let $b' = \bdry f_{\Bow}(b)$, $p'=f_\Bow(p)$ and $x'=f_\Bow(x)=f(x)$.

	Let $P'\in \cP'$ be the peripheral space corresponding to $f(P)$, with $a_{P'}\in \bdry(X,\cP)$ the limit point of $\cH(P')\subset \Bow(X',\cP')$.
	Since $f(P)$ is in a bounded neighbourhood of $P'$, $x'=f(x)$ is close to $P'$ and by the Morse lemma also close to $[o',a_{P'})$, thus as $\cH(P')$ is quasiconvex, $x'$ is bounded distance from a closest point in $P'$ to $o'$.

	Since $\rho(a_P,b) \leq r_P$, we have $d(o,p)\approx (a_P|b) \gtrsim d(o,x)$.
	Since $p,x$ are both close to a geodesic $[o,a_P)$, 
	which is mapped by $f_\Bow$ within bounded distance of $[o',a_{P'})$,
	we have $d'(o',p') \gtrsim d'(o',x')$.
	By Lemma~\ref{splittosplit} $p'$ is approximately a quasi-centre for $o',a_{P'},b'$.
	Thus $(a_{P'}|b')\approx d'(o',p') \gtrsim d'(o',x')\approx d'(o',P')$ so $\rho'(a_{P'},b') \preceq r_{P'}$ as desired.

	\vspace{0.5ex}{\noindent \emph{Claim (2):}} 

	Suppose $\rho(a_P,b) \geq r_P$ for some $(a_P,r_P) \in \cB, b\in \bdry(X,\cP)$.  
	Then the argument of Claim (1) applies, except this time the quasi-centre $p$ of $o,a_P,b$ has $d(o,p) \approx (a_P|b) \lesssim d(o,x)$ and so $(a_{P'}|b')\approx d'(o',p') \lesssim d'(o',x') \approx d'(o',P')$ and $\rho'(a_{P'},b') \succeq r_{P'}$.

	Before completing the proof of (2), we state a useful lemma.
	\begin{lemma}\label{lem:bdry-dist-to-parabolic}
		Suppose $X$ is a geodesic metric space hyperbolic relative to $\cP$ and $\bdry(X,\cP)$ has a metric with visual parameter $\epsilon$.
		Suppose we have $P \in \cP$ and $b \in \bdry \Bow(X,\cP)$ with $b\neq a_P$.
		Let $p$ denote a quasi-centre of $o,a_P,b$, 
		and let $x \in P \subset \Bow(X,\cP)$ denote the closest point in $P$ to $o$.

		There exists $C>0$ so that if $\rho(b,a_P)\leq r_P/C$,
		we have $d(o,p) \geq d(o,x)$, any geodesic $[o,b)$ must intersect $P$,
		and
		\[
			\rho(a_P,b) \asymp r_P \cdot e^{-\eps d(x,q)/2},
		\]
		where $q$ is the last point in $P$ on $[o,b)$.
	\end{lemma}
	\begin{proof}
		Since $\log \rho(b,a_P) \approx -\eps (b|a_P)$ and
		$\log r_P \approx -\eps d(o,P)=-\eps d(o,x)$,
		for any $C'$ any $C$ sufficiently large will have $(b|a_P) \geq d(o,x)+C'$.
		Since $(b|a_P) \approx d(o,p)$, choosing $C'$ large enough ensures that $d(o,p) \geq d(o,x)$.
		
		Denote by $\beta_{a_P}(\cdot, o)$ the Busemann function corresponding to $a_P$ and $o$. 
		As $P$ is approximately a level set of $\beta_{a_P}(\cdot, o)$,
		if $q$ is the last point in 
		a geodesic $[o,b)$ to meet $P$ then 
		we have $\beta_{a_P}(q,o) \approx \beta_{a_P}(x,o) \approx -d(o,x)$.
		On the other hand, using that $p$ is the quasi-centre of $o,a_P,b$,
		we have $\beta_{a_P}(q,o) \approx d(q,p)-d(p,x)-d(x,o)$.

		Combined, these show that $d(p,q) \approx d(p,x)$.
		Since $p$ approximately lies on a geodesic from $x$ to $q$, we have
		$d(p,x)\approx d(q,x)/2$.

		Applying this to $\log \rho(b,a_P) \approx -\eps (b|a_P)$,
		we have $(b|a_P) \approx d(o,p) \approx d(o,x)+d(x,p) \approx d(o,x)+d(q,x)/2$,
		and the conclusion follows since $-\eps d(o,x)\approx \log r_P$.
	\end{proof}

	Suppose now we have $b \in \bdry \Bow(X,\cP)$, and $P' \in \cP'$ with $a_{P'} \notin h(Z)$.
	Let $p'$ be a quasi-centre of $o',a_{P'}, b'=\bdry f_\Bow(b)$, and let $x'$ be a closest point in $P'$ to $o'$.
	Then, in view of Lemma~\ref{lem:bdry-dist-to-parabolic}, we have either $\rho'(a_{P'},b') \geq r_{P'}/C$, proving (2), or
	$\rho'(a_{P'},b') \asymp r_{P'} e^{-\epsilon' d'(x',q')/2}$,
	where $q'$ is the last point in a geodesic $[o',b')$ to meet $P'$.

	Suppose that the latter holds.
	Since $[o',b')$ and $f_\Bow([o,b))$ are at bounded Hausdorff distance,  there exist $x, q \in \Bow(X,\cP)$ with $d'(f_\Bow(x),x')$ and $d'(f_\Bow(q),q')$ uniformly bounded.
	Since $f_\Bow$ sends points far in horoballs in $\Bow(X,\cP)$ far into horoballs in $\Bow(X,\cP)$, and $x',q' \in X'\subset \Bow(X',\cP')$, we may assume that $x$ and $q$ lie in $X\subset\Bow(X,\cP)$ by moving them a bounded distance.

	Thus for a suitable $C'$, by Lemma~\ref{lem:bdry-dist-to-parabolic}, 
	\[ \diam \left( f(X) \cap N_{C'}(P')\right) \gtrsim d'(x',q') \approx \frac{2}{\epsilon'}\log \left( \frac{r_{P'}}{\rho'(a_{P'},b')}\right). \]

	If, by contradiction,  $r_{P'}/\rho'(a_{P'},b')>A$ for some large $A$ then by the uniform properness and coarsely respecting peripherals properties of $f_\Bow$ and $f$ there exists $P \in \cP$ with $x,q \in f^{-1}(N_{C'}(P')) \subset N_{C''}(P)$.
	In the definition of $f_\Bow$, we extend $f$ by mapping $\cH(P)$ to some $\cH(\hat{P}'), \hat{P}' \in \cP'$.  But as $x',q'$ are in a bounded neighbourhood of $\hat{P}'$ also, by hyperbolicity we have $\hat{P}' =P'$ assuming $A$ is large enough.
	This is a contradiction, so $\rho'(a_{P'},b')\geq r_{P'}/A$.
	
	\vspace{0.5ex}{\noindent \emph{Claim (3):}} 

	Suppose, as hypothesised, that we have $(a,r)=(a_P,r_P) \in \cB$
	and points $b,c \in \bdry \Bow(X,\cP)$ so that
	\[
		\left( \frac{\rho(a,b)}{r} \right) \rho(a,b) \leq 
		\rho(b,c) \leq \rho(a,b) \leq r.
	\]
	This implies that
	\begin{equation}\label{eq:two-point-control-inside}
		2(a|b)-d(o,P) \gtrsim (b|c) \gtrsim (a|b) \gtrsim d(o,P).
	\end{equation}
	Let $x \in \Bow(X,\cP)$ be the closest point to $o$ of $P$.
	Let $x_b, x_c$ be the last points of $[o,b),[o,c)$ to meet $P$.
	Let $p_{a,b}$ be a quasi-centre of $o,a,b$, and likewise for other pairs of points.

	By Lemma~\ref{lem:bdry-dist-to-parabolic} and \eqref{eq:two-point-control-inside} we have that
	\[ d(x_b,o) \gtrsim d(p_{b,c},o) \gtrsim d(p_{a,b},o) \gtrsim d(o,P).\]
	
	Let $a',b',c'$, etc., be the images under $f_\Bow$  of the corresponding points in $\Bow(X',\cP')\cup \bdry\Bow(X',\cP')$.
	By Lemma~\ref{splittosplit} we have that $p'_{a,b}$, the image of $p_{a,b}$, is within bounded distance of the quasi-centre
	of $o',a',b'$, and likewise for the other quasi-centres.

	Moreover, $x$ is characterised up to bounded distance as the point of intersection of $P$
	and $[o,a)$.
	Since $f_\Bow$ maps $[o,a)$ within bounded Hausdorff distance of $[o',a')$, and 
	also $P$ close to $P'$, we have that $x'=f_\Bow(x)$ is within bounded distance the intersection
	of $P'$ and $[o',a')$, and thus of the closest point to $o'$ in $P'$.
	Similarly, since $x_b$ is the final point of intersection of $P$ and $[o,b)$,
	it is sent to roughly the final point of intersection of $P'$ and $[o',b')$.

	Putting this together, we have two cases for the configuration of a tree approximating  $o,a,b,c$: in case (i) we have that $[o,c)$ breaks off from $[o,a)$ after $p_{a,b}$,
	while in case (ii) we have that $[o,c)$ breaks off from $[o,b)$ after $p_{a,b}$ but
	before $x_b$. See Figure~\ref{fig:extension1}.
\begin{figure}[h]
   \def\svgwidth{0.7\textwidth}
   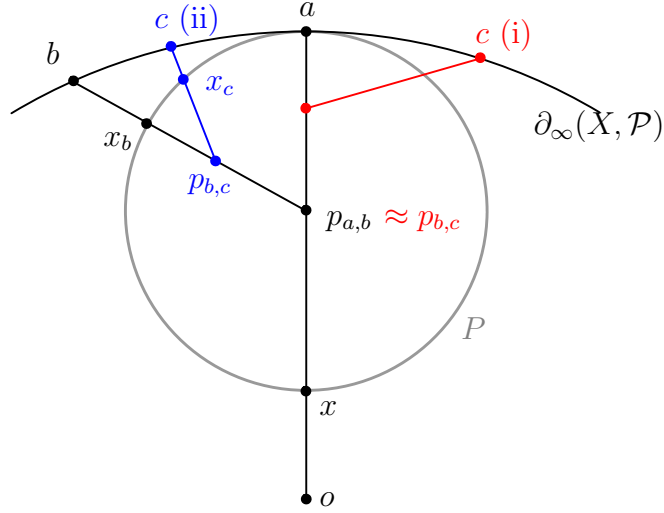
	\caption{Configurations of $a,b,c$}
	\label{fig:extension1}
\end{figure}

	In case (i) we have $p_{a,b} \approx p_{b,c}$, and so $(a|b) \approx (b|c)$,
	thus 
	\[
		\frac{1}{-\eps} \log \frac{\rho(b,c)}{r} \approx (b|c)-d(o,P) 
			\approx \frac12 d(x,x_b),
	\]
	while since this configuration is preserved in the image we also have
	\[
		\frac{1}{-\eps'} \log \frac{\rho'(b',c')}{r'} \approx (b'|c')-d(o',P')
			\approx \frac12 d'(x',x'_b).
	\]
	Now since $f$ is a snowflake on peripherals, there exists $\lambda_P$ so that
	$f$ restricted to $P$ is a rough $\lambda_P$-snowflake map.
	This means that 
	\[
		d_{X'}(x',x'_b) \qi d_X(x,x_b)^{\lambda_P}.
	\]
	(Recall that the notation $A\qi_{\lambda,\mu} B$ means $A/\lambda-\mu\leq B\leq \lambda A+\mu$.) 
	However, distances in $\Bow(X,\cP),\Bow(X',\cP')$ for points in the boundary of horoballs
	are approximately double the logarithm of the corresponding distances in $X,X'$.
	Thus
	\[
		d'(x',x'_b) \approx \lambda_P d(x,x_b).
	\]
	Combining these results, we have
	\begin{equation}\label{eq:two-point-control-outside}
		\frac{\rho'(b',c')}{r'} \asymp \left( \frac{\rho(b,c)}{r} \right)^{\lambda_P \eps'/\eps}.
	\end{equation}

	Case (ii) is similar, but here
	\[
		\frac{1}{-\eps} \log \frac{\rho(b,c)}{r} \approx (b|c)-d(o,P) 
			\approx d(x,x_b)-d(x_b,x_c)/2,
	\]
	and
	\[
		\frac{1}{-\eps'} \log \frac{\rho'(b',c')}{r'} \approx (b'|c')-d(o',P') 
			\approx d'(x',x'_b)-d'(x'_b,x'_c)/2.
	\]
	Since distances between pairs of points in a horoball are again roughly scaled by $\lambda_P$,
	we have $d'(x',x'_b)-d'(x'_b,x'_c)/2 \approx \lambda_P (d(x,x_b)-d(x_b,x_c)/2)$,
	and so \eqref{eq:two-point-control-outside} again holds.

	Finally, observe that if $\lambda_P=1$, e.g. if $f$ is a quasi-isometry, then by \eqref{eq:two-point-control-outside}
	$\bdry f_\Bow$ asymptotically $\frac{\eps'}{\eps}$-snowflakes.
\end{proof}
Before completing the proof of Theorem~\ref{thm:inside-to-boundary}, we note some properties of Definition~\ref{def:snowflake-on-peripherals}.
\begin{proposition}
	\label{prop:category-snowflake-periph}
	If $(X,\cP), (X',\cP'), (X'',\cP'')$ are relatively hyperbolic spaces, and $f:(X,\cP)\to (X',\cP')$ and $g:(X',\cP')\to (X'',\cP'')$ are snowflakes on peripherals then so is the composition $g \circ f: (X,\cP)\to (X'',\cP'')$.
	The identity map on a relatively hyperbolic space is a snowflake on peripherals.  Therefore we have a category where objects are relatively hyperbolic spaces, and morphisms are snowflake on peripheral maps, where we consider $f:X \to X'$ and $g:X\to X'$ equivalent if we have  $\sup_{x\in X} d_{X'}(f(x),g(x)) < \infty$.

	Moreover, if $f:(X,\cP)\to (X',\cP')$ is a snowflake on peripherals and is coarsely onto, then a (any) coarse inverse $g:(X',\cP')\to(X,\cP)$ is also a snowflake on peripherals.  
\end{proposition}
\begin{proof}
	The properties of being polynomially distorted, coarsely respecting peripherals and being a snowflake on peripherals (which implies the former two properties), are all easily seen to be preserved under composition.

	The only subtle point is the coarse inverse: if $f:(X,\cP)\to (X',\cP')$ is a snowflake on peripheral map that is coarsely surjective, then by Lemma~\ref{lem:unifprop-coarsesurj} $f$ is a quasi-isometry, and so any coarse inverse will also satisfy the snowflake on peripheral conditions (with $\lambda =1$).
\end{proof}
We now conclude our study of the properties of $\bdry (\cdot)_\Bow$.
\begin{proof}[Proof of Theorem~\ref{thm:inside-to-boundary}]
	The extension of $f$ to a quasi-isometry $f_\Bow$ is given by Proposition~\ref{prop:fbow:qi}.  That $\bdry f_\Bow$ is a shadow-respecting quasisymmetric embedding, and asymptotically $\frac{\epsilon'}{\epsilon}$-snowflakes if $f$ is a quasi-isometry, follows from Proposition~\ref{prop:bdry-map-control}.

	It remains to check the functorial properties.
	For $\id_X :(X,\cP) \to (X,\cP)$, the extension into horoballs satisfies $(\id_X)_\Bow = \id_\Bow$ and so $\bdry (\id_X)_\Bow = \id_{\bdry (X,\cP)}$.
	Suppose $f:(X,\cP)\to (X',\cP')$ and $g:(X',\cP')\to (X'',\cP'')$ are snowflake on peripheral maps.
	The construction of the extensions into cusped spaces implies that $g_\Bow \circ f_\Bow$ and $(g\circ f)_\Bow$ are equal on $X \subset \Bow(X,\cP)$.
	Given a peripheral set $P \in \cP$, coarsely respecting peripherals implies that there is $P' \in \cP'$ with $f(P)$ coarsely in $P'$, and $P'' \in \cP''$ with $g(P')$ coarsely in $P''$.
	Suppose $f$ is a rough $\lambda$-snowflake on $P$, and $g$ is a rough $\lambda'$-snowflake on $P'$; consequently $g \circ f$ is a rough $\lambda\lambda'$-snowflake on $P$.
	For $x \in P$, let $\gamma_x:[0,\infty)\to \Bow(X,\cP)$, $\gamma_x':[0,\infty)\to\Bow(X',\cP')$ and $\gamma_x'':[0,\infty)\to\Bow(X'',\cP'')$ be unit-speed geodesic rays from $x$ to $a_P$, from $f(x)$ to $a_{P'}$, and from $g(f(x))$ to $a_{P''}$ respectively.
	Construction~\ref{constr:bowditch-extension} and hyperbolicity imply that for any $t\geq 0$, we have uniform bounds on $d_{\Bow(X',\cP')}(f_\Bow(\gamma_x(t)),\gamma_x'(\lambda t))$, on $d_{\Bow(X'',\cP'')}(g_\Bow(\gamma_x'(t)),\gamma_x''(\lambda' t))$, and on $d_{\Bow(X'',\cP'')}((g\circ f)_\Bow(\gamma_x(t)),\gamma_x''(\lambda\lambda' t))$.  Consequently, $(g\circ f)_\Bow(x)$ and $g_\Bow\circ f_\Bow(x)$ are at bounded distance from each other.

	Being at bounded distance, the boundary extensions $\bdry (g\circ f)_\Bow$ and $(\bdry g_\Bow)\circ (\bdry f_\Bow)$ are equal.
\end{proof}

\section{Characterisation of boundary extensions}\label{sec:extend-bdry-to-qi}

In this section we prove Theorem \ref{thm:boundary-to-inside}. Namely, given a shadow-respecting $\eta$-quasisymmetric embedding $h: \bdry(X,\cP) \ra \bdry(X',\cP')$ between boundaries of relatively hyperbolic spaces, we  construct a polynomially distorted embedding $\hat h:X \ra X'$ inducing $h$, and we show various further properties of this correspondence.

We fix the notation of Theorem \ref{thm:boundary-to-inside}.
As before, we choose basepoints $o,o'$ of $\Bow(X,\cP),\Bow(X',\cP')$, and denote the metrics on $\Bow(X,\cP), \Bow(X',\cP')$ by $d,d'$ respectively.  

\subsection{Quasi-centre estimates}
We now collect technical lemmas needed for the proof of Theorem~\ref{thm:boundary-to-inside}. 

\begin{lemma}\label{horobdistchar}
Suppose a space $X$ is hyperbolic relative to $\cP$ and is visually complete.  Let $o\in \Bow(X,\cP)$ be a basepoint.
	For any $C'\geq 0$ there exists $C$ with the following property. Let $P\in\cP$, $x\in\Bow(X,\cP)$, and $s$ be a quasi-centre of $a_P,x,o$.
\begin{enumerate}
	\item If $x\in N_{C'}(\cH(P))$, then
 $$d(o,x)-d(o,P)\lesssim_C  2(d(o,s)-d(o,P));$$
 if moreover $x\in N_{C'}(P)$ then we can replace $\lesssim_C$ by $\approx_C$.
 \item If $x\in \Bow(X,\cP)$ satisfies for some $K$ that 
 $$d(o,x)-d(o,P)\lesssim_K 2(d(o,s)-d(o,P)),$$
 then $x\in N_{C+K}(\cH(P))$.
		Moreover, if $d(o,x)-d(o,P)\approx_K 2(d(o,s)-d(o,P))$ then $x\in N_{C+K}(P)$.
\end{enumerate}
\end{lemma}

\begin{proof}
	(1) Any geodesic from $o$ to $x$ is the concatenation of a geodesic of length approximately $d(o,P)$ and a geodesic coarsely contained in $\cH(P)$.  The quasiconvexity of $\cH(P)$ means that $s$ coarsely lies in $\cH(P)$, and considering tree approximations as in Figure~\ref{horobdistchar:fig} the conclusion follows.
\begin{figure}[h]
   \def\svgwidth{0.7\textwidth}
   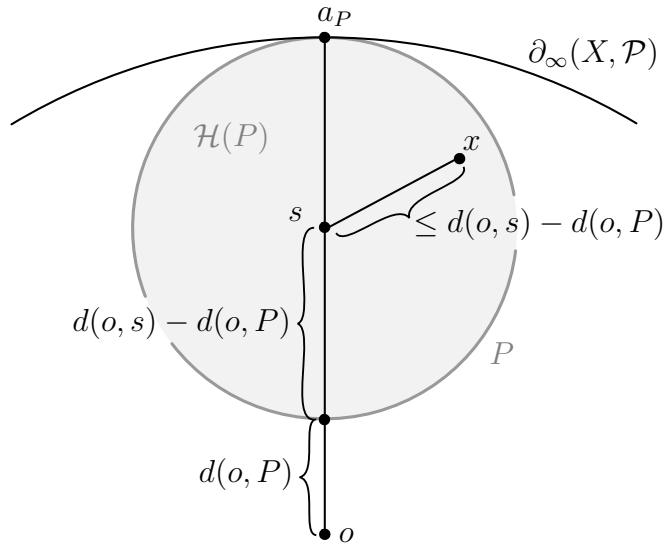
	\caption{Points in horoballs}
   \label{horobdistchar:fig}
\end{figure}

	(2) 
	The Busemann function $\beta_{a_P}(\cdot, o)$ satisfies $\beta_{a_P}(x,o)\approx d(s,x)-d(o,s)\approx d(o,x)-2d(o,s)$. We used that $s$ is at distance coarsely equal to $(a_P|x)_o$ from $o$ and coarsely equal to $(a_P|o)_x$ from $x$ in the first coarse equality, and that it is close to geodesics from $o$ to $x$ in the second one. Using the hypothesis with $\lesssim_K$ (the proof when we have $\approx_K$ being analogous), we get $\beta_{a_P}(x,o)\lesssim_K -d(o,P)$, which implies $x\in N_{C+K}(\cH(P))$, as required.
\end{proof}

\begin{lemma}\label{anypointissplit}
Suppose a space $X$ is hyperbolic relative to $\cP$ and is visually complete.  Let $o\in \Bow(X,\cP)$ be a basepoint.
 For any $C'\geq 0$ there exists $C$ so that for every $x\in \Bow(X,\cP)$, there exist $a,b\in\bdry \Bow(X,\cP)$ with the properties that
 \begin{enumerate}
  \item $\rho(a,b)\asymp_C e^{-\epsilon d(o,x)}$,
  \item $x$ is within distance $C$ from a quasi-centre of $a,b,o$,
\end{enumerate}
and if $x \in N_{C'}(\cH(P))$ for some $P\in\mathcal P$ then one can arrange that
\begin{enumerate}
	\item[(3)] $\rho(a,a_P)^2/r_P \leq \rho(a,b) \leq \rho(a,a_P) \leq r_P$, where $r_P =e^{-\epsilon d(o,P)}$ for $\epsilon$ the visual parameter of $\rho$.
 \end{enumerate}
\end{lemma}

\begin{proof}
 If $x$ lies in no horoball, then as the first and second items follow directly from visual completeness we are done.

 Now suppose $x\in N_{C'}(\cH(P))$ for some $P \in \cP$.
	Since horoballs are coarsely sub-level sets of Busemann functions, we have $\beta_{a_P}(x,o) \lesssim -d(o,P)$.
	We have two cases.
	First, let $s$ be a quasi-centre for $x,a_P,o$ and suppose for a large constant $K_1$ to be determined we have
	\begin{gather}
		\label{eq:splitpt1}
		d(s,o) + 2K_1 \leq d(x,o).
	\end{gather}
	Let $x' \in [o,x]$ be chosen with $d(x',o)=d(x,o)-K_1 \geq d(s,o)+K_1$.
	Apply visual completeness to $x'$ to find $a,b \in \bdry \Bow(X,\cP)$ so that $x'$ is within bounded distance from a quasi-centre for $a,b,o$, so (1,2) are satisfied.
	By \eqref{eq:splitpt1} and Lemma~\ref{horobdistchar}(1) 
	\begin{equation}\label{eq:splitpt2}
		d(o,s)+2K_1-d(o,P) \leq d(o,x)-d(o,P) \lesssim 2(d(o,s)-d(o,P)),
	\end{equation}
	so $-d(o,P)\gtrsim 2K_1-d(o,s)$ and 
	$r_P =e^{-\epsilon d(o,P)} \succeq e^{-\epsilon d(o,s)} e^{\epsilon 2K_1}$.
	By the choice of $x'$ we have $\rho(a,b) \asymp_{M_1} e^{-\epsilon d(x',o)} = e^{-\epsilon d(x,o)} e^{\epsilon K_1}$ for some $M_1$.
	By Lemma~\ref{splitdist} $\rho(a,a_P) \asymp_{M_2} e^{-\epsilon d(o,s)}$ for some $M_2$.
	By \eqref{eq:splitpt2},
	\[
	\frac{\rho(a,a_P)^2}{r_P} \asymp \frac{e^{-\epsilon 2d(o,s)}}{e^{-\epsilon d(o,P)}}
		\preceq e^{-\epsilon d(o,x)}.
	\] 
	So for some $M_3$ we have
	\begin{gather*}
		\frac{\rho(a,a_P)^2}{r_P} \leq M_3 e^{-\epsilon d(o,x)}, 
		\\
		\frac{1}{M_1}e^{-\epsilon d(o,x)}e^{\epsilon K_1} \leq \rho(a,b)
		\leq M_1 e^{-\epsilon d(o,x)} e^{\epsilon K_1} 
		\leq M_1 e^{-\epsilon d(o,s)} e^{-\epsilon K_1},
		\\
		\frac{1}{M_2} e^{-\epsilon d(o,s)} \leq \rho(a,a_P)
		\leq M_2 e^{-\epsilon d(o,s)} \leq M_2 e^{-\epsilon d(o,P)} e^{-2\epsilon K_1}, \text{ and }
		\\
		e^{-\epsilon d(o,P)} \leq r_P.
	\end{gather*}
	Fixing $K_1$ so that $e^{\epsilon K_1} \geq \max\{ M_1 M_3, M_1 M_2, \sqrt{M_2}\}$ we have shown (3) holds.

	Second, if \eqref{eq:splitpt1} fails, then $x$ is within a bounded distance of $[o,a_P)$ (depending on $K_1$, which we now regard as a constant).
	For a large $K_2$ to be determined, let $x'\in[o,a_P)$ be chosen with $d(x',o)=d(x,o)+K_2$.
	Applying visual completeness to $x'$, we find $a',b' \in \bdry \Bow(X,\cP)$ so that $x'$ is within bounded distance from a quasi-centre of $a',b',o$.
	Since $d(x',o) \approx (a'|b')_o \gtrsim \min\{(a'|a_P)_o,(a_P|b')_o\}$, swapping $a'$ and $b'$ if necessary, we may assume $(a'|a_P)_o \lesssim d(x',o)$. 

	Now find $x''\in[o,a')$ with $d(x'',o)=d(x',o)+K_3 = d(x,o)+K_2+K_3$ for some large $K_3$ to be determined.
	Let $s''$ be a quasi-centre for $x'', a_P,o$; by construction it is coarsely $x'$, so $d(s'',o) \approx_{M_4} d(x',o)$ and $d(x'',o) \approx_{M_4} d(x',o)+K_3$ for some $M_4$, so fixing $K_3 \geq 2K_1+2M_4$ we have that $x''$ and $s''$ satisfy \eqref{eq:splitpt1}.
	Moreover, 
	\begin{align*}
		\beta_{a_P}(x'',o) & \approx -d(x',o)+d(x'',x') \approx -d(x,o)-K_2+K_3 
		\\ & \lesssim -d(o,P)-K_2+K_3,
	\end{align*}
	so provided we choose $K_2$ large enough depending on $K_3$ we can ensure that $x'' \in \cH(P)$.
	Thus $x''$, which is coarsely $x$, satisfies the hypotheses of the first case and so we can find the desired points $a,b$.
\end{proof}

\begin{lemma}\label{insidehor}
Suppose a space $X$ is hyperbolic relative to $\cP$ and is visually complete, and fix a basepoint $o\in \Bow(X,\cP)$.
 There exists $C$ with the following property. If the ray from $o$ to $a\in\bdry \Bow(X,\cP)$ contains a point $x\in \cH(P)$, for some $P\in\cP$, then
 $$\rho(a_P,a)\leq C r_P e^{-\epsilon d(x,X)}.$$
\end{lemma}

\begin{proof}
 Using the definition of visual metric, we see that we have to prove $(a|a_P)\gtrsim d(o,P)+d(x,X)$. 
	Let $y,z \in X\cap \cH(P)$ be the first and last points of the ray in $\cH(P)$.
	Let $s$ be a quasi-centre for $o,a,a_P$.
	There exists $m,n$ so that $s$ is bounded distance to both $\gamma_y(m)$ and $\gamma_z(n)$ where $\gamma_y,\gamma_z$ are the geodesic rays to $a_P$ in $\cH(P)$ given by Definition~\ref{def:admissible-horoballs}.
	(By Definition~\ref{def:admissible-horoballs}(4), $m \approx n$.)
	Using tree approximations, we see there exists $C$ so that either $d(x,\gamma_y(m')) \leq C$ for some $m' \leq m$ or $d(x,\gamma_z(n')) \leq C$ for some $n' \leq n$.
	Suppose the former; the latter case is analogous.
	Then by Definition~\ref{def:admissible-horoballs}(4) we have $m'+C \geq d(x,X) \gtrsim m'$ and so
	\[(a|a_P) \approx d(o,s) \approx d(o,y)+m \geq d(o,y) + m' \approx d(o,P)+d(x,X). \qedhere\]
\end{proof}

\subsection{Embedding of cusped spaces}

We return to the proof of Theorem~\ref{thm:boundary-to-inside}.
In the remainder of this section, we denote by $C_i,K_i$ constants that depend on the data only (including $h$). The index of the $K_i$'s will be reset to $1$ at the end of each proof. 

First of all, work of Bonk--Schramm allow us to extend $h$ to the cusped spaces.

\begin{theorem}[Bonk--Schramm]\label{thm:bonk-schramm}
	Suppose $(X,\cP), (X',\cP')$ are two visually complete relatively hyperbolic spaces, and that $h:\bdry (X, \cP)\to \bdry (X', \cP')$ is a quasisymmetric embedding.
	Then there exists a $(C_1,C_1)$-quasi-isometric embedding $\hat f:\Bow(X,\cP)\to \Bow(X',\cP')$ so that $\bdry \hat f=h$ and that $\hat f(o)=o'$.
\end{theorem}
\begin{proof}
	By Lemma~\ref{lem:bowditch-unif-pfct} $\bdry(X,\cP)=\bdry\Bow(X,\cP)$ (and $\bdry(X',\cP')=\bdry\Bow(X',\cP')$) are uniformly perfect, hence by \cite[Corollary 3.12]{TV-80-qs} $h$ is a ``power'' quasisymmetric embedding in the sense of \cite[\S6]{BS-00-gro-hyp-embed}.   
	Therefore there is a quasi-isometric embedding $\hat{f}:\Bow(X,\cP)\to\Bow(X',\cP')$ so that $\bdry \hat{f}=h$ by \cite[Theorems 7.4, 8.2]{BS-00-gro-hyp-embed}.
	We change $\hat{f}$ by defining $\hat{f}(o)=o'$; the resulting quasi-isometry constant $C_1$ depends only on $h$ and the data of the spaces.
\end{proof}

\subsection{Images of horoballs}
Recall that throughout this subsection, $h:\bdry(X,\cP)\to\bdry(X',\cP')$ is the shadow-respecting quasisymmetric embedding of Theorem~\ref{thm:boundary-to-inside}, with $\Bow(X,\cP),\Bow(X',\cP')$ having basepoints $o,o'$ and metrics $d,d'$.  Additionally, now we fix the quasi-isometric embedding extension $\hat{f}:\Bow(X,\cP)\to\Bow(X',\cP')$ of Theorem~\ref{thm:bonk-schramm}.

For $P\in \cP$, we denote by $h_\#(P)$ the element of $\cP'$ so that $h(a_P)=a_{h_\#(P)}$.

\begin{lemma}\label{bilipinhor}
	For each $C$ there exists $C_2$ so that
	for any $P\in \cP$ and $x\in N_C(\cH(P))$, with $\lambda_P=\lambda_{a_P}$ from Definition~\ref{def:shadow-qs}(3), we have
	$$d'(o',\hat f(x))\approx_{C_2} d'(o',{h_\#(P)})+\frac{\epsilon}{\epsilon'}\lambda_P(d(o,x)-d(o,P)).$$
\end{lemma}
This lemma will only be applied with $C$ depending only on the quasiconvexity of horoballs, hence $C_2$ will depend only on the data.
\begin{proof}
 Let $\gamma_P$ be a ray from $o$ to $a_P$. By Lemma \ref{anypointissplit} there exist $a,b\in \bdry \Bow(X,\cP)$ so that $\rho(a,b)\asymp e^{-\epsilon d(o,x)}$, $x$ is $K_1$-close to a quasi-centre of $a,b,o$, and $\rho(a,a_P)^2/r_P \leq \rho(a,b) \leq \rho(a,a_P) \leq r_P$.
 Thus by Definition~\ref{def:shadow-qs}(3), we have
 $$\frac{\rho'(h(a),h(b))}{r_{h_\#(P)}}\asymp_{K_2} \left(\frac{\rho(a,b)}{r_P}\right)^{\lambda_P}.$$
 
 Hence, since quasi-centres are coarsely mapped to quasi-centres (see Lemma~\ref{splittosplit}) and the distance of the quasi-centre of two boundary points and the basepoint from the basepoint controls the visual distance (Lemma~\ref{splitdist}), we get
	\begin{align*}
		d'(o',\hat f(x)) & \approx_{K_3} \frac{-1}{\epsilon'} \log \rho'(h(a),h(b))
		\\ & \approx_{K_4} d'(o',{P'})+\frac{\epsilon\lambda_P}{\epsilon'}\left(d(o,x)-d(o,P)\right),
	\end{align*}
 as required.
\end{proof}

\begin{corollary}\label{imageshorob}
 There exists $C_3$ so that for each $P\in \cP$ we have
 \begin{enumerate}
  \item $\hat f(\cH(P)) \subseteq N_{C_3}(\cH(h_\#(P)))$,
  \item $\hat f(P)\subseteq N_{C_3}(h_\#(P))$,
  \item $\hat f(\cH(P)^c) \subseteq N_{C_3}(\cH(h_\#(P)^c)$.
 \end{enumerate}
\end{corollary}

\begin{proof}
	If $x\in \cH(P)$, then by Lemma \ref{horobdistchar} with $C'=0$, we have $d(o,x)-d(o,P)\lesssim_{K_1} 2(d(o,s)-d(o,P))$, where $s$ is a quasi-centre of $a_P, x, o$. By Lemma \ref{bilipinhor} with $C=0$, we have
	$$d'(o',\hat f(x))\approx_{K_2} d'(o',{h_\#(P)})+\frac{\epsilon\lambda_P}{\epsilon'} (d(o,x)-d(o,P)).$$
	In view of Lemma \ref{splittosplit}, if $s'$ is a quasi-centre of $\hat f(x), a_{h_\#(P)}, o'$ then $s'$ coarsely coincides with $\hat f (s)$, so that again by Lemma \ref{bilipinhor} (with $C$ chosen uniformly so that $s\in N_C(\cH(P))$) we have
	$$d'(o',s')\approx_{K_3} d'(o',{h_\#(P)})+\frac{\epsilon\lambda_P}{\epsilon'} (d(o,s)-d(o,P)).$$
 So,
 \begin{align*}
	 d'(o',\hat f(x))-d'(o',{h_\#(P)})
	 & \approx_{K_2} \frac{\epsilon\lambda_P}{\epsilon'} (d(o,x)-d(o,P))
	 \\ & \lesssim_{K_4} \frac{2\epsilon\lambda_P}{\epsilon'} (d(o,s)-d(o,P))
	 \\ & \approx_{2K_3} 2( d'(o',s')-d'(o',{h_\#(P)})).
 \end{align*}
	By Lemma \ref{horobdistchar}(2) with $K=K_2+K_4+2K_3$, we have that $\hat f(x)$ lies within a distance $C_3$ of $\cH(h_\#(P))$.
 
 The second part can be proven in the same way replacing $\lesssim$ by $\approx$.
 
 Finally, if the third part did not hold (for a much larger $C_3$ than in (2)), then there would be $x \in \cH(P)^c$ whose image is ``well-inside'' $\cH(h_\#(P))$.  A geodesic ray from $x$ towards $a_P$ has image a quasi-geodesic ray with controlled constants, that starts ``well-inside'' $\cH(h_\#(P))$, then enters $N_{C_3}(h_\#(P))$ (by (2)) and then goes to $a_{h_\#(P)}$. Such a quasi-geodesic ray does not exist.
\end{proof}

In a similar spirit to parts (1) and (2) of the previous corollary, one also has:

\begin{corollary}\label{roughsimonO}
	For each $P\in \cP$, $\hat f|_{\cH(P)}$ is a $C_4$-rough $({\epsilon\lambda_P}/{\epsilon'})$-similarity.
\end{corollary}

\begin{proof}
	For $x,y\in \cH(P)$, the quasiconvexity of horoballs implies that a quasi-centre $s$ of $x,y,o$ is also contained in (a controlled neighbourhood of) $\cH(P)$. 
By hyperbolicity,
$$d(x,y)\approx_{K_1} d(o,x)+d(o,y)-2d(o,s),$$
and the same holds for pairs of points in $\Bow(X',\cP')$.

	If $s'$ is a quasi-centre of $\hat f(x),\hat f(y), o'$, then bearing in mind Lemma \ref{splittosplit}, we get from Lemma \ref{bilipinhor} (with $C$ chosen so that $s,s'$ are in the $C$-neighbourhoods of $\cH(P), \cH(P')$):
\begin{align*}
	d'(\hat f(x),\hat f(y))
	& \approx_{K_1} \Big(d'(o',\hat f(x)) -d'(o',{h_\#(P)})\Big)+\Big(d'(o',\hat f(y)) 
	\\ & \qquad 
	-d'(o',{h_\#(P)})\Big)
	 -2\Big(d'(o',s') -d'(o',{h_\#(P)})\Big)
	\\ & \approx_{K_2} \frac{\epsilon\lambda_P}{\epsilon'}(d(o,x)-d(o,P))+\frac{\epsilon\lambda_P}{\epsilon'}(d(o,y)-d(o,P)) 
	\\ & \qquad -2\frac{\epsilon\lambda_P}{\epsilon'}(d(o,s)-d(o,P))
	\\ & \approx_{{\epsilon\lambda_P K_1}/{\epsilon'}} \frac{\epsilon\lambda_P}{\epsilon'} d(x,y)
\end{align*}
as required.
\end{proof}

Finally, Corollary \ref{imageshorob}(3) can be strengthened to:

\begin{corollary}\label{GtoG'}
  If $P'\in\cP'$ is so that $a_{P'}\notin h(\bigcup_{\cP} a_{P})$, then $f(X)\cap \cH(P')\subseteq N_{C_5}(X')$. Moreover, $f(X)\subseteq N_{C_5}(X')$.
\end{corollary}

\begin{proof}
 Consider any $x\in \Bow(X,\cP)$. Then $x$ lies within uniformly bounded distance from a geodesic ray from $o$ to, say, $a\in\bdry \Bow(X,\cP)$. Suppose that $\hat f (x)$ lies inside the horoball $\cH(P')$ in $\Bow(X',\cP')$ and at distance $R$ from $X'\subseteq \Bow(X',\cP')$.

 By Lemma \ref{insidehor}, we have $\rho(a_{P'},h(a))\leq K_1r_{P'}e^{-\epsilon R}$. For $R$ large enough, this contradicts Definition~\ref{def:shadow-qs}(2) which gives $\rho(a_{P'},h(Z)) \geq r_{P'}/C$, proving the first part of the corollary.
 
Suppose now $x\in X$ and let us show that its image lies in a controlled neighbourhood of $X'$. If we had that $\hat f (x)$ lies well inside a horoball, say $\cH(P')$, then the argument above would provide $P\in\cP$ so that $h(a_P)=a_{P'}$. But then Corollary \ref{imageshorob}(3) implies that $x$ cannot lie in $\cH(P)^c$, contradicting $x\in X$. This shows the ``moreover'' part.
\end{proof}

 {\bf Announcement:} We perturb $\hat{f}$ up to bounded distance and assume from now on that it maps $X$ into $X'$.

\subsection{Projection terms of the distance formula}

Our goal is now to show that $\hat f|_X:X\to X'$ has the required metric properties. In order to do so, we will need a version of the distance formula for relatively hyperbolic spaces, as described below. Since we will have to be careful about distinguishing between distances in a relatively hyperbolic space and in the corresponding cusped space, we start using the notation $d_X,d_{\Bow(X,\cP)}$ and similar to emphasise the difference.

\smallskip
We require some notation to describe the distance formula.
Suppose a space $Y$ is hyperbolic relative to $\cQ$, with a fixed cusped space $\Bow(Y,\cQ)$. For $Q\in\cQ$ and $x,y\in Y$, we set $\theta_Q(x,y)$ to be the following value. Choose any geodesic $[x,y]$ in $\Bow(Y,\cQ)$ from $x$ to $y$. If $[x,y]$ does not intersect $\cH(Q)$, then set $\theta_Q(x,y)=0$. If it does, then consider the entrance and exit points $x',y'\in Q$, and let $\theta_Q(x,y)=d_Y(x,y)$. We emphasize that the distance is taken in $Y$, not in $\Bow(Y,\cQ)$.

Let $\big\{\big\{A\big\}\big\}_L$ denote $A$ if $A>L$, and $0$ otherwise. 
\begin{theorem}[{Distance formula, \cite[Theorem 0.1]{Si-projrelhyp}}]\label{distform}
In the notation above, there exists $L_0$ so that for each $L\geq L_0$ there exist $\lambda,\mu$ so that the following holds. If $x,y\in Y$ then
\begin{equation}\label{eqn-distform}
	d_Y(x,y)\qi_{\lambda,\mu} \sum_{Q\in\cQ} \big\{\big\{\theta_Q(x,y)\big\}\big\}_L+d_{\Bow(Y,\cQ)}(x,y).
\end{equation}
\end{theorem}

\smallskip
Returning to our study of $\hat f|_X:X\to X'$,
we will make a term-by-term comparison of the distance formulas in $X$ and $X'$. The following lemma will be used to compare the $\theta$ terms.

\begin{lemma}\label{termscomp}
 Under the assumptions of Theorem~\ref{thm:boundary-to-inside} with extension $\hat{f}|_X:X\to X'$ as above, there exists $C_6$ with the following property. Let $P\in\cP$. Then for each $x,y\in X$ we have
 $$\theta_{h_\#(P)}(\hat{f}(x),\hat{f}(y))\qi_{C_6,C_6} \theta_P(x,y)^{{\epsilon\lambda_P}/{\epsilon'}}.$$
 Moreover, if $P'\in\cP'$ is so that $a_{P'}\notin h(\bigcup_{\cP} a_{P})$, then for each $x,y\in X$ we have $\theta_{P'}(\hat{f}(x),\hat{f}(y))\leq C_6$.
\end{lemma}

\begin{proof}
Consider $x,y\in X$ and suppose that a geodesic $[x,y]$ goes deep into a horoball $\cH(P)$ for $P\in\cP$. Then, Corollary \ref{imageshorob} (together with the fact that $\hat f([x,y])$ is a quasi-geodesic and hence stays close to a geodesic with the same endpoints) implies that the entrance and exit points of $[x,y]$ get mapped close to the entrance and exit point of $[\hat{f}(x),\hat{f}(y)]$ in $h_\#(P)$. Moreover, if $[x,y]$ does not go deep into $\cH(P)$ then $[\hat{f}(x),\hat{f}(y)]$ does not go deep into $h_\#(P)$ either.

	Hence, the conclusion follows using that $\hat{f}$ is a rough similarity on $\cH(P)$ (Corollary \ref{roughsimonO}) and Definition~\ref{def:admissible-horoballs}(4).
 
 To show the ``moreover'' part, notice that the geodesic from $\hat{f}(x)$ to $\hat{f}(y)$ cannot go deep into $\cH(P')$ because of Corollary \ref{GtoG'}.
\end{proof}
\begin{corollary}
	\label{cor:snowflake-at-most-one}
	Under the assumptions of Lemma~\ref{termscomp}, for any $P\in \cP$ we have $\epsilon\lambda_P/\epsilon' \leq 1$.
\end{corollary}
\begin{proof}
	Each $P \in \cP$ is uniformly coarsely connected, and is unbounded since $(X,\cP)$ is visually complete: were $P\in \cP$ bounded $a_P$ would be an isolated point.
	Take a quasi-geodesic ray $\gamma:[0,\infty)\to P$.  Then using the bounds of Lemma~\ref{termscomp}, taking large $t > 2C_6$ and large $N$, we have
	\begin{align*}
		d_{X}(\gamma(0),\gamma(Nt))^{\epsilon\lambda_P/\epsilon'}
		& \asymp d_{X'}(\hat{f}(\gamma(0)),\hat{f}(\gamma(Nt)))
		\\ & \leq \sum_{i=1}^{N} d_{X'}(\hat{f}(\gamma((i-1)t)),\hat{f}(\gamma(it)))
		\\ & \asymp \sum_{i=1}^{N} d_{X}(\gamma((i-1)t),\gamma(it))^{\epsilon\lambda_P/\epsilon'}.
	\end{align*}
	The left-hand side is $\asymp (Nt)^{\epsilon\lambda_P/\epsilon'}$, and the right-hand side is $\asymp N t^{\epsilon\lambda_P/\epsilon'}$, so fixing $t$ and letting $N\to\infty$ we have $\epsilon\lambda_P/\epsilon' \leq 1$.
\end{proof}

\begin{proof}[Proof of Theorem \ref{thm:boundary-to-inside}]
	It is easily seen that $\hat f|_X$ is coarsely Lipschitz. 
	By Proposition~\ref{prop:bounded-snowflake-param} we have $\lambda_P \geq \alpha >0$ for all $P \in \cP$, and some $\alpha \in (0,1]$.
	In the following we use that if $r_i \geq 1$ and $\lambda_i \geq \alpha$ for $i = 1,2,\ldots$ then 
	\begin{equation}
		\label{eq:boundary-to-inside1}
		\sum_i r_i^{\lambda_i} \geq \sum_i r_i^\alpha
		\geq \left(\sum_i r_i\right)^\alpha.
	\end{equation}
	Using the distance formula Theorem~\ref{distform}, Lemma~\ref{termscomp}, and choosing $T$ large enough we get

\begin{align*}
	d_{X'}(\hat{f}(x),\hat{f}(y))
	& \qi_{K_1,K_1} \sum_{P'\in\cP'} \{\{\theta_{P'}(\hat{f}(x),\hat{f}(y))\}\}_T +d_{\Bow(X',\cP')}(\hat{f}(x),\hat{f}(y)) 
	\\ & = \sum_{P\in\cP} \{\{ \theta_{h_\#(P)}(\hat{f}(x),\hat{f}(y))\}\}_T +d_{\Bow(X',\cP')}(\hat{f}(x),\hat{f}(y)) 
	\\ & \grqi_{K_2,K_2} \sum_{P\in\cP} \{\{\theta_{P}(x,y)^{\epsilon\lambda_P/\epsilon'}\}\}_{T'}+d_{\Bow(X,\cP)}(x,y)
	\\ & \geq \left(\sum_{P\in\cP} \{\{\theta_{P}(x,y)\}\}_{T'}+d_{\Bow(X,\cP)}(x,y)\right)^{(\epsilon/\epsilon')\inf_{P\in\cP}\lambda_P}
	\\ & \qi_{K_2,K_2} d_X(x,y)^{(\epsilon/\epsilon')\inf_{P\in\cP}\lambda_P},
\end{align*}
	where the fourth inequality follows from \eqref{eq:boundary-to-inside1}.
	Notice that we changed the threshold when passing from $X$ to $X'$.
	Thus $\hat{h}=\hat{f}|_X$ is polynomially distorting, and coarsely respects peripherals by Corollary~\ref{imageshorob}.
	That $\hat{h}$ is a snowflake on peripherals follows from Lemma~\ref{termscomp} and Corollary~\ref{cor:snowflake-at-most-one}. 
	
	We have $h=\bdry (\hat{h})_\Bow$: by construction $h=\bdry \hat{f}$ and $\hat{f}=\hat{h}$ on $X$.  For any geodesic ray $[o,a) \subset \Bow(X,\cP)$, either $a=a_P$ for some $P \in \cP$, in which case $\bdry (\hat{h})_\Bow(a)=h(a)$ by the construction of $(\hat{h})_\Bow$, or $a$ is a limit of points in $X$, and so again $\bdry (\hat{h})_\Bow(a) = \bdry \hat{f}(a) = h(a)$.
	
	Moreover, if $h$ asymptotically $\frac{\epsilon'}{\epsilon}$-snowflakes then for all $P\in \cP$, $\lambda_P =\epsilon'/\epsilon$, so the bound above shows $f$ has at most linear distortion, i.e., $f$ is a quasi-isometric embedding.

	Finally, we verify the uniqueness of $\hat{h}$ and its functorial properties.
	Since $(X,\cP)$ is visually complete, any $x \in \Bow(X,\cP)$ is bounded distance to a quasi-centre for some $a,b \in \bdry(X,\cP)$ and $o$.  By Lemma~\ref{splittosplit} any quasi-isometry $\Bow(X,\cP)\to \Bow(X',\cP')$ which sends $a,b$ to $h(a),h(b)$ must coarsely send a quasi-centre for $a,b,o$ to a quasi-centre for $h(a),h(b),o'$; in other words, $\hat{f}(x)$ is uniquely defined up to bounded ($\Bow(X',\cP')$) distance.  Since $X' \subset \Bow(X',\cP')$ is uniformly embedded, this means that $\hat{h} = \hat{f}|_{X}$ is uniquely defined up to bounded distance too.

	Likewise, if $j:\bdry(X',\cP')\to\bdry(X'',\cP'')$ is another shadow-respecting quasisymmetric embedding, then by this same uniqueness argument $\widehat{j \circ h}$ and $\hat{j}\circ \hat{h}$ are at bounded distance from each other.
	Finally, for $\id:\bdry(X,\cP)\to\bdry(X,\cP)$, by uniqueness again, $\hat{\id}$ and $\id_X$ are at bounded distance from each other.
\end{proof}

\bibliographystyle{alpha}
\bibliography{biblio}

\end{document}